\documentclass[a4paper, 10pt]{amsart}
\usepackage[top=80pt,bottom=65pt,left=70pt,right=70pt]{geometry}
\usepackage{graphicx, eepic}
\usepackage{times}

\normalfont
\usepackage{bbm}
\usepackage{xcolor}
\usepackage{verbatim}
\usepackage{kotex}
\usepackage{mathrsfs}
\usepackage{amscd}
\usepackage{hhline}
\usepackage{float}
\usepackage[all]{xy}
\usepackage[subnum]{cases}
\ifpdf
\usepackage{tikz}
\fi
\usetikzlibrary{arrows,matrix,positioning}

\usepackage{makecell}

\usepackage{amsthm,amsmath,amsfonts,amssymb}
\usepackage{hyperref}
\hypersetup{colorlinks}
\hypersetup{
    bookmarks=true,         
    unicode=false,          
    pdftoolbar=true,        
    pdfmenubar=true,        
    pdffitwindow=false,     
    pdfstartview={FitH},    
    pdftitle={My title},    
    pdfauthor={Author},     
    pdfsubject={Subject},   
    pdfcreator={Creator},   
    pdfproducer={Producer}, 
    pdfkeywords={keyword1} {key2} {key3}, 
    pdfnewwindow=true,      
    colorlinks=true,       
    linkcolor=blue,          
    citecolor=red,        
    filecolor=violet,      
    urlcolor=violet       
}
\usepackage{multirow, url}
\usepackage{color}
\usepackage[all]{xy}
\theoremstyle{plain}
\newtheorem{theorem}{Theorem}[section]

\newtheorem{proposition}[theorem]{Proposition}
\newtheorem{lemma}[theorem]{Lemma}

\newtheorem{corollary}[theorem]{Corollary}

\theoremstyle{definition}

\newtheorem{conjecture}[theorem]{Conjecture}
\newtheorem{definition}[theorem]{Definition}

\newtheorem{example}[theorem]{Example}

\theoremstyle{remark}
\newtheorem{remark}[theorem]{Remark}

\newtheorem{notation}[theorem]{Notation}

\renewcommand{\bar}{\overline}

\newcommand{\C}{\mathbb{C}}

\newcommand{\Q}{\mathbb{Q}}
\newcommand{\R}{\mathbb{R}}

\newcommand{\Z}{\mathbb{Z}}
\newcommand{\p}{\mathbb{C}P}

\newcommand{\pp}{\mathbb{P}}

\newcommand{\mcal}{\mathcal}
\newcommand{\pd}{\mathrm{PD}}

\def\mcal{\mathcal}
\def\frak{\mathfrak}

\newcommand{\ds}{\displaystyle}
\newcommand{\vs}{\vspace}
\newcommand{\hs}{\hspace}

\numberwithin{equation}{section} \numberwithin{table}{section}

\begin{document}                                                                          

\title{Classification of six dimensional monotone symplectic manifolds admitting semifree circle actions I}
\author{Yunhyung Cho}
\address{Department of Mathematics Education, Sungkyunkwan University, Seoul, Republic of Korea. }
\email{yunhyung@skku.edu}

\begin{abstract}
	Let $(M,\omega_M)$ be a six dimensional closed monotone symplectic manifold admitting an effective semifree Hamiltonian $S^1$-action.
	We show that if the minimal (or maximal) fixed component of the action is an isolated point, then $(M,\omega_M)$ is $S^1$-equivariant symplectomorphic 
	to some K\"{a}hler Fano manifold $(X,\omega_X, J)$ with a certain holomorphic $\C^*$-action.  
	We also give a complete list of all such Fano manifolds and describe all semifree $\C^*$-actions on them specifically. 
\end{abstract}
\maketitle
\setcounter{tocdepth}{1} 
\tableofcontents

\section{Introduction}
\label{secIntroduction}

According to Koll\'{a}r-Miyaoka-Mori \cite{KMM}, there are only finitely many deformation types of smooth Fano varieties for each dimension. 
For example, there is only one $1$-dimensional smooth Fano variety $\C P^1.$ 
In dimension two, there are 10 types of smooth Fano surfaces, called \textit{del Pezzo surfaces}, classified as
$\C P^2$, $\C P^1 \times \C P^1$, and the blow-up of $\C P^2$ at $k$ generic points for $1 \leq k \leq 8.$ 
For the 3-dimensional case, Iskovskih \cite{I1} \cite{I2} classified all smooth Fano 3-folds having Picard number one. Later, Mori and Mukai \cite{MM} completed the classification of smooth Fano 3-folds.
(There are 105 types of smooth Fano 3-folds overall.) Note that any smooth Fano variety $X$ admits a K\"{a}hler form $\omega_X$ such that $[\omega_X] = c_1(TX)$, which is known to a consequence of Yau's proof of Calabi's conjecture.

A \textit{monotone symplectic manifold $(M,\omega)$} is a symplectic analogue of a smooth Fano variety in the sense that it satisfies 
$\langle c_1(M), [\Sigma] \rangle > 0$ for any symplectic surface 
$\Sigma \subset M$. (See Section \ref{secMonotoneSemifreeHamiltonianS1Manifolds} for the precise definition.)
Then it is obvious that the category of monotone symplectic manifolds contains all smooth Fano varieties. A natural question that arises is whether a given monotone symplectic manifold is
K\"{a}hler (and hence Fano) with respect to some integrable almost complex structure compatible with the given symplectic form. It turned out that the answer for the question is negative in general, 
where a counter-example was found in dimension twelve by Fine and Panov \cite{FP}. 

In the low dimensional case, where $\dim M = 2$ or $4$, the answer is positive. Ohta and Ono \cite[Theorem 1.3]{OO2} proved that if $\dim M = 4$, then $M$ is diffeomorphic to a del Pezzo surface. 
Thus, from the uniqueness of a symplectic structure on a rational surface (due to McDuff \cite{McD3}), it follows that  every closed monotone symplectic four manifold is K\"{a}hler. 
As far as the author knows, the existence of a closed monotone symplectic manifold which is not K\"{a}hler is not known for dimension 6, 8, 10. 

The aim of this paper is to study six-dimensional monotone symplectic manifolds admitting Hamiltonian circle actions. More specifically, we deal with the following conjecture. 

\begin{conjecture}\cite[Conjecture 1.1]{LinP}\cite[Conjecture 1.4]{FP2}\label{conjecture_main}
	Let $(M,\omega)$ be a six dimensional closed monotone symplectic manifold equipped with an effective Hamiltonian circle action.  Then $(M,\omega)$ is $S^1$-equivariantly symplectomorphic
	to some K\"{a}hler manifold $(X,\omega_X, J)$ with some holomorphic Hamiltonian $S^1$-action.
\end{conjecture}

Note that Conjecture \ref{conjecture_main} is known to be true when $b_2(M) = 1$ by McDuff \cite{McD2} and Tolman \cite{Tol}. 
(There are four types of such manifolds up to $S^1$-equivariant symplectomorphism.)
Recently, Lindsay and Panov \cite{LinP} provided some evidences that Conjecture \ref{conjecture_main} is possibly true.  For instance, they proved that $M$ given in Conjecture
\ref{conjecture_main} is simply-connected as other smooth Fano varieties are. 

The author is preparing a series of papers originated in an attempt to give an answer to Conjecture \ref{conjecture_main} 
under the assumptions that the action is {\em semifree}\footnote{An $S^1$-action is called {\em semifree} if 
the action is free outside the fixed point set.}. In this article, we prove the following.

\begin{theorem}\label{theorem_main}
		Let $(M,\omega)$ be a six-dimensional closed monotone symplectic manifold equipped with a semifree Hamiltonian 
		circle action. Suppose that the maximal or the minimal fixed component of the action is an isolated point. 
		Then $(M,\omega)$ is $S^1$-equivariantly symplectomorphic to some 
		K\"{a}hler Fano manifold with some holomorphic Hamiltonian circle action. 
\end{theorem}

The proof of Theorem \ref{theorem_main} is essentially based on Gonzalez's approach \cite{G}. He introduced a notion so-called a {\em fixed point data} for a semifree Hamiltonian circle action,  
which is a collection of {\em a symplectic reductions\footnote{A reduced space at a critical level is not a smooth manifold nor an orbifold in general. However, if $\dim M = 6$ and the action is semifree, then a symplectic reduction at any (critical) level is a smooth manifold with the induced symplectic form. See 
Proposition \ref{proposition_topology_reduced_space}.} at critical levels} together with an information of 
critical submanifolds (or equivalently fixed components) as embedded symplectic submanifolds of reduced spaces. (See Definition \ref{definition_fixed_point_data} or \cite[Definition 1.2]{G}.)
He then proved that a fixed point data determines a 
semifree Hamiltonian $S^1$-manifold up to $S^1$-equivariant symplectomorphism under the assumption that every reduced space is {\em symplectically rigid}\footnote{See Section \ref{secFixedPointData} for the definition.}. 

\begin{theorem}\cite[Theorem 1.5]{G}\label{theorem_Gonzalez}
		Let $(M,\omega)$ be a six-dimensional closed semifree Hamiltonian $S^1$-manifold. 
		Suppose that every reduced space is symplectically rigid.
		Then $(M,\omega)$ is determined by its fixed point data up to $S^1$-equivariant symplectomorphism.
\end{theorem}

The proof of Theorem \ref{theorem_main} goes as follows : if $(M,\omega)$ is a closed six-dimensional monotone semifree Hamiltonian $S^1$-manifold with an isolated fixed point 
as an extremal fixed point, then we show that
\begin{itemize}
	\item (first step :) every reduced space of $(M,\omega)$ is symplectically rigid, and 
	\item (second step :) the fixed point data of $(M,\omega)$ coincides with some smooth Fano variety equipped with some holomorphic semifree Hamiltonian $S^1$-action. 
\end{itemize}
The main difficulty in the second step is it is almost hopeless to determine whether two given fixed point data coincide or not in general. 
To overcome the difficulty, we first classify all possible {\em topological fixed point data}\footnote{See Definition \ref{definition_topological_fixed_point_data}.}, or TFD shortly, of $(M,\omega)$. 
A {\em topological fixed point data} of $(M,\omega)$ is a topological version of a fixed point data 
in the sense that it records ``homology classes'', 
not embeddings themselves, of fixed components in reduced spaces. With the aid of the Duistermaat-Heckman theorem (Theorem \ref{theorem_DH}), 
the localization theorem (Theorem \ref{theorem_localization}), and some theorems about symplectic four manifolds (cf. \cite{LL}, \cite{Li}), we classify all possible 
TFD as in Table \ref{table_list}. 

An immediate consequence of the classification (Table \ref{table_list}) of TFD is that every reduced space of $(M,\omega)$ is either $\p^2$, $\p^1 \times \p^1$, or $\p^2 \# ~k\overline{\p^2}$
for $k \leq 3$, where these manifolds are known to be symplectically rigid. (See Theorem \ref{theorem_uniqueness} and Theorem \ref{theorem_symplectomorphism_group}.)
Moreover, each topological fixed point data determines the first Chern number $\langle c_1(TM), [M] \rangle$ as well as the Betti numbers of $M$. This enables us to expect a candidate for $(M,\omega)$
in the list of smooth Fano 3-folds given by Mori-Mukai \cite{MM}. Indeed, we could succeed in finding holomorphic Hamiltonian $S^1$-actions on those Fano candidates whose TFD match up with 
ours in Table \ref{table_list}. (See the examples given in Section \ref{secCaseIDimZMax}, \ref{secCaseIIDimZMax2}, \ref{secCaseIIIDimZMax4}.) 

And then we will show that each TFD in Table \ref{table_list} determines a fixed point data uniquely. The following two facts, due to Siebert-Tian \cite{ST} and Zhang \cite{Z}, 
are essentially used in this process. 
\begin{itemize}
	\item Any possible fixed point data whose topological type is given in Table \ref{table_list} is algebraic, i.e., any fixed component as an embedded symplectic submanifold in a reduced space is 
	symplectically isotopic to an algebraic curve. (See Theorem \ref{theorem_ST} and Theorem \ref{theorem_Z}.)
	\item Any two algebraic curves in a reduced space are symplectically isotopic to each other. (See Lemma \ref{lemma_isotopic}.)
\end{itemize}

This paper is organized as follows. In Section \ref{secHamiltonianCircleActions}, 
we give a brief introduction to Hamiltonian $S^1$-actions, including the Duistermaat-Heckman theorem that we will use quite often.
An equivariant cohomology theory for Hamiltonian $S^1$-actions, especially about the Atiyah-Bott-Berline-Vergne localization theorem and equivariant Chern classes, 
is explained in Section \ref{secEquivariantCohomology}. In Section \ref{secMonotoneSemifreeHamiltonianS1Manifolds},
we restrict our attention to a closed monotone semifree Hamiltonian $S^1$-manifold and explain
how the topology of a reduced space and a reduced symplectic form change when crossing critical values of a moment map. We also explain how a reduced space 
inherits a monotone reduced symplectic form from $\omega$. In Section \ref{secFixedPointData}, we give a definition of (topological) fixed point data and introduce the Gonzalez's Theorem 
\cite[Theorem 1.5]{G}. From Section \ref{secCaseIDimZMax} to \ref{secCaseIIIDimZMax4}, we classify all topological 
fixed point data and describe the corresponding Fano candidates with specific holomorphic circle actions. In Section \ref{secMainTheorem}, we prove Theorem \ref{theorem_main}.

In appendix, we add two sections. Section \ref{secMonotoneSymplecticFourManifoldsWithSemifreeS1Actions}
is about a classification of closed monotone semifree Hamiltonian four manifolds. We apply our arguments used in this paper
to four dimensional cases and obtain a complete list of such manifolds. See Table \ref{table_list_4dim}. Finally in Section \ref{secSymplecticCapacitiesOfSmoothFano3Folds},
as a by-product of our classification, we calculate the Gromov width and the Hofer-Zehnder capacity for each manifold in Table \ref{table_list} by applying theorem of Hwang-Suh \cite{HS}.

\subsection*{Acknowledgements} 
The author would like to thank Dmitri Panov for bringing the paper \cite{Z} to my attention.
The author would also like to thank Jinhyung Park for helpful comments. 
This work is supported by the National Research Foundation of Korea(NRF) grant funded by the Korea government(MSIP; Ministry of Science, ICT \& Future Planning) (NRF-2017R1C1B5018168).

\section{Hamiltonian circle actions}
\label{secHamiltonianCircleActions}
    
    In this section, we briefly review some facts about Hamiltonian circle actions.
    Throughout this section, we assume that $(M,\omega)$ is a $2n$-dimensional closed symplectic manifold and $S^1$ is the unit circle group in $\C$ 
    acting on $M$ smoothly with the fixed point set $M^{S^1}$. 

\subsection{Hamiltonian actions}
\label{ssecHamiltonianActions} 

	Let $\frak{t}$ be the Lie algebra of $S^1$ and let $X \in \frak{t}$. Then a vector field $\underbar{X}$ on $M$ defined by
	\[
		\underbar{X}_p = \left.\frac{d}{dt}\right|_{t=0} \exp(tX) \cdot p 
	\] is called a \textit{fundamental vector field with respect to} $X$.
	We say that the $S^1$-action on $(M,\omega)$ is \textit{symplectic} if it preserves the symplectic form $\omega$, i.e.
	\[
		\mcal{L}_{\underbar{X}} \omega = 0
	\] for any $X \in \frak{t}$. By Cartan's magic formula, we have
	\[
		\mcal{L}_{\underbar{X}} \omega = d \circ i_{\underbar{X}} \omega + i_{\underbar{X}} \circ d \omega = d \circ i_{\underbar{X}} \omega.
	\]		
	So, the action is symplectic if and only if $i_{\underbar{X}} \omega$ is a closed 1-form on $M$. If $i_{\underbar{X}} \omega$ is exact, then we say that the action is \textit{Hamiltonian}.
	In particular, any symplectic circle action is locally Hamiltonian by the classical Poincar\'{e} lemma.
	
	When the $S^1$-action is Hamiltonian, there exists a smooth function $H : M \rightarrow \R$, called a {\em moment map}, such that 
	\[
		i_{\underbar{X}} \omega = -dH.
	\]
	This equation immediately implies that $p \in M$ is a fixed point of the action if and only if $p$ is a critical point of $H$.  	
	The following theorem describes a local behavior of the action near each fixed point. 
	
	\begin{theorem}\label{theorem_equivariant_darboux}(Equivariant Darboux theorem)
		Let $(M,\omega)$ be a $2n$-dimensional symplectic manifold equipped with a Hamiltonian circle action and $H : M \rightarrow \R$ be a moment map. 
		For each fixed point $p \in M^{S^1}$, there is an $S^1$-invariant complex coordinate chart $(\mcal{U}_p, z_1, \cdots, z_n)$ with weights 
		$(\lambda_1, \cdots, \lambda_n) \in \Z^n$ such that
		\begin{enumerate}
			\item $\omega|_{\mcal{U}_p} = \frac{1}{2i} \sum_i dz_i \wedge d\bar{z_i},$ and
                                \item for any $t \in S^1$, the action can be expressed by $$t \cdot (z_1, \cdots, z_n) = (t^{\lambda_1}z_1, \cdots, t^{\lambda_n}z_n)$$ 
                                so that the moment map can be written as 
                                \[
                                	H(z_1, \cdots, z_n) = H(p) + \frac{1}{2}\sum_i \lambda_i |z_i|^2.
			\]
		\end{enumerate}
	\end{theorem}
	
	Using Theorem \ref{theorem_equivariant_darboux}, we obtain the following.
	
	\begin{corollary}\label{corollary_properties_moment_map}\cite[Chapter 4]{Au}
		Let $(M,\omega)$ be a closed symplectic manifold equipped with a Hamiltonian circle action with a moment map $H : M \rightarrow \R$. Then $H$
                     satisfies the followings.
		\begin{enumerate}
			\item $H$ is a Morse-Bott function.
			\item Let $p \in M^{S^1}$ be a fixed point of the action and let $(\mcal{U}_p, z_1, \cdots, z_n)$ be an equivariant Darboux chart near $p$ with weights 
			$(\lambda_1, \cdots, \lambda_n) \in \Z^n$. Then we have $$ \mathrm{ind}(p) = 2n_p $$ where $\mathrm{ind}(p)$ is the Morse-Bott index of $p$ and $n_p$ is the number of
			 negative weights of tangential representation at $p$. Also, the twice number of zeros in $\{\lambda_1, \cdots, \lambda_n \}$ is a real dimension of the fixed component
			  containing $p$.
			\item Any fixed component is a symplectic submanifold.
			\item Each level set of $H$ is connected. In particular, an extremal fixed component is connected.
		\end{enumerate}
	\end{corollary}

\subsection{Symplectic reduction} 
\label{ssecSymplecticReduction}

	Let $r \in \R$ be a regular value of $H$. Then the level set $H^{-1}(r)$ does not have any fixed point so that $H^{-1}(r)$ is a fixed point-free $S^1$-manifold of dimension $2n-1$. 
	The quotient space $M_r = H^{-1}(r) / S^1$ is an orbifold of dimension $2n-2$ with cyclic quotient singularities. Since the restriction of $\omega$ on $H^{-1}(r)$ satisfies
	\begin{itemize}
		\item $i_{\underbar{X}} \omega = -dH = 0$ on $H^{-1}(r)$, and
		\item $\mcal{L}_{\underbar{X}} \omega = i_{\underbar{X}} \circ d\omega + d \circ i_{\underbar{X}} \omega = 0$.
	\end{itemize} 
	Thus we can push-forward $\omega$ to $M_r$ via the quotient map $$\pi_r : H^{-1}(r) \rightarrow M_r$$ and so that we obtain a symplectic form $\omega_r$ on $M_r$.
	We call $(M_r, \omega_r)$ a \textit{symplectic reduction at $r$}.

\subsection{Duistermaat-Heckman theorem} 
\label{ssecDuistermaatHeckmanTheorem}

	Let $J$ be an $S^1$-invariant $\omega$-compatible almost complex structure on $M$ so that $g_J(\cdot, \cdot) := \omega(J \cdot, \cdot)$ is a Riemannian metric on $M$. Since 
	\[
		 g( JX, Y) = \omega(-X, Y) = -\omega(X,Y) = dH(Y)
	\]
	for every smooth vector field $Y$ on $M$, $JX$ is the gradient vector field of $H$ with respect to $g$.
	
	Let $(a,b) \subset \R$ be an open interval which does not contain any critical value of $H$.
	For any $r, s \in (a,b)$ with $r < s$, we may identify $H^{-1}(r)$ with $H^{-1}(s)$ via the diffeomorphism $\phi_{r, s} : H^{-1}(r) \rightarrow H^{-1}(s)$ which sends a point 
	$z \in H^{-1}(r)$ to a point in $H^{-1}(s)$ along the gradient vector field $JX$. Thus one gets a diffeomorphism
	\begin{equation}
		\begin{array}{ccc}
			\phi : H^{-1}([r,s]) &\stackrel{\cong}\rightarrow& H^{-1}(s) \times [r,s]\\[0.5em]
			p &\mapsto& (\phi_{H(p), s}(p), H(p)).
		\end{array}
	\end{equation}
	By pulling back $\omega$ to $H^{-1}(s) \times [r,s]$ via $\phi^{-1}$, we have an $S^1$-equivariant symplectomorphism
	\[
		 (H^{-1}([r,s]), \omega|_{(H^{-1}([r,s])}) \cong (H^{-1}(r) \times [r,s], (\phi^{-1})^*\omega) 
	\] with a moment map
	$(\phi^{-1})^*H : H^{-1}(r) \times [r,s] \rightarrow [r,s]$ which is simply a projection on the second factor. Therefore, we may identify $M_s$ with $M_r$ via $\phi$. 
	This identification allows us to think of reduced symplectic forms $\{\omega_t ~|~ t \in (a,b) \}$ as a one-parameter family of symplectic forms on $M_r$.

	\begin{theorem}\label{theorem_DH}\cite{DH}
		Let $\omega_s$ and $\omega_r$ be the reduced symplectic forms on $M_s$ and $M_r$, respectively. By identifying $M_s$ with $M_r$ as described above, we have
		$$ [\omega_s] - [\omega_r] = (r-s)e $$
		where $e \in H^2(M_r;\Q)$ is the Euler class of $S^1$-fibration $\pi_r : H^{-1}(r) \rightarrow M_r.$
	\end{theorem}

	Note that if the action is semifree\footnote{We call an $S^1$-action on $M$ is {\em semifree} if the action is free on $M \setminus M^{S^1}$.}, then the reduced space 
	becomes a smooth manifold and the fibration $\pi_r$ in Theorem \ref{theorem_DH} becomes an $S^1$-bundle so that the Euler class $e \in H^2(M_r; \Z)$ is integral. 

\section{Equivariant cohomology}
\label{secEquivariantCohomology}

    In this section, we recall some well-known facts and theorems about equivariant cohomology of Hamiltonian $S^1$-manifolds. 
	Throughout this section, we take cohomology with the coefficients in $\R$, unless stated otherwise.

	Let $(M,\omega)$ be a $2n$-dimensional closed symplectic manifold equipped with a Hamiltonian circle action. Then the equivariant cohomology $H^*_{S^1}(M)$ is defined by
	\[
		H^*_{S^1}(M) = H^*(M \times_{S^1} ES^1) 
	\] where $ES^1$ is a contractible space on which $S^1$ acts freely. In particular, the equivariant cohomology of
     a point $p$ is given by 
     \[
     		H^*_{S^1}(p) = H^*(p \times_{S^1} ES^1) = H^*(BS^1) 
	\]
	where $BS^1 = ES^1 / S^1$ is the classifying space of $S^1$. Note that $BS^1$ can be constructed as an inductive limit of the sequence of Hopf fibrations
	\begin{equation}
		\begin{array}{ccccccccc}
			S^3          & \hookrightarrow & S^5        & \hookrightarrow & \cdots & S^{2n+1} & \cdots & \hookrightarrow & ES^1 \sim S^{\infty} \\
		\downarrow   &                 & \downarrow &                 & \cdots & \downarrow & \cdots &                 & \downarrow \\
		   \C P^1       & \hookrightarrow & \C P^2     &\hookrightarrow  & \cdots & \C P^n      & \cdots & \hookrightarrow &BS^1 \sim \C P^{\infty}
		\end{array}
	\end{equation}
	so that  $H^*(BS^1) = \R[x]$ where $x$ is an element of degree two such that $\langle x, [\C P^1] \rangle = 1$.

\subsection{Equivariant formality}
\label{ssecEquivariantFormality}
	
	One remarkable fact on the equivariant cohomology of a Hamiltonian $S^1$-manifold is that it is {\em equivariantly formal}.
	Before we state the equivariant formality of $(M,\omega)$, recall that $H^*_{S^1}(M)$ has a natural $H^*(BS^1)$-module structure as follows.
	The projection map $M \times ES^1 \rightarrow ES^1$ on the second factor is $S^1$-equivariant and it induces the projection map
	\[	
		\pi : M \times_{S^1} ES^1 \rightarrow BS^1 
	\] which makes $M \times_{S^1} ES^1$ into an $M$-bundle over $BS^1$ :
	\begin{equation}\label{equation_Mbundle}
		\begin{array}{ccc}
			M \times_{S^1} ES^1 & \stackrel{f} \hookleftarrow & M \\[0.3em]
			\pi \downarrow          &                             &   \\[0.3em]
			BS^1                   &                             &
		\end{array}
	\end{equation}
	where $f$ is an inclusion of $M$ as a fiber. Then $H^*(BS^1)$-module structure on $H^*_{S^1}(M)$ is given by the map $\pi^*$ such that 
	\[
		y \cdot \alpha = \pi^*(y)\cup \alpha 
	\] for $y \in H^*(BS^1)$ and $\alpha \in H^*_{S^1}(M)$. In particular, we have the following sequence of ring homomorphisms
	\[
		 H^*(BS^1) \stackrel{\pi^*} \rightarrow H^*_{S^1}(M) \stackrel{f^*} \rightarrow H^*(M). 
	\]

	\begin{theorem}\label{theorem_equivariant_formality}\cite{Ki}
		Let $(M,\omega)$ be a closed symplectic manifold equipped with a Hamiltonian circle action. Then $M$ is equivariatly formal, that is,
		$H^*_{S^1}(M)$ is a free $H^*(BS^1)$-module so that $$H^*_{S^1}(M) \cong H^*(M) \otimes H^*(BS^1).$$
		Equivalently, the map $f^*$ is surjective with kernel $x \cdot H^*_{S^1}(M)$ where $\cdot$ means the scalar multiplication of $H^*(BS^1)$-module structure on $H^*_{S^1}(M)$.
	\end{theorem}
	
\subsection{Localization theorem}
\label{ssecLocalizationTheorem} 

	Let $\alpha \in H^*_{S^1}(M)$ be any element of degree $k$. Theorem \ref{theorem_equivariant_formality} implies that $\alpha$ can be uniquely
	expressed as 
	\begin{equation}\label{equation_expression}
		 \alpha = \alpha_k \otimes 1 + \alpha_{k-2} \otimes x + \alpha_{k-4} \otimes x^2 + \cdots 
	\end{equation}
	where $\alpha_i \in H^i(M)$ for each $i = k, k-2, \cdots$. We then obtain $f^*(\alpha) = \alpha_k$ where $f$ is given in \eqref{equation_Mbundle}.

	\begin{definition}
		An \textit{integration along the fiber $M$} is an $H^*(BS^1)$-module homomorphism $\int_M : H^*_{S^1}(M) \rightarrow H^*(BS^1)$ defined by
		\[
			\int_M \alpha = \langle \alpha_k, [M] \rangle \cdot 1 + \langle \alpha_{k-2}, [M] \rangle \cdot x + \cdots 
		\]
		for every $ \alpha = \alpha_k \otimes 1 + \alpha_{k-2} \otimes x + \alpha_{k-4} \otimes x^2 + \cdots \in H^k_{S^1}(M).$ 
		Here, $[M] \in H_{2n}(M; \Z)$ denotes the fundamental homology class of $M$.
	\end{definition}

	Note that $\langle \alpha_i, [M] \rangle$ is zero for every $i < \dim M = 2n$, and $\alpha_i = 0$ for every $i > \deg \alpha$ by dimensional reason. 
	Therefore, we have the following corollary.

	\begin{corollary}\label{corollary : localization degree 2n}
		Let $\alpha \in H^{2n}_{S^1}(M)$ be given in \eqref{equation_expression}.
		Then we have 
		\[
			\int_M \alpha = \langle \alpha_{2n}, [M] \rangle = \langle f^*(\alpha), [M] \rangle.
		\]
		Furthermore, if $\alpha$ is of degree less than $2n$, then we have 
		\[
			\int_M \alpha = 0.
		\]	
	\end{corollary}

	Now, let $M^{S^1}$ be the fixed point set of the $S^1$-action on $M$ and let $F \subset M^{S^1}$ be a fixed component. Then the inclusion $i_F : F \hookrightarrow M$ 
	induces a ring homomorphism $$i_F^* : H^*_{S^1}(M) \rightarrow H^*_{S^1}(F) \cong H^*(F) \otimes H^*(BS^1).$$
	For any $\alpha \in H^*_{S^1}(M)$, we call the image $i_F^*(\alpha)$ \textit{the restriction of $\alpha$ to $F$} and denote by 
	\[
		\alpha|_F := i_F^*(\alpha).
	\] 
	Then we may compute $\int_M \alpha$ concretely by using the following theorem due to Atiyah-Bott \cite{AB} and Berline-Vergne \cite{BV}.

	\begin{theorem}[ABBV localization]\label{theorem_localization}
		For any $ \alpha \in H^*_{S^1}(M)$, we have
		\[
			\int_M \alpha = \sum_{F \subset M^{S^1}} \int_F \frac{\alpha|_F}{e^{S^1}(F)}
		\]
		where $e^{S^1}(F)$ is the equivariant Euler class of the normal bundle $\nu_F$ of $F$ in $M$. That is, $e^{S^1}(F)$ is the Euler class of the bundle 
		\[
			\nu_F \times_{S^1} ES^1 \rightarrow F \times BS^1.
		\] induced from the projection $\nu_F \times ES^1 \rightarrow F \times ES^1$.
	\end{theorem}

	One more important advantage of a Hamiltonian $S^1$-action is that any two equivariant cohomology classes are distinguished by their images of the restriction to the fixed point set.

          \begin{theorem}\cite{Ki}(Kirwan's  injectivity theorem)\label{theorem_Kirwan_injectivity}
		Let $(M,\omega)$ be a closed Hamiltonian $S^1$-manifold and $i : M^{S^1} \hookrightarrow M$ be the inclusion of the fixed point set.
		Then the induced ring homomorphism $i^* : H^*_{S^1}(M) \rightarrow H^*_{S^1}(M^{S^1})$ is injective.
	\end{theorem}

\subsection{Equivariant Chern classes}
\label{ssecEquivariantChernClasses} 
                                           
	\begin{definition}
		Let $\pi : E \rightarrow B$ be an $S^1$-equivariant complex vector bundle over a topological space $B$. 
		Then the \textit{$i$-th equivariant Chern class $c_i^{S^1}(E)$} is defined as an $i$-th Chern class of the complex vector bundle $\widetilde{\pi}$
		\[\xymatrix{E \times ES^1 \ar[r]^{/S^1} \ar[d]^{\pi} & E \times_{S^1} ES^1 \ar[d]^{\widetilde{\pi}} \\ B \times ES^1 \ar[r]^{/S^1} & B \times_{S^1} ES^1}\]
		For an almost complex $S^1$-manifold $(X,J)$, we denote by $c_i^{S^1}(X) := c_i^{S^1}(TX,J)$ the $i$-th equivariant Chern class of 
		the complex tangent bundle $(TX,J)$.
	\end{definition}

	For a closed symplectic manifold $(M,\omega)$ equipped with a Hamiltonian circle action, there exists an $S^1$-invariant 
	$\omega$-compatible almost complex structure $J$ on $M$. 
	Moreover, the space of such almost complex structures is contractible so that the Chern classes of $(M,J)$ do not depend on the choice of $J$.
	The following proposition gives an explicit formula for the restriction of the equivariant first Chern class of $(M, J)$ on each fixed component.

	\begin{proposition}\label{proposition_equivariant_Chern_class}
		Let $(M,\omega)$ be a $2n$-dimensional closed symplectic manifold equipped with a Hamiltonian circle action.
		Let $F \subset M^{S^1}$ be any fixed component with weights $\{ w_1(F), \cdots, w_n(F) \} \in \Z^n$ of the tangential $S^1$-representation at $F$. 
		Then the restriction $c_1^{S^1}(M)|_F \in H^*_{S^1}(F) \cong H^*(F) \otimes H^*(BS^1)$
		is given by 
		\[
			c_1^{S^1}(M)|_F = c_1(M)|_F \otimes 1 + 1 \otimes (\sum_i w_i)x.
		\]
	\end{proposition}
                                           
	\begin{proof}
		Let $TF$ be the tangent bundle of $F$ and $\nu_F$ be the normal bundle of $F$, respectively, and consider the following bundle map : 
		\[\xymatrix{TM|_F \times ES^1 \ar[r]^{/S^1} \ar[d]^{\pi} & TM|_F \times_{S^1} ES^1 \ar[d]^{\widetilde{\pi}} \\ F \times ES^1 \ar[r]^{/S^1} & F \times_{S^1} ES^1}\]
		By definition, we have $c_1^{S^1}(M)|_F = c_1(\widetilde{\pi})$. Since $TM|_F = TF \oplus \nu_F$, 
		$\widetilde{\pi}$ is decomposed into $\widetilde{\pi}_1 \oplus \widetilde{\pi}_2$ where $\widetilde{\pi}_1$ and $\widetilde{\pi}_2$ are given by
		\[\xymatrix{TF \times_{S^1} ES^1 \ar[d]^{\widetilde{\pi}_1}  & \nu_F \times_{S^1} ES^1 \ar[d]^{\widetilde{\pi}_2} \\ F \times BS^1  & F \times BS^1}\]
		Since $F$ is fixed by the $S^1$-action, the induced $S^1$-action on $TF$ is trivial, and therefore we get 
		\[
			c_1(\widetilde{\pi}_1) = c_1(TF) \otimes 1.
		\]
		Note that the restriction of $\widetilde{\pi}_2$ on $F \times \{\mathrm{pt}\} \subset F \times BS^1$ 
		is $\nu_F$. Also, the restriction of $\widetilde{\pi}_2$ on $\{ \mathrm{pt} \} \times BS^1$ is $\C^n \times_{S^1} ES^1$. More precisely, we have 
		\[\xymatrix{\C^n \times ES^1 \ar[d] \ar[r]^{/S^1} & \C^n \times_{S^1} ES^1 \ar[d]^{\widetilde{\pi}_2|_{\{\mathrm{pt}\} \times BS^1}} \\  ES^1 \ar[r]^{/S^1} & BS^1}\]
		and the total Chern class of the restricted bundle $\widetilde{\pi}_2 : \C^n \times_{S^1} ES^1 \rightarrow BS^1$ is $(1+w_1x)(1+w_2x)\cdots(1+w_nx)$. 
		Thus it follows that 
		\begin{displaymath}
			\begin{array}{ll}
				c_1(\widetilde{\pi}) = c_1(\widetilde{\pi}_1) + c_1(\widetilde{\pi}_2) & = c_1(TF) \otimes 1 + c_1(\nu(F)) \otimes 1 + 1 \otimes (\sum_i w_i)x\\[1em]
				& =c_1(M)|_F \otimes 1 + (\sum_i w_i)x
			\end{array}
		\end{displaymath}
		and this completes the proof.
	\end{proof}

\subsection{Equivariant symplectic classes} 
\label{ssecEquivariantSymplecticClasses}

	Let $H : M \rightarrow \R$ be a moment map for the $S^1$-action on $(M,\omega)$.  
	Let $\omega_H := \omega + d(H \cdot \theta)$ be a two-form on $M \times ES^1$ where $\theta$ is a connection form of the principal $S^1$-bundle 
	$ES^1 \rightarrow BS^1$. 
	By the $S^1$-invariance of $\omega$, $\theta$, and $H$, we obtain 
	\[
		\mcal{L}_{\underbar{X}} \omega_H = i_{\underbar{X}} \omega_H = 0
	\] 
	where $\underbar{X}$ denotes the vector field generated by the diagonal action on $M \times ES^1$. 
	Thus we may push-forward $\omega_H$ to $M \times_{S^1} ES^1$ 
	and denote the push-forward of $\omega_H$ by $\widetilde{\omega}_H$, which we call the \textit{equivariant symplectic form with respect to $H$}.
	Also, the corresponding cohomology class $[\widetilde{\omega}_H] \in H^2_{S^1}(M)$ is called the \textit{equivariant symplectic class with respect to $H$}.
	Note that the restriction of $\widetilde{\omega}_H$ on each fiber $M$ is precisely $\omega$.

	\begin{proposition}\label{proposition_equivariant_symplectic_class}
		Let $F \in M^{S^1}$ be a fixed component. Then we have
		\[
			[\widetilde{\omega}_H]|_F = [\omega]|_F \otimes 1 - H(F) \otimes x.
		\]
	\end{proposition}
	
	\begin{proof}
		Consider the push-forward of 
		$\widetilde{\omega}_H|_F = (\omega - dH \wedge \theta - H \cdot d\theta)|_{F \times ES^1}$ to $F \times BS^1$. 
		Since the restriction $dH|_{F \times ES^1}$ vanishes, we have $[\widetilde{\omega}_H]|_F = [\omega]|_F \otimes 1 - H(F) \otimes \widetilde{[d\theta]}|_{BS^1}$ where 
		$\widetilde{[d\theta]}$ is the push-forward of $[d\theta]$ to $F \times BS^1$. Since the push-forward of $d\theta$ is a curvature form which represents the first Chern class of 
		$ES^1 \rightarrow BS^1$, we have $\widetilde{[d\theta]} = -x$. Therefore, we have $[\widetilde{\omega}_H]|_F = [\omega]|_F \otimes 1 - H(F) \otimes x$.
	\end{proof}

\section{Monotone semifree Hamiltonian $S^1$-manifolds}
\label{secMonotoneSemifreeHamiltonianS1Manifolds}
	
	A \textit{monotone symplectic manifold} is a symplectic manifold $(M,\omega)$ such that $c_1(M) = \lambda \cdot [\omega]$ for some positive real number 
	$\lambda \in \R^+$ called the {\em monotonicity constant}. 
	In this section, we serve crucial ingredients for the classification of closed monotone semifree Hamiltonian $S^1$-manifolds.			
	Throughout this section, we assume that $(M,\omega)$ is a monotone symplectic manifold equipped with a 
	semifree\footnote {We call an $S^1$-action \textit{semifree} if it is free outside the fixed point set.} Hamiltonian circle action
	such that $c_1(M) = [\omega]$ with a moment map $H : M \rightarrow \R$. 

\subsection{Topology of reduced spaces} 
\label{ssecTopologyOfReducedSpaces}

	Let $r \in \R$ be a regular value of $H$. The ``semifreeness'' implies that the level set $H^{-1}(r)$ is a free $S^1$-manifold of dimension $2n-1$. 
	The quotient space $M_r := H^{-1}(r) / S^1$, together with the reduced symplectic form $\omega_r$, becomes is a closed symplectic manifold of dimension $2n-2$. 
	We call $(M_r, \omega_r)$ {\em the reduced space}, or the {\em symplectic reduction at level $r$}.
 
	Let $I \subset \R$ be an open interval consisting of regular values of $H$, then any two reduced spaces $M_r$ and $M_s$ ($r, s \in I$) 
	can be identified via the map $\phi_{r,s} / S^1 : M_r \rightarrow M_s$ induced from 
	\[
		\phi_{r,s}  : H^{-1}(r) \rightarrow H^{-1}(s)
	\] given in Section \ref{ssecDuistermaatHeckmanTheorem}.
	In particular, $M_r$ and $M_s$ are diffeomorphic. 
	
	The topology of a reduced space $M_t$ changes as $t$ crosses a critical value of $H$. 
	More precisely, let $c \in \R$ be a critical value of $H$ and let $p \in H^{-1}(c)$ be a fixed point of the $S^1$-action whose Morse-Bott index is $2k$. 
	By the equivariant Darboux theorem \ref{theorem_equivariant_darboux}, there is an $S^1$-equivariant complex coordinate chart $(\mcal{U}_p, z_1, \cdots, z_n)$ near $p$ 
	such that
	\begin{itemize}
        		\item $\omega|_{\mcal{U}_p} = \frac{1}{2i} \sum_i dz_i \wedge d\bar{z_i},$ and
	        	\item the action can be expressed by 
        		\[
        			t \cdot (z_1, \cdots, z_n) = (t^{-1}z_1, \cdots, t^{-1}z_k, z_{k+1}, \cdots, z_{k+l}, tz_{k+l+1}, \cdots, tz_n), \quad t \in S^1, 
		\] 
		\item the moment map is given by 
		\[
			H(z_1, \cdots, z_n) = H(p) - \frac{1}{2}(\sum_{i=1}^k |z_i|^2 - \sum_{i=k+l+1}^n |z_i|^2)
		\]
	    \end{itemize}
	where $2l$ is a dimension of the fixed component $Z_p$ containing $p$. 
	Thus, $\mcal{U}_p \cap H^{-1}(c)$ locally looks like the solution space of the equation 
	\[
		\sum_{i=1}^k |z_i|^2 - \sum_{i=k+l+1}^n |z_i|^2 = 0.
	\]
	For a small parameter $\epsilon > 0$, we have 
	\[
		\mcal{U}_p \cap H^{-1}(c-\epsilon) \cong \{(z_1,\cdots,z_n) \in \mcal{U}_p ~| ~\sum_{i=1}^k |z_i|^2 = \sum_{i=k+l+1}^n |z_i|^2 + 2\epsilon \}.
	\]
	In particular, $\mcal{U}_p \cap H^{-1}(c-\epsilon) \cong S^{2k-1} \times \C^l \times \C^{n-k-l} $ and it contains 
	\begin{equation}\label{equation_c_epsilon}
		\{ (z_1,\cdots,z_n) \in \mcal{U}_p \cap H^{-1}(c-\epsilon) ~| ~ z_{k+l+1} = \cdots = z_n = 0 \} = \mcal{S}_{2\epsilon} \times \C^l \times 0 \in \C^k \times \C^l \times \C^{n - k - l}
	\end{equation}
	where $\mcal{S}_{2\epsilon} = \{ (z_1, \cdots, z_n) \in \mcal{U}_p ~| ~z_{k+1} = \cdots = z_n = 0 \}$ is the $(2k-1)$-dimensional sphere of radius $\sqrt{2\epsilon}$ 
	centered at the origin in $\mcal{U}_p$.
	
	If we let $g_J$ be an $S^1$-invariant metric on $M$ whose restriction onto $\mcal{U}_p$ is the standard complex structure, then 
	the set in \eqref{equation_c_epsilon} is the intersection of $\mcal{U}_p \cap H^{-1}(c-\epsilon)$ and the stable submanifold of $Z_p$ with respect to 
	$g_J$. (See Section \ref{ssecDuistermaatHeckmanTheorem}.)
	The gradient flow induces a surjective continuous map 
	\[
		\phi_{c-\epsilon, c} : H^{-1}(c-\epsilon) \cap \mcal{U}_p \rightarrow H^{-1}(c) \cap \mcal{U}_p
	\] and this map sends $\mcal{S}_\epsilon \times \C^l \times 0$ to $0 \times \C^l \times 0$. We similarly apply this argument to other fixed point in $H^{-1}(c)$ so that 
	we get a surjective map 
	\[
		\phi_{c-\epsilon, c} /S^1 : M_{c-\epsilon} \rightarrow M_c
	\] When $k=1$, i.e., if $Z_p$ is of index two, 
	then $\mcal{S}_\epsilon \times_{S^1} \C^l \cong \C^l$ so that $\phi_{c-\epsilon, c} / S^1$ is bijective on $\mcal{U}_p \cap H^{-1}(c-\epsilon)$. 
	Similarly, if $Z_p$ is of co-index two (i.e., $k=n-l-1$), 
	we have 
	\[
		\mcal{U}_p \cap H^{-1}(c-\epsilon) \cong S^{2n-2l-3} \times \C^l \times \C
	\] so that $\phi_{c-\epsilon, c} / S^1$ maps $\left(S^{2n-2l-3} \times_{S^1} \C \right) \times \C^l$ to 
	$\C^{n-l-1} \times \C^l$, that is, $\phi_{c-\epsilon, c} / S^1$ is the blow-down map along $\C^l$.  
	Therefore we have the following.

	\begin{proposition}\label{proposition_topology_reduced_space}\cite{McD2}\cite{GS}
		Let $(M,\omega)$ be a closed semifree Hamiltonian $S^1$-manifold with a moment map $H : M \rightarrow \R$ and $c \in \R$ be a critical value of $H$. 
		If $Z_c := H^{-1}(c) \cap M^{S^1}$ consists of index-two (index-four, resp.) fixed points, then $M_c = H^{-1}(c) / S^1$ is smooth and is diffeomorphic to $M_{c-\epsilon}$. 
		Also, $M_{c+\epsilon}$ is the blow-up (blow-down, resp.) of $M_c$ along $Z_c$.
	\end{proposition}

	More generally, Guillemin and Sternberg \cite[Theorem 11.1]{GS} described how the topology of a reduced space varies when crossing a critical value of $H$.  
	Here we introduce their result briefly, even though we apply it to a very special case.
	Let $Z \subset M^{S^1}$ be a critical submanifold (fixed component) of $H$ at level $c \in \R$ and its signature is $(2p, 2q)$. 
	For a sufficiently small $\epsilon >0$, if we perform an $S^1$-equivariant symplectic blow-up of $(M,\omega)$ along $Z$ an $\epsilon$-amount, then we get a new Hamiltonian $S^1$-manifold 
	$(\widetilde{M}, \widetilde{\omega})$ and it has two fixed components $\widetilde{Z}_{c-\epsilon}$ and $\widetilde{Z}_{c+\epsilon}$ at level $c-\epsilon$ and $c+\epsilon$, respectively.
	Moreover, $\widetilde{Z}_{c-\epsilon}$ (respectively $\widetilde{Z}_{c+\epsilon}$) is of signature $(2,2q)$ (respectively signature $(2, 2p)$). Since the blow-up changes nothing outside 
	$H^{-1}([c-\epsilon, c+\epsilon])$, 
	we see that the reduced space $M_{c-\eta}$ 
	(respectively $M_{c+\eta}$) is diffeomorphic to $\widetilde{M}_{c-\epsilon-\eta}$ (respectively $\widetilde{M}_{c+\epsilon+\eta}$) for a sufficiently small $\eta >0$.
	Therefore, $M_{c+\eta}$ is obtained by ``blowing down along $\widetilde{Z}_{c+\epsilon}$'' of a ``blow-up'' of $M_{c-\eta}$ along $\widetilde{Z}_{c - \epsilon}$.
	 (This relation is so-called a {\em birational equivalence of reduced spaces}.
	See \cite{GS} for more details.)
	Note that Proposition \ref{proposition_topology_reduced_space} is the special case ($p=2$, or $q=2$) of \cite[Theorem 11.1]{GS}.

\subsection{Variation of Euler classes} 
\label{ssecVariationOfEulerClasses}
	
	Recall that the Duistermaat-Heckman's theorem \ref{theorem_DH} says that 
	\[
		[\omega_r] - [\omega_s] = (s-r)e, \quad r,s \in I
	\]
	where $I$ is an interval consisting of regular values of $H$ and $e \in H^2(M_r; \Z)$ denotes the Euler class of the principal $S^1$-bundle 
	$\pi_r : H^{-1}(r) \rightarrow M_r$. 
	
	As the topology of a reduced space changes, the topology of a principal bundle $\pi_r : H^{-1}(r) \rightarrow M_r$ changes. 
	In \cite{GS}, Guillemin and Sternberg also provide a variation formula of the Euler class of the principal $S^1$-bundle $\pi_r : H^{-1}(r) \rightarrow M_r$
	when $r$ passes through a critical value of index (or co-index) two.  

	\begin{lemma}\cite[Theorem 13.2]{GS}\label{lemma_Euler_class}
		Suppose that $F_c = M^{S^1} \cap H^{-1}(c)$ consists of fixed components $F_1, \cdots, F_k$ each of which is of index two.   
		Let $e^-$ and $e^+$ be the Euler classes of principal $S^1$-bundles $\pi_{c-\epsilon} : H^{-1}(c-\epsilon) \rightarrow M_{c-\epsilon}$ and 
		$\pi_{c+\epsilon} : H^{-1}(c+\epsilon) \rightarrow M_{c+\epsilon}$, respectively.
		Then 
		\[
			e^+ = \phi^*(e^-) + E \in H^2(M_{c+\epsilon}; \Z)
		\] where $\phi : M_{c+\epsilon} \rightarrow M_{c-\epsilon}$ is the blow-down map and $E$ is the Poincar\'{e} dual of the exceptional divisor of $\phi$.
	\end{lemma}
	
	The {\em Dustermaat-Heckman function}, denoted by ${\bf DH}$, is a function defined on $\mathrm{Im}~H$ and it assigns the symplectic area of the reduced space $M_r$, 
	i.e.,
	\[
		\begin{array}{ccccl}\vs{0.1cm}
			{\bf DH} & : & \mathrm{Im}~H & \rightarrow & \R \\ \vs{0.1cm}
					&	& r  & \mapsto & \ds {\bf DH}(r) := \int_{M_r} \omega_r^{n-1}. \\
		\end{array}
	\]
	It follows from Theorem \cite{DH} that the Duistermaat-Heckman function is a piecewise polynomial function, that is, if $(b,c) \subset \mathrm{Im}~H$ is an open interval consisting of 
	regular values of $H$, then the restriction of ${\bf DH}$ onto $(b,c)$ is a polynomial in one variable.

	Using Lemma \ref{lemma_Euler_class}, we obtain the following. 
	
	\begin{lemma}\label{lemma_DH_dec}
		Let $(M,\omega)$ be a six-dimensional semifree 
		Hamiltonian $S^1$-manifold with a proper moment map $H : M \rightarrow \R$. Suppose that $b < c < d$ are consecutive critical values of $H$ where
		$H^{-1}(c) \cap M^{S^1}$ consists of $k$ isolated fixed points of index two. 
		Let ${\bf DH}_- : (b,c) \rightarrow \R$ and ${\bf DH}_+ : (c,d) \rightarrow \R$ be the restriction of the Duisermaat-Heckman function on $(b,c)$ and $(c,d)$, respectively.
		Extending ${\bf DH}_-$ to the polynomial function on $\R$, we have 
		\[
			{\bf DH}_-(r) > {\bf DH}_+(r) \quad \text{for every $r \in (c,d)$.}
		\]
	\end{lemma}

	\begin{proof}
		By the assumption ``{\em semifreeness}'', all reduced spaces are smooth by Lemma \ref{proposition_topology_reduced_space}. 
		Let $E_1, \cdots, E_k \in H^2(M_{c + \epsilon}, \Z)$ be the Poincar\'{e} dual of the exceptional divisors in $M_{c+\epsilon}$. 
		By Lemma \ref{lemma_Euler_class}, we have 
		\[
			{\bf DH}_- (c + t)  - {\bf DH}_+(c + t) = \ds \int_{M_{c+\epsilon}} \left([\omega_c] - t e^- \right)^{2} - \left([\omega_c] - t (e^-  + E_1 + \cdots + E_k) \right)^{2} = k t^2 > 0, \quad t \in (0, d-c)
		\]
		since 
		\[
			[\omega_c] \cdot E_i = 0, \quad e_- \cdot E_i = 0, \quad E_i \cdot E_j = 0 \hs{0.1cm} \text{for} \hs{0.1cm} i \neq j, \quad  \text{and} \hs{0.1cm} \int_{M_{c + \epsilon}} E_i \cdot E_i = -1. 
		\]
		This completes the proof.
	\end{proof}

\subsection{Equivariant monotone symplectic form}
\label{ssecEquivariantMonotoneSymplecticForm}

	The monotonicity of $(M,\omega)$ guarantees the existence of 
	so-called the {\em equivariant monotone symplectic form}.

	\begin{proposition}\label{proposition_normalized_moment_map}
		Suppose that $(M,\omega)$ is a monotone closed Hamiltonian $S^1$-manifold such that $c_1(TM) = [\omega] \in H^2(M)$.
		Then there exists a unique moment map $H : M \rightarrow \R$ such that 
		\begin{equation}\label{equation_balanced}
			[\widetilde{\omega}_H] = c_1^{S^1}(TM) \in H^2_{S^1}(M)
		\end{equation}
		where $\widetilde{\omega}_H$ is the equivariant symplectic form defined in Section \ref{ssecEquivariantSymplecticClasses}.
	\end{proposition}

	\begin{proof}
		Recall that the equivariant formality of $(M,\omega)$ yells that
		\[
			 f^* : H^*_{S^1}(M) \rightarrow H^*(M)
		\]
		is a surjective ring homomorphism by Theorem \ref{theorem_equivariant_formality}
		whose kernel is given by
		\[
			\ker{f^*} = x \cdot H^*_{S^1}(M)
		\] 
		where $x \in H^2(BS^1)$ is the generator of $H^*(BS^1)$.
           
           Choose any moment map $\bar{H} : M \rightarrow \R$. Since $i^*([\widetilde{\omega}_{\bar{H}}]) = [\omega]$ and $i^*(c_1^{S^1}(TM)) = c_1(TM) = [\omega]$, the difference 
           $[\widetilde{\omega}_{\bar{H}}] - c_1^{S^1}(TM) \in H_{S^1}^2(M)$ is in $\ker f^*$ so that 
           we have
		\[
			[\widetilde{\omega}_{\bar{H}}] + x\cdot a = c_1^{S^1}(TM) 
		\]
		for some $a \in \R$ by Theorem \ref{theorem_equivariant_formality}. Set $H := \bar{H} - a$ as a new moment map.
		From the definition of equivariant symplectic form in Section \ref{ssecEquivariantSymplecticClasses}, we get 
		\[
			[\widetilde{\omega}_H] - [\widetilde{\omega}_{\bar{H}}] = [\widetilde{\omega}_{\bar{H}-a}] - [\widetilde{\omega}_{\bar{H}}] = -a[d\theta] = x \cdot a,
		\]
		and therefore $[\widetilde{\omega}_H] = c_1^{S^1}(TM)$. 
	\end{proof}

	\begin{definition}\label{definition_balanced}
		The moment map $H$ in Proposition \ref{proposition_normalized_moment_map} is called the {\em balanced moment map}. 
	\end{definition}

	\begin{corollary}\label{corollary_sum_weights_moment_value}
		Assume that $c_1(M) = [\omega]$ and let $H : M \rightarrow \R$ be the balanced moment map. For each fixed component $Z \subset M^{S^1}$, 
		we have $H(Z) = -\Sigma(Z)$ where $\Sigma(Z)$ denotes the sum of all weights of the $S^1$-action at $Z$.
	\end{corollary}

	\begin{proof}
		Recall that Proposition \ref{proposition_equivariant_symplectic_class} and Proposition \ref{proposition_equivariant_Chern_class} imply that 
		\[
			[\widetilde{\omega}_H]|_Z = [\omega]|_Z - x \cdot H(Z), \quad \text{and} \quad c_1^{S^1}(TM)|_Z = c_1(TM)|_Z + x \cdot \Sigma(Z)  
		\]
		for each fixed component $Z \subset M^{S^1}$.
		Since $[\omega] = c_1(TM)$ and $H$ is balanced by our assumption, we have $[\widetilde{\omega}_H] = c_1^{S^1}(TM)$
		by Proposition \ref{proposition_normalized_moment_map}, and therefore we have
		$H(Z) = -\Sigma(Z)$ for every fixed component $Z \subset M^{S^1}$ by Kirwan's injectivity theorem \ref{theorem_Kirwan_injectivity}.
        \end{proof}

	\begin{remark}\label{remark_balanced}
		Note that a moment map $H$ is called {\em normalized} if 
		\[
			\int_M H\omega^n = 0.
		\]
		It is worth mentioning that two notions `{\em normalized}' and `{\em balanced}' are different. 
		Indeed, we can easily check the difference between two notions in the case where $M$ is a monotone blow-up of $\p^1 \times \p^1$.
	\end{remark}
	
%
%
	
\subsection{Monotonicity of symplectic reduction}
\label{ssecMonotonicityOfSymplecticReduction}

		If $H$ is balanced, then the reduced space $H^{-1}(0) / S^1$ inherits the monotone reduced symplectic form.
		To show this, consider the embedding $H^{-1}(0) \hookrightarrow M$, which is obviously $S^1$-equivariant, 
		and let 
		\[
			\kappa : H^*_{S^1}(M) \rightarrow H^*_{S^1}(H^{-1}(0)) \cong H^*(M_0) 
		\]
		be an induced ring homomorphism, which we call the \textit{Kirwan map}.
        
		\begin{proposition}\label{proposition_monotonicity_preserved_under_reduction}
			Let $(M,\omega)$ be a semifree Hamiltonian $S^1$-manifold with $c_1(TM) = [\omega]$ and $H$ be the balanced moment map.
			If 0 is a regular value of $H$, then $(M_0, \omega_0)$ is a monotone symplectic manifold with $[\omega_0] = c_1(TM_0)$
		\end{proposition}

		\begin{proof}
			Observe that $\kappa$ takes $c_1^{S^1}(TM)$ to $c_1(TM_0)$ and $[\widetilde{\omega}_H]$ to 
			$[\omega_0]$. Since $[\widetilde{\omega}_H] = c_1^{S^1}(TM)$ by Proposition \ref{proposition_normalized_moment_map}, we have 
			$c_1(TM_0) = [\omega_0]$.
		\end{proof}

		\begin{remark}\label{remark_monotone_reduction_criticalvalue}
			Note that Proposition \ref{proposition_monotonicity_preserved_under_reduction} holds even for the case where $0$ is a critical value of $H$ under the assumption 
			that the symplectic reduction at level $0$ is well-defined. This is the case where $M^{S^1} \cap H^{-1}(0)$ consists of index two or co-index two fixed components, respectively,
			See Proposition \ref{proposition_topology_reduced_space}.
		\end{remark}

		It is an immediate consequence from Proposition \ref{proposition_monotonicity_preserved_under_reduction} that if $\dim M = 6$, then $M_0$ should be diffeomorphic to 
		either $\p^2$, $\p^1 \times \p^1$, or $\p^2 \# k \overline{\p}^2$ for $1 \leq k \leq 8$ by the classification (by Ohta and Ono \cite{OO2}) of closed monotone symplectic four manifolds. 
		
		We end this section by the following lemma, 
		which give a list of all cohomology classes,
		called {\em exceptional classes}, in $H^2(M_0 ; \Z)$ each of which is represented by a symplectically embedded 2-sphere with the self-intersection number $-1$.
		
		\begin{lemma}\label{lemma_list_exceptional}\cite[Section 2]{McD2}
			Suppose that $M_0$ is the $k$-times simultaneous symplectic blow-up of $\p^2$ with the exceptional divisors $C_1, \cdots, C_k$.
			Denote by $\mathrm{PD}(C_i) = E_i \in H^2(M_0; \Z)$.  
			Then all possible exceptional classes are listed as follows (modulo permutations of indices) : 
			\[
				\begin{array}{l}
					E_1, u - E_{12},  \quad 2u - E_{12345}, \quad 3u - 2E_1 - E_{234567}, \quad 4u - 2E_{123} - E_{45678}  \\ \vs{0.1cm}
					5u - 2E_{123456}  - E_{78},  \quad 6u - 3E_1 - 2E_{2345678}  \\
				\end{array}
			\]
			Here, we denote by $E_{j \cdots n} := \sum_{i=j}^n E_i = $. Furthermore, elements involving $E_i$ do not appear in $X_j$ with $j < i$. 
		\end{lemma}

\section{Fixed point data}
\label{secFixedPointData}
	
	Consider a $2n$-dimensional closed Hamiltonian $S^1$-manifold $(M,\omega)$ with a moment map $H : M \rightarrow I \subset \R$. Assume that the critical values of $H$ are given by 
	\[
		\min H = c_1 < \cdots < c_k = \max H.
	\]	
	One can decompose $M$ into a union of $2n$-dimensional Hamiltonian $S^1$-manifolds $\{ (N_j, \omega_j) \}_{1 \leq j \leq 2k-1}$ with boundary where 
	\[
		N^{2j-1} = H^{-1}(\underbrace{[c_j - \epsilon, c_j + \epsilon]}_{ =: I_{2j-1}}), \quad N^{2j} = H^{-1}(\underbrace{[c_j + \epsilon, c_{j+1} - \epsilon]}_{=: I_{2j}})
	\]
	 and $\epsilon > 0$ is chosen to be sufficiently small so that $I_{2j-1}$ contains exactly one critical value $c_j$ of $H$ for each $j$.
	 We call those $N_j$'s {\em slices}. 
	 
	 \begin{definition}\label{definition_regular_slice}\cite[Definition 2.3]{G}
		A {\em regular slice} $(N,\sigma,K, I)$ is a free Hamiltonian $S^1$-manifold $(N, \sigma)$ with boundary and 
		$K : N \rightarrow I$ is a surjective proper moment map where $I = [a,b]$ is a closed interval. 
	\end{definition}

	By definition, a regular slice does not contain a fixed point and the image of a moment map consists of regular values. 

	\begin{definition}\label{definition_critical_slice}
		A {\em critical slice} $(N, \sigma, K, I)$ is a semifree Hamiltonian $S^1$-manifold $(N, \sigma)$ with boundary together with a surjective proper moment map 
		$K : N \rightarrow I = [a,b]$ such that 
		there exists exactly one critical value $c \in [a,b]$ satisfying one the followings : 
		\begin{itemize}
			\item (interior slice) $c \in (a,b)$, 
			\item (maximal slice) $c = b$ and $K^{-1}(c)$ is a critical submanifold, 
			\item (minimal slice) $c = a$ and $K^{-1}(c)$ is a critical submanifold. 
		\end{itemize}
		An interior critical slice is called {\em simple} if every fixed component in $K^{-1}(c)$ has the same Morse-Bott index. 
	\end{definition}
	
	One can define an isomorphism of slices as follows : two slices $(N_1,\sigma_1,K_1, I_1)$ and $(N_2,\sigma_2,K_2, I_2)$ are said to be 
	{\em isomorphic} if there exists an $S^1$-equivariant symplectomorphism $\phi : (N_1, \sigma_1) \rightarrow (N_2, \sigma_2)$ satisfying 
	\[
		\xymatrix{N_1 \ar[r]^{\phi} \ar[d]_{K_1} & N_2 \ar[d]^{K_2} \\ I_1 \ar[r]^{ + k} & I_2}
	\]
	where $+k$ denotes the translation map as the addition of some constant $k \in \R$. 
	We note that the notion of ``slices'' are already introduced by Li \cite{Li3} for constructing a closed Hamiltonian $S^1$-manifold by gluing slices. (She call a slice a {\em local piece} in \cite{Li3}.)

	\begin{lemma}\cite[Lemma 13]{Li3}\cite[Lemma 1.2]{McD2}\label{lemma_gluing}
		Two slices $(N_1, \sigma_1, K_1, [a,b])$ and $(N_2, \sigma_2, K_2, [b,c])$ can be glued along $K_i^{-1}(b)$ if there exists a diffeomorphism 
		\[
			\phi : (N_1)_b \rightarrow (N_2)_b, \quad \quad (N_i)_b := K_i^{-1}(b) / S^1
		\]
		such that 
		\begin{itemize}
			\item $\phi^* (\sigma_2)_b = (\sigma_1)_b$, and 
			\item $\phi^* (e_2)_b = (e_1)_b$ 
		\end{itemize}
		where $(\sigma_i)_b$ and $(e_i)_b$ denote the reduced symplectic form on $(N_i)_b$ and the Euler class of the principal $S^1$-bundle 
		$K_i^{-1}(b) \rightarrow (N_i)_b$, respectively.
	\end{lemma}
	
	Thus if we have a collection of slices (containing maximal and minimal critical slices) which satisfy the compatibility conditions given in Lemma \ref{lemma_gluing}, then 
	we can construct a closed Hamiltonian $S^1$-manifold. It is worth mentioning that the resulting closed manifold may not be unique, i.e., it might depend on the choice of 
	gluing maps. 
	Gonzalez \cite{G} used slices to classify semifree Hamiltonian $S^1$-manifolds in terms of so-called {\em fixed point data}. 
	Roughly speaking, he considered which conditions on a fixed point data of a given Hamiltonian $S^1$-manifold $(M,\omega)$ determine $(M,\omega)$ uniquely up to $S^1$-equivariant 
	symplectomorphism.

	Now, we focus on the case where $(M,\omega)$ is a six-dimensional closed monotone symplectic manifold 
	equipped with an effective semifree Hamiltonian $S^1$-action with the balanced moment map $H : M \rightarrow \R$.
	We further assume that, by scaling $\omega$ if necessary, $c_1(TM) = [\omega]$.
	
	\begin{definition}\cite[Definition 1.2]{G}\label{definition_fixed_point_data} 
		Let $(M,\omega)$ be a six-dimensional closed semifree Hamiltonian $S^1$-manifold equipped with a moment map $H : M \rightarrow I$ such that all critical level sets are simple
		in the sense of Definition \ref{definition_critical_slice}. 
		A {\em fixed point data} of $(M,\omega, H)$, denoted by $\frak{F}(M, \omega, H)$, is a collection 
		\[
			 \frak{F} (M, \omega, H) := \left\{(M_{c}, \omega_c, Z_c^1, Z_c^2, \cdots,  Z_c^{k_c}, e(P_{c}^{\pm})) ~|~c \in \mathrm{Crit} ~H \right\}
		\]
		which consists of the information below.
		\begin{itemize}
			\item $(M_c, \omega_c)$\footnote{$M_c$ is smooth manifold under  the assumption that the action is semifree and the dimension of $M$ is six.
				See Proposition \ref{proposition_topology_reduced_space}.} is the reduced symplectic manifold at level $c$.
			\item $k_c$ is the number of fixed components at level $c$. 
			\item Each $Z_c^i$ is a connected fixed component and hence is a symplectic submanifold of $(M_c, \omega_c)$ via the embedding
				\[
					Z_c^i \hookrightarrow H^{-1}(c) \rightarrow H^{-1}(c) / S^1 = M_c.
				\]
				(This information contains a normal bundle of $Z_c^i$ in $M_c$.)
			\item The Euler class $e(P_c^{\pm})$ of principal $S^1$-bundles $H^{-1}(c \pm \epsilon) \rightarrow M_{c \pm \epsilon}$.
		\end{itemize}		
	\end{definition}
	
	Gonzalez proved that the fixed point data determines $(M,\omega)$ uniquely under the assumption that every reduced symplectic form is {\em symplectically rigid}. 
	Following \cite[Definition 2.13]{McD2} or \cite[Definition 1.4]{G}, a manifold $B$ is said to be {\em symplectically rigid} if 
	\begin{itemize}
		\item (uniqueness) any two cohomologous symplectic forms are diffeomorphic, 
		\item (deformation implies isotopy) every path $\omega_t$ ($t \in [0,1]$) of symplectic forms such that $[\omega_0] = [\omega_1]$ can be homotoped through families of symplectic forms 
		with the fixed endpoints $\omega_0$ and $\omega_1$ to an isotopy, that is, a path $\omega_t'$ such that $[\omega_t']$ is constant in $H^2(B)$. 
		\item For every symplectic form $\omega$ on $B$, the group $\text{Symp}(B,\omega)$ of symplectomorphisms that act trivially on $H_*(B;\Z)$ is path-connected.
	\end{itemize}
	Using this terminology, together with Definition \ref{definition_fixed_point_data}, Gonzalez proved the following. 
	
	\begin{theorem}\cite[Theorem 1.5]{G}\label{theorem_Gonzalez_5}
		Let $(M,\omega)$ be a six-dimensional closed semifree Hamiltonian $S^1$-manifold such that every critical level is simple. 
		Suppose further that every reduced space is symplectically rigid.
		Then $(M,\omega)$ is determined by its fixed point data up to $S^1$-equivariant symplectomorphism.
	\end{theorem}
	
	\begin{remark}\label{remark_Gonzalez_5}
		 Note that Theorem \ref{theorem_Gonzalez_5} is a six-dimensional version of the original statement of the Gonzalez Theorem \cite[Theorem 1.5]{G}
		 so that we may drop ``(co)-index two'' condition in his original statement because every non-extremal fixed component has index two or co-index two in a six-dimensional case.
	 \end{remark}

	Now, we introduce the notion ``{\em topological fixed point data}'', which is a topological analogue of a fixed point data, as follows. 
				
	\begin{definition}\label{definition_topological_fixed_point_data}
		Let $(M,\omega)$ be a six-dimensional closed semifree Hamiltonian $S^1$-manifold equipped with a moment map $H : M \rightarrow I$ such that all critical level sets are simple.
		A {\em topological fixed point data} of $(M,\omega, H)$, denoted by $\frak{F}_{\text{top}}(M, \omega, H)$, is defined as a collection 
		\[
			 \frak{F}_{\text{top}}(M, \omega, H) := \left\{(M_{c}, [\omega_c], [Z_c^1], [Z_c^2], \cdots, [Z_c^{k_c}], e(P_c^{\pm}) ) ~|~c \in \mathrm{Crit} ~H \right\}
		\]
		where 
		\begin{itemize}
			\item $(M_c, \omega_c)$ is the reduced symplectic manifold at level $c$, 
			\item $k_c$ is the number of fixed components at level $c$, 
			\item each $Z_c^i$ is a connected fixed component lying on the level $c$ and $[Z_c^i] \in H^*(M_c)$ denotes the Poincar\'{e} dual class of the image of the embedding
				\[
					Z_c^i \hookrightarrow H^{-1}(c) \rightarrow H^{-1}(c) / S^1 = M_c.
				\]
			\item the Euler class $e(P_c^{\pm})$ of principal $S^1$-bundles $H^{-1}(c \pm \epsilon) \rightarrow M_{e \pm \epsilon}$.
		\end{itemize}		
	\end{definition}
	
	The following lemma allows us 
	to compute the data $e(P_c^{\pm})$ in terms of $\frak{F}_{\mathrm{top}}$ under the assumption 
	that the minimal fixed component is an isolated point. 
	
	\begin{lemma}\label{lemma_Euler_condition}
		If $(M,\omega)$ is a six-dimensional closed monotone semifree Hamiltonian $S^1$-manifold with isolated minimum, then 
		the Euler classes $\{e(P_c^{\pm}) ~|~ c \in \mathrm{Crit} ~H\}$ is completely determined by other topological fixed point data
		\[
			\left\{(M_{c}, [\omega_c], [Z_c^1], [Z_c^2], \cdots, [Z_c^{k_c}] ) ~|~c \in \mathrm{Crit} ~H \right\}.
		\] 
		So, we may omit them in both $\frak{F}$ and $\frak{F}_{\mathrm{top}}$. 
	\end{lemma}
	
	\begin{proof}
		If the minimal fixed component is isolated, then $e(P_{-3}^+) = -u \in H^2(M_{-3 + \epsilon} ; \Z) \cong H^2(\p^2 ; \Z)$. (See the second paragraph of Section \ref{secCaseIDimZMax}.)
		Then the lemma follows from Proposition \ref{proposition_topology_reduced_space} and Lemma \ref{lemma_Euler_class}.
	\end{proof}
	
	Our aim is to classify all such manifolds up to $S^1$-equivariant symplectomorphism, in addition, to show that each manifold is indeed algebraic Fano. 
	(See Theorem \ref{theorem_main}.)
	The rest of this paper consists of two parts : 
	\begin{itemize}
		\item {\bf First part  : classification of all topological fixed point data.} Through Section \ref{secCaseIDimZMax}, \ref{secCaseIIDimZMax2}, and \ref{secCaseIIIDimZMax4}, 
		we give a complete list of possible topological fixed point data that $(M,\omega)$ might have. We also show that there exists a smooth Fano 3-fold (in the Mori-Mukai list \cite{MM})
		with a semifree holomorphic $\C^*$-action having a (any) given topological fixed point data in our list. 
		\item {\bf Second part : uniqueness.} Based on our classification result, we will show that a topological fixed point data determines a fixed point data uniquely. Moreover, 
		all conditions in Theorem \ref{theorem_Gonzalez} are satisfied, and hence a topological fixed point data determines a manifold uniquely. Consequently, every $(M,\omega)$ 
		is $S^1$-equivariantly symplectomorphic to one of smooth Fano 3-folds described in the first part. 
	\end{itemize}

	We finalize this section with the following lemma which shows that a possible topological type of a fixed component is very restrictive. 
	We denote by $Z_{\min}$ and $Z_{\max}$ the minimal and the maximal fixed components of the action, respectively.

	\begin{lemma}\label{lemma_possible_critical_values} 
		All possible critical values of $H$ are $\pm 3, \pm 2, \pm 1$, and $0$. Moreover, any connected component $Z$ of $M^{S^1}$ satisfies one of the followings : 
		\begin{table}[H]
			\begin{tabular}{|c|c|c|c|}
			\hline
			    $H(Z)$ & $\dim Z$ & $\mathrm{ind}(Z)$ & $\mathrm{Remark}$ \\ \hline 
			    $3$ &  $0$ & $6$ & $Z = Z_{\max} = \mathrm{point}$ \\ \hline
			    $2$ &  $2$ & $4$ & $Z = Z_{\max} \cong S^2$ \\ \hline
			    $1$ &  $4$ & $2$ & $Z = Z_{\max}$ \\ \hline
			    $1$ &  $0$ & $4$ & $Z = \mathrm{pt}$ \\ \hline
			    $0$ &  $2$ & $2$ & \\ \hline
			    $-1$ &  $0$ & $2$ & $Z = \mathrm{pt}$ \\ \hline
			    $-1$ &  $4$ & $0$ & $Z = Z_{\min}$ \\ \hline
			    $-2$ &  $2$ & $0$ & $Z = Z_{\min} \cong S^2$ \\ \hline
			    $-3$ &  $0$ & $0$ & $Z = Z_{\min} = \mathrm{point}$ \\ \hline
			\end{tabular}
			\vspace{0.2cm}
			\caption{\label{table_fixed} List of possible fixed components}
		\end{table}

	\end{lemma}

	\begin{proof}
		Let $z$ be any point in the fixed component $Z \subset M^{S^1}$.
		Since the action is semifree, every weight of the $S^1$-representation on $T_z M$ is either $0$ or $\pm 1$. Thus all possible (unordered) weights at $z$ are 
		$(\pm 1, \pm1, \pm1)$, $(\pm 1, \pm1, 0)$, and $(\pm 1, 0, 0)$. Thus the first statement follows from Corollary \ref{corollary_sum_weights_moment_value}.
		
		For the second statement, it is enough to consider the case where $H(Z) \geq 0$ due to the symmetry of the table \ref{table_fixed}.
		Note that the zero-weight subspace of $T_z M$ is exactly the tangent space $T_z Z$ whose dimension equals the twice the multiplicity of the zero weight on $T_z M$.
		If $H(Z) =3$, then the weights at $z$ is $(-1, -1, -1)$. Thus $\dim Z = 0$ (i.e., $Z = z$.) 
		Moreover, since twice the number of negative weights at $Z$ is equal to the Morse index of $Z$ 
		by Corollary \ref{corollary_properties_moment_map}, we have $\mathrm{ind}(Z) = 6$.
		Therefore, $H(Z)$ is the maximum value of $H$. 
		 
		We can complete the table \ref{table_fixed} in a similar way. The only non-trivial part of the lemma is that $Z \cong S^2$ when $H(Z) = 2$ (and hence $\dim Z = 2$.) 
		To show this, recall that $M_0$ is a monotone symplectic 4 manifold, diffeomorphic to a del Pezzo surface, by Proposition 
		\ref{proposition_monotonicity_preserved_under_reduction}. Since any del Pezzo surface simply connected, we have 
		\[
			\pi_1(M) \cong \pi_1(M_0) \cong \pi_1(Z_{\max}) \cong \{ 0 \}
		\] 
		by the Theorem \cite[Theorem 0.1]{Li1} of Li. Therefore we have $Z_{\max} = S^2$. 
	\end{proof}

	\begin{notation}\label{notation} From now on, we use the following notation. Let $c$ be a critical value of $H$. 
	\begin{itemize}
		\item $\mathrm{Crit}~H$ : set of critical values of $H$.
		\item $\mathrm{Crit}~ \mathring{H}$ : set of non-extremal critical values of $H$.
		\item $P_c^{\pm}$  : the principal bundle $\pi_{c \pm \epsilon} : H^{-1}(c \pm \epsilon) \rightarrow M_{c \pm \epsilon}$ where $\epsilon > 0$ is sufficiently small.
		\item $Z_c$ : fixed point set lying on the level set $H^{-1}(c)$. That is, $Z_c = M^{S^1} \cap H^{-1}(c)$.
		\item $\R[x]$ : the cohomology ring of $H^*(BS^1;\R)$, where $-x$ is the Euler class of the universal Hopf bundle $ES^1 \rightarrow BS^1.$
		\item $\R[u] / \langle u^3 \rangle$ : the cohomology ring of $H^*(\C P^2 ; \R)$ where $u$ is the Poincar\'{e} dual to a line.
		\item $P_M(t)$ : the Poincar\'{e} polynomial of $M$.
	\end{itemize}
	\end{notation}

\section{Case I : $\dim Z_{\max} = 0$}
\label{secCaseIDimZMax}

Let $(M,\omega)$ be a six-dimensional closed monotone semifree Hamiltonian $S^1$-manifold with the balanced moment map $H$ where $H(Z_{\min}) = -3$ and $H(Z_{\max}) = 3$.
In this section, we classify all possible topological fixed point data of $(M,\omega)$ as well as we provide algebraic Fano examples for each cases. 

By Lemma \ref{lemma_possible_critical_values}, the only possible non-extremal critical values are $\{\pm1, 0\}$, i.e., $\mathrm{Crit}~ \mathring{H} \subseteq \{ 0, \pm 1\}$,  
and each non-extremal fixed component $Z$ satisfies either 
\[
	\begin{cases}
		\text{$Z$ = pt} \hspace{1cm} \text{if $H(Z) = \pm 1$, \quad or} \\
		\text{$\dim Z = 2$} \quad \text{if $H(Z) = 0$.}
	\end{cases}
\]
Moreover, since $H$ is a perfect Morse-Bott function, we can easily see that 
\[
	|Z_1| = |Z_{-1}|
\]
by the Poincar\'{e} duality.

By the equivariant Darboux theorem, the $S^1$-action near the minimum (maximum, resp.) is locally identified with the standard semifree $S^1$-action 
\[
	 t \cdot (z_1, z_2, z_3) = (tz_1, tz_2, tz_3) \quad \left(\text{$(t^{-1}z_1, t^{-1}z_2, t^{-1}z_3)$, resp.}\right)
\]
on $(\C^3, \omega_{\mathrm{std}})$. Also, the balanced moment map is written by
\[
	H(z_1, z_2, z_3) = \frac{1}{2}|z_1|^2 + \frac{1}{2}|z_2|^2 + \frac{1}{2}|z_3|^2 - 3.
\] 
Therefore, the level set $H^{-1}(-3 + \epsilon)$ near the minimum is homeomorphic to $S^5$ and hence the reduced space is
\[
	M_{-3 + \epsilon} = H^{-1}(-3 + \epsilon) / S^1 \cong S^5 / S^1 \cong \p^2.
\]
Note that the Euler class $e(P_{-3}^+)$ of the principal $S^1$-bundle $H^{-1}(-3+\epsilon) \rightarrow M_{-3 + \epsilon} \cong \p^2$ is $-u$ where $u$ is the generator of $H^*(\p^2 ; \Z)$.
Similarly, for the maximal fixed point $Z_{\max}$, we can apply the previous argument to show that the reduced space near the maximum is $M_{3 - \epsilon} \cong \p^2$ 
and the Euler class of $H^{-1}(3 - \epsilon) \rightarrow M_{3-\epsilon}$ is given by $e(P_3^{-}) = u$. \\

\subsection{${\mathrm{Crit} ~\mathring{H}} = \{0\}$}
\label{sssecMathrmCritMathringH}
~ \\

In this case, the reduced space $M_0$ is $\p^2$ with the reduced symplectic form $\omega_0$ where $[\omega_0] = c_1(TM_0) = 3u$. 

\begin{lemma}\label{lemma_connectedness_1_1}
	$Z_0$ is connected.
\end{lemma}

\begin{proof}
	Recall  that $Z_0$ can be regarded as a symplectic submanifold of the reduced space $(M_0, \omega_0)$. If $Z_0$ is the disjoint union of two disjoint set $Z_0^1$ and $Z_0^2$, 
	then $[Z_0^1] \cdot [Z_0^2] = 0$. On the other hand, if we let $\mathrm{PD} (Z_0^1) = au$ and $\mathrm{PD}(Z_0^2) = bu$, then 
	\[
		\int_{Z_0^1} \omega_0 = 3a > 0, \quad \text{and} \quad \quad \int_{Z_0^2} \omega_0 = 3b > 0
	\]
	In particular both $a$ and $b$ are non-zero, and therefore $[Z_0^1] \cdot [Z_0^2] = ab \neq 0$ which leads to a contradiction.
\end{proof}

\begin{lemma}\label{lemma_Z0_1_1}
	$\mathrm{PD}(Z_0) = 2u$ and $Z_0 \cong S^2$ 
\end{lemma}

\begin{proof}
	It is straightforward from Lemma \ref{lemma_Euler_class} that
	\[
		e(P_{0}^+) = e(P_{0}^-) + \mathrm{PD}(Z_0)
	\]
	where $e(P_{0}^+)  = e(P_{3 - \epsilon}^+) = u$ and $e(P_{0}^-)  = e(P_{-3 + \epsilon}^+) = -u$. On the other hand, 
	the adjunction formula\footnote{Any embedded symplectic surface $\Sigma$ in a closed symplectic four manifold $(X,\omega)$ can be made into an image of some embedded
	$J$-holomorphic curve for some $\omega$-compatible almost complex structure. Therefore, we may apply the adjunction formula to $(M, \Sigma, J)$.}  for the symplectic surface 
	$Z_0$ gives
	\[
		[Z_0] \cdot [Z_0] + 2 - 2g = \langle c_1(TM_0), Z_0 \rangle, \quad \quad \text{$g$ : genus of $Z_0$}
	\]	
	So, we get $g = 0$.
\end{proof}

Summing up, we have the following. 

\begin{theorem}\label{theorem_1_1}
	Let $(M,\omega)$ be a six-dimensional closed monotone semifree Hamiltonian $S^1$-manifold such that $\mathrm{Crit} H = \{ 3, 0, -3\}$. Then the only possible 
	topological fixed point data is given by 	
		\begin{table}[H]
			\begin{tabular}{|c|c|c|c|c|c|}
				\hline
				    & $(M_0, [\omega_0])$ & $Z_{-3}$ &  $Z_0$ & $Z_3$ \\ \hline \hline
				    {\bf (I-1)} & $(\p^2, 3u)$ & $\mathrm{pt}$ &  $Z_0 \cong S^2$, $[Z_0] = 2u$ & $\mathrm{pt}$\\ \hline    
			\end{tabular}
		\end{table}
	\noindent In particular, we have $b_2(M) = 1$ and $b_{\mathrm{odd}}(M) = 0$.
\end{theorem}

\begin{example}[Fano variety of type {\bf (I-1)}]\cite[16th in the list in p. 215]{IP}\label{example_1_1} 
    Let $Q$ be a smooth quadric in $\p^4$, also known as a co-adjoint orbit of $\mathrm{SO}(5)$, is an example of algebraic Fano 3-fold with $b_2 = 1$.
    With respect to the $\mathrm{SO}(5)$-invariant K\"{a}hler form $\omega$ with $[\omega] = c_1(TQ)$, 
    the diagonal maximal torus $T$ of $\mathrm{SO}(5)$ acts on $Q$ in a Hamiltonian fashion and its moment map image in the dual Lie algebra $\frak{t}^*$ of $T$
    is described as follows. (See \cite{Li3}, \cite{McD2}, or \cite{Tol} for more details.)
    \begin{figure}[H]
		\scalebox{1}{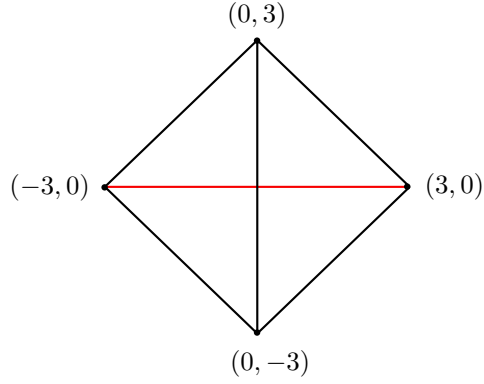}
		\caption{\label{figure_1_1} Moment map image of $Q$}
    \end{figure}
    \noindent 
    In this figure, each vertex (on the boundary of the image) corresponds to a fixed point and 
    each edge indicates an image of an invariant 2-sphere (called a 1-skeleton in \cite{GKM}). If we take a circle subgroup generated by $(0,1) \in \frak{t}$, then 
    the fixed point set is given by $\{ Z_{-3} = \mathrm{pt}, Z_0 \cong S^2, Z_3 = \mathrm{pt} \}$ where the image of $Z_0$ is colored by  red.
\end{example}
\vs{0.1cm}

\subsection{${\mathrm{Crit} \mathring{H}} = \{-1,1\}$}
\label{ssecMathrmCritMathringH11}
~\\

In this case, all fixed point in $Z_{-1}$ and $Z_1$ are isolated (see Table \ref{table_fixed}) and their Morse indices are two and four, respectively, so that 
the Poincar\'{e} polynomial $P_M$ of $M$ is given by 
\[
	P_M(t) = \sum b_i(M) t^i = 1 + |Z_{-1}|t^2 + |Z_{1}|t^4 + t^6.
\]
Let  $k = |Z_{-1}| = |Z_1| \in \Z_+$. For a sufficiently small $\epsilon >0$, the reduced space $M_{-1 + \epsilon}$
is diffeomorphic to the blow-up of $M_{-1-\epsilon} \cong M_{-3+\epsilon} \cong \p^2$ at $k$ generic points by Proposition \ref{proposition_topology_reduced_space}. 
Denote each classes of the exceptional divisors by 
\[
	E_1, \cdots, E_k \in H^2(M_{-1+\epsilon}; \Z).
\]
Then, since $e(P_{-1}^-) = -u$, Lemma \ref{lemma_Euler_class} implies
\[
	e(P_{-1}^+) = -u + E_1 + \cdots + E_k.
\]
Note that $M_0 \cong M_t$ for every $t \in (-1, 1)$ and therefore each reduced space $M_t$ can be identified with $M_0$ with the reduced symplectic form $\omega_t$
where $[\omega_t] \in H^2(M_0; \R)$. We also note that the set $\{u, E_1, \cdots, E_k\} \subset H^2(M_0; \Z)$ is an integral basis of $H^2(M_0; \Z)$ satisfying
\begin{equation}\label{equation_basis}
	\int_{M_0} u^2 = 1, \quad \int_{M_0} {E_i}^2 = -1, \quad \int_{M_0} u \cdot E_i = 0, \quad \text{and $\int_{M_0} E_i \cdot E_j = 0$}
\end{equation}
for every $0 \leq i,j \leq k$ with $i \neq j$. 

Now, let us compute the symplectic area of $(M_1, \omega_1)$. Using the Duistermaat-Heckman theorem \ref{theorem_DH}, we get 
\[
	\lim_{t \rightarrow 1^-} \int_{M_t} [\omega_t]^2 = \int_{M_0} \left([\omega_0] - e(P_{-1}^+) \right)^2 = \int_{M_0} \left( 4u - 2E_1 - \cdots - 2E_k \right)^2 = 16 - 4k.
\]	
On the other hand, the right limit of the symplectic area of $(M_t, \omega_t)$ at $t = 1$ is given by 
\[
	\lim_{t \rightarrow 1^+} \int_{M_t} [\omega_t]^2 = \lim_{t \rightarrow 1^+} \int_{\p^2} \left((3-t)u\right)^2 = \int_{\p^2} (2u)^2 = 4. 
\]
Since the Duistermaat-Heckman function is continuous, we have $k=3$ and therefore $\mathfrak{F}_{\mathrm{top}}(M)$ is given as follows.

\begin{theorem}\label{theorem_1_2}
	Let $(M,\omega)$ be a six-dimensional closed monotone semifree Hamiltonian $S^1$-manifold such that $\mathrm{Crit} H = \{ 3, 1, -1, -3\}$. Then the only possible 
	topological fixed point data is given by 	
		\begin{table}[H]
			\begin{tabular}{|c|c|c|c|c|c|c|}
				\hline
				   & $(M_0, [\omega_0])$ & $Z_{-3}$ & $Z_{-1}$ & $Z_1$ & $Z_3$ \\ \hline \hline
				   {\bf (I-2)} & $(\p^2 \# 3 \overline{\p^2}, 3u - E_1 - E_2 - E_3)$ & {\em pt} & {\em 3 ~pts} & {\em 3 ~pts} & {\em pt} \\ \hline    
			\end{tabular}
		\end{table}
	\noindent In particular, we have $b_2(M) = 3$ and $b_{\mathrm{odd}}(M) = 0$.
\end{theorem}

\begin{example}[Fano variety of type {\bf (I-2)}]\label{example_1_2}\cite[No. 27 in the list in Section 12.4]{IP}\label{example_1_1}  	
	We denote by $\omega_{\mathrm{FS}}$ the Fubini-Study form on $\p^1$ which is {\em normalized}, i.e., $\int_{\p^1} \omega_{\mathrm{FS}} = ~1.$
	Consider $(M,\omega) = (\p^1 \times \p^1 \times \p^1, 2\omega_{\mathrm{FS}} \oplus 2\omega_{\mathrm{FS}} \oplus 2\omega_{\mathrm{FS}})$ with the standard 
	Hamiltonian $T^3$-action 
	\[
    		(t_1, t_2, t_3) \cdot ([x_1, y_1], [x_2, y_2], [x_3, y_3]) = ([t_1x_1, y_1], [t_2x_2, y_2] , [t_3x_3, y_3]), \quad (t_1, t_2, t_3) \in T^3
	\] 
	with a moment map $H : M \rightarrow \frak{t}^* \cong \R^3$ given by 
	\[
		H(([x_1, y_1], [x_2, y_2], [x_3, y_3])) = \left( \frac{2|x_1|^2}{|x_1|^2+|y_1|^2}, \frac{2|x_2|^2}{|x_2|^2+|y_2|^2}, \frac{2|x_3|^2}{|x_3|^2+|y_3|^2} \right) - (1,1,1)
	\]
	so that the image of the moment map is pictorially described as follows. \vs{0.1cm}
	
	\begin{figure}[H]
		\scalebox{1}{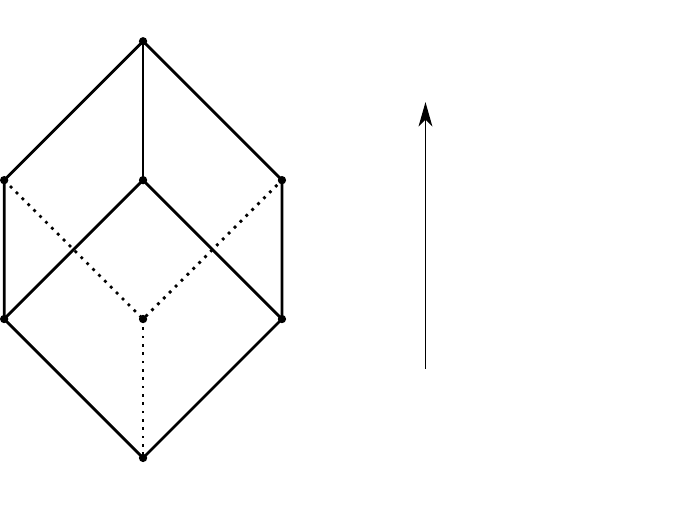}
		\caption{Moment graph of $\p^1 \times \p^1 \times \p^1$}
	\end{figure}
	\noindent 
	The diagonal subgroup $S^1$ of $T^3$ is generated by $\xi = (1,1,1)$ and the induced $S^1$-action has 
	the associated balanced moment map is given by $\mu = \langle H , \xi \rangle$.
	Then the $S^1$-action has the same topological data as in Theorem \ref{theorem_1_2}.
\end{example}
\vs{0.1cm}

\subsection{${\mathrm{Crit} \mathring{H}} = \{-1,0,1\}$}
\label{ssecMathrmCritMathringH11}
~\\

Suppose that $|Z_{-1}| = |Z_1| = k \geq 1$ and that $Z_0$ has $r$ connected components. The Poincar\'{e} polynomial of $M$ is given by 
\[
	\begin{array}{ccl}
		P_M(t) & = & 1 + |Z_{-1}|t^2 + \left(P_{Z_0}(t) \right)t^2 + |Z_{1}|t^4 + t^6 \\
				& = & 1 + |Z_{-1}|t^2 + (r + st + rt^2) t^2 + |Z_{1}|t^4 + t^6 \\
				& = & 1 + (k+ r) t^2 + st^3 + (k + r) t^4 + t^6
	\end{array}
\]
where $s$ is the rank of $H^1(Z_0; \Z)$. 
In this case, the reduced space $M_{-1+\epsilon}$ is 
a $k$-times blow-up of $M_{-1-\epsilon} \cong \p^2$ with the exceptional classes $E_1, \cdots, E_k \in H^2(M_{-1+\epsilon}; \Z)$.

\begin{lemma}\label{lemma_1_3_k}
	Following the above notation, we have $k = 1$ and $\mathrm{Vol}(Z_0) = 4$. 
\end{lemma}

\begin{proof}
	First, we apply Theorem \ref{theorem_localization} to the equivariant first Chern class $c_1^{S^1}(TM)$ : 
	\begin{equation}\label{equation_k_1_vol_4}
		\begin{array}{ccl}\vs{0.1cm}
				   		0 & = & \ds \int_M c_1^{S^1}(TM) \\ \vs{0.1cm}
							& = & \ds \sum_{Z \subset M^{S^1}} \int_Z \frac{c_1^{S^1}(TM)|_Z}{e_Z^{S^1}} \\ \vs{0.1cm}
							& = & \ds  \frac{3x}{x^3} - \frac{kx}{x^3} + 
							\sum_{Z \subset Z_0} \int_{Z} \frac{c_1^{S^1}(TM)|_{Z}}{e_{Z}^{S^1}} - \frac{kx}{x^3} + \frac{3x}{x^3}. \\
		\end{array}
	\end{equation}
	Moreover, it follows from Corollary \ref{corollary_sum_weights_moment_value} and 
	Proposition \ref{proposition_monotonicity_preserved_under_reduction} that 
	\[
		\begin{cases}
			c_1^{S^1}(TM)|_{Z} = c_1(TM)|_{Z} + x \cdot \underbrace{\sum (Z)}_{\text{sum of weights} = 0}  = c_1(TM)|_{Z} \quad \text{and} \\
			[\omega_0] = c_1(TM_0).  
		\end{cases}
	\]
	
	Let $q$ be the positive generator of $H^2(Z; \Z)$ (so that $q^2 = 0$). Since the action is semifree, 
	the equivariant first Chern classes of the positive and negative normal bundle of $Z$ in $M$ can be written by $x + mq$ and $-x + nq$ for some $m, n \in \Z$, respectively.
	Thus 
	\[
		\begin{array}{ccl}\vs{0.3cm}
			\ds \int_{Z} \frac{c_1^{S^1}(TM)|_{Z}}{e_{Z}^{S^1}} & = & \ds \int_{Z} \frac{c_1(TM)|_{Z}}{(x + mq) (-x + nq)} 
														= \int_{Z} \frac{c_1(TM)|_{Z}}{-x^2 + (n - m)xq}  \\ \vs{0.3cm}
												& = & \ds \int_{Z} \frac{c_1(TM)|_{Z}\cdot (x - (m-n)q)}{-x(x + (m-n)q)(x - (m-n)q)}  \\ \vs{0.3cm}
												& = & \ds \int_{Z} \frac{c_1(TM)|_{Z}\cdot x}{-x^3}  \\ \vs{0.3cm}
												& = &\ds -\frac{\langle c_1(TM), [Z] \rangle x}{x^3} \\ \vs{0.3cm}
												& = & \ds -\frac{\mathrm{Vol}(Z)x}{x^3} ~(\text{since $c_1(TM) = [\omega]$}).
		\end{array}
	\]
	From \eqref{equation_k_1_vol_4}, we get $6- 2k - \mathrm{Vol}(Z_0) = 0$ so that there are only two possibilities
	\[
		\left(k, \mathrm{Vol}(Z_0) \right) = \begin{cases} (1,4), \hs{0.2cm} \text{or} \\ (2,2) \end{cases}
	\]
	
	It remains to show that $(k, \mathrm{Vol}(Z_0)) \neq (2,2)$.
	Suppose that $k=2$. Then $M_0 \cong \p \# 2 \overline{\p}^2$ with two exceptional classes $E_1, E_2 \in H^2(M_0; \Z)$. 
	Let $\mathrm{PD}(Z_0) = au + bE_1 + cE_2 \in H^2(M_0; \Z)$. 
	Then Lemma \ref{lemma_Euler_class} implies that
	\[
		e(P_0^+) = e(P_0^-) + \mathrm{PD}(Z_0) = (a-1)u + (b+1)E_1 + (c+1)E_2.
	\]
	Note that $c_1(TM_0) = [\omega_0] = 3u - E_1 - E_2$ and the Duistermaat-Heckman theorem \ref{theorem_DH} yells that
	\[
		[\omega_t] = [\omega_0] - t e(P_0^+) = (3 - t(a-1))u + (-1 - t(b+1))E_1 + (-1 - t(c+1))E_2, \quad \quad t \in [0,1).
	\]
	By Proposition \ref{proposition_topology_reduced_space}, two symplectic blow-downs occur simultaneously on $M_1$.
	We denote by $C_1$ and $C_2$ the corresponding two exceptional divisors on $M_0$. 	
	Since the only possible exceptional classes in $H^2(M_0; \Z)$ is $E_1, E_2$, and $u-E_1-E_2$ by Lemma \ref{lemma_list_exceptional} 
	and $C_1$ and $C_2$ are disjoint, we have $\mathrm{PD}(C_1) = E_1$ and $\mathrm{PD}(C_2) = E_2$.
	As the symplectic areas of $C_1$ and $C_2$ go to zero as $t \rightarrow 1$, we get 
	\[
		\langle [\omega_1], [C_1] \rangle = -2-b = 0, ~~\langle [\omega_1], [C_2] \rangle = -2-c = 0, 
	\]
	i.e., $b = c = -2$. 
	
	To compute $a$, consider a symplectic volume of $M_1$. By the Duistermaat-Heckman theorem \ref{theorem_DH}, we have 
	\[
		\lim_{t\rightarrow 1^+} \int_{M_t} [\omega_t]^2 = \int_{\p^2} (2u)^2 = 4, \hs{0.5cm}
				\lim_{t\rightarrow 1^-} \int_{M_t} [\omega_t]^2 = \int_{M_0} (4-a)^2u^2 = (4-a)^2.
	\]
	Thus we obtain $4 = (4-a)^2$ and hence $a=2, 6$. However, if $a=6$, then the symplectic area of $M_t$ is given by
	\[
		\int_{M_t} [\omega_t]^2 = \int_{M_0} \left( (3-5t)^2 u^2 + (t - 1)^2 E_1^2 + (t - 1)^2 E_2^2 \right) = (5t - 3)^2 - 2(t-1)^2, \quad t \in (0,1)
	\]
	which is negative for some $t$ (e.g. $t = \frac{3}{5}$). Thus we get $a=2$ and it follows that 
	$\mathrm{PD}(Z_0) = 2u - 2E_1 - 2E_2$ so that the symplectic area of $Z_0$ is given by $\langle [\omega_0], [Z_0] \rangle = \langle c_1(TM_0), [Z_0] \rangle = 2$. 
	
	Consequently, the number of connected components of $Z_0$ is at most two (since the symplectic area of each component should be a positive integer.)
	On the other hand, the adjunction formula 
	\begin{equation}\label{equation_adjunction}
		[Z_0] \cdot [Z_0] + \sum_i (2 - 2g_i)  = \langle c_1(TM_0), [Z_0] \rangle
	\end{equation}
	implies that $2 = -4 + \sum_{i} (2-2g_i)$ where the sum is taken over all connected components of $Z_0$ and $g_i$ denotes the genus of the component of $Z_0$ index by $i$.
	This equality implies that $Z_0$ should contain at least three spheres which contradicts that the number of component of $Z_0$ is at most two.
	This finishes the proof.
\end{proof}

\begin{theorem}\label{theorem_1_3}
	Let $(M,\omega)$ be a six-dimensional closed monotone semifree Hamiltonian $S^1$-manifold such that $\mathrm{Crit} H = \{ 3, 1, 0, -1, -3\}$. 
	Then the topological fixed point data is given by 	
		\begin{table}[H]
			\begin{tabular}{|c|c|c|c|c|c|c|c|}
				\hline
				    & $(M_0, [\omega_0])$ & $Z_{-3}$ & $Z_{-1}$ & $Z_0$ & $Z_1$ & $Z_{3}$\\ \hline \hline
				    {\bf (I-3)} & $(\p^2 \# \overline{\p^2}, 3u - E_1)$ & {\em pt} & {\em pt} &  \makecell{ $Z_0 = Z_0^1 ~\dot \cup ~ Z_0^2$ \\
				    $Z_0^1 \cong Z_0^2 \cong S^2$ \\ $[Z_0^1] = [Z_0^2] = u - E_1$} & {\em pt} & {\em pt} \\ \hline    
			\end{tabular}
		\end{table}
	\noindent In particular, we have $b_2(M) = 3$, $b_{\mathrm{odd}}(M) = 0$, and $\langle c_1(TM)^3, [M] \rangle = 52$.
\end{theorem}

\begin{proof}
	Since $|Z_{-1}| 1$ by Lemma \ref{lemma_1_3_k}, we have $M_{-1 + \epsilon} \cong M_0 \cong \p^2 \# \overline{\p}^2$. 
	Let $\mathrm{PD}(Z_0) = au + bE_1 \in H^2(M_0)$. Since $e(P_{-1}^-) = -u$, $e(P_{-1}^+) = -u + E_1$, and
    	\begin{equation}\label{equation_ep0}
    	   	e(P_0^+) = (a-1)u + (b+1)E_1, \quad \quad \text{(by Lemma \ref{lemma_Euler_class})}
	\end{equation}
	by the Duistermaat-Heckman theorem \ref{theorem_DH}, we get	   
	\[
		[\omega_t] = [\omega_0] - e(P_0^+) t = (3 - t(a-1))u + (-1 - t(b+1))E_1 \quad \text{for} \quad t \in [0,1)
	\]
	where $[\omega_0] = c_1(TM_0) = 3u - E_1$. 

	 On the level $t=1$, the symplectic blow-down occurs by Proposition \ref{proposition_topology_reduced_space}. We denote by $C$ the corresponding divisor where
	$\mathrm{PD}(C) = E_1$ by Lemma \ref{lemma_list_exceptional}. 
	Since the symplectic area of $C$ goes to zero as $t \rightarrow 1$, we get 
	\[
		0 = \langle [\omega_1], E_1 \rangle = -2-b \quad \Rightarrow \quad b = -2.
	\]

	To compute $a$, consider the equation 
           \[
        		(a-1)^2 - (b+1)^2 = \langle e(P_0^+)^2, [M_0] \rangle = \langle e(P_1^-)^2, [M_0] \rangle = \int_{M_0} (u-E_1)^2 = 0.
	\]
	where the first equality comes from \eqref{equation_ep0} and the last inequality is obtained from the fact that 
	\[
		e(P_1^-) = e(P_1^+) - E_1 = e(P_3^-) - E_1 = u - E_1 \quad \text{(by Lemma \ref{lemma_Euler_class})}
	\]
	This induces $(a-1)^2 - (b+1)^2 = (a-1)^2 - 1 = 0$ (since $b = -2$) 
	so that $a=0$ or $2$. 
	Moreover, since 
	\[
		\langle [\omega_0], [Z_0] \rangle = \langle 3u-E_1, [Z_0] \rangle = 3a + b = 3a - 2 > 0,
	\] we have $a=2$ and therefore $\mathrm{PD}(Z_0) = 2u - 2E_1$.

	Now, we apply the adjunction formula to $Z_0 = \sqcup Z_0^i$. Then 
	\[
		\underbrace{\langle c_1(TM_0) , [Z_0] \rangle}_{ = \langle [\omega_0], [Z_0] \rangle}
		 = \underbrace{\int_{M_0} (3u - E_1) \cdot (2u - 2E_1)}_{ = 4}  = \underbrace{\int_{M_0}(2u - 2E_1)^2}_{ = [Z_0] \cdot [Z_0] = 0} + 
		\sum_{i}\underbrace{(2-2g_i)}_{\langle c_1(TZ_0^i), [Z_0^i] \rangle} = \sum_{i}(2-2g_i), 
	\]
	where $g_i$ is the genus of $Z_0^i$. By direct computation, we may check that each $\mathrm{Z_0^i}$ is of the form $pu - pE_1$ for some $p \in \Z$ (since $Z_0^i$'s are disjoint)
	and we see that $Z_0$ is the disjoint union of two spheres $Z_0^1$ and $Z_0^2$ with 
	$\mathrm{PD}(Z_0^1) = \mathrm{PD}(Z_0^2) = u - E_1$. Using the perfectness of the moment map, it is straightforward that $b_2(M) = 3$ and $b_{\mathrm{odd}}(M) = 0$.

	For a computation of the Chern number, it follows from the localization theorem \ref{theorem_localization} that
	\begin{equation}\label{equation_CN_1_3}
		\begin{array}{ccl}\vs{0.1cm}
			\ds \int_M c_1^{S^1}(TM)^3 & = &  \ds \sum_{Z \subset M^{S^1}} \int_Z \frac{c_1^{S^1}(TM)|_Z}{e_Z^{S^1}} \\ \vs{0.1cm}
							& = & \ds  \frac{(3x)^3}{x^3} - \frac{(x)^3}{x^3} + 
							\sum_{Z \subset Z_0} \int_{Z} \frac{\left(c_1^{S^1}(TM)|_{Z}\right)^3}{e_{Z}^{S^1}} + \frac{(-x)^3}{x^3} - \frac{(-3x)^3}{x^3}. \\ \vs{0.1cm}	
		\end{array}
	\end{equation}
	Since $c_1^{S^1}(TM)|_Z = c_1(TM)|_Z$ by Proposition \ref{proposition_equivariant_Chern_class}, the term 
	$\left(c_1^{S^1}(TM)|_{Z}\right)^3$ vanishes so that 
	$\langle c_1(TM)^3, [M] \rangle = 52.$ This finishes the proof.

\end{proof}

\begin{example}[Fano variety of type {\bf (I-3)}]\label{example_1_3}\cite[No. 31 in the list in Section 12.4]{IP}\label{example_1_1}  
	Let $M$ be a projective toric variety whose moment polytope 
	$\mcal{P}$ is given on the left of Figure \ref{figure_1_3}. Note that $\mcal{P}$ is a Delzant polytope so that our variety $M$ is smooth. 
	Moreover, it is easy to check that $\mcal{P}$ is reflexive (with the unique interior point $(1,1,1)$) which guarantees that $M$ is Fano. 
	
	Now, let $T$ be the 3-dimensional compact subtorus of $(\C^*)^3$ acting on $M$ and $S^1$ be the circle subgroup $S^1$ of $T$ generated by $(1,1,1) \in \frak{t}$. 
	Then the $S^1$-fixed point set consists of four points corresponding to vertices $(0,0,0), (0,0,2), (1,1,2),$ $(3,3,0)$ of $\mcal{P}$ and two spheres corresponding to
	edges $\overline{(0,1,2) (0,3,0)}$ and $\overline{(1,0,2) (3,0,0)}$ (colored by red in Figure \ref{figure_1_3}.) If we denote by $\mu : M \rightarrow \mcal{P}$ the moment map
	for the $T$-action, then the balanced moment map for the $S^1$-action is written by
	\[
		H(p) := \langle \mu(p), (1,1,1) \rangle - 3, \quad p \in M
	\]  
	Then it is straightforward to check that the topological fixed point data for the $S^1$-action is exactly the same as in Theorem \ref{theorem_1_3}. 
	
	It is sometimes useful to use so-called a GKM-graph\footnote{See \cite{GKM}, \cite{GZ}, or \cite{CK1} for the precise definition of a GKM-graph and its properties.} 
	to describe higher dimensional Delzant polytopes. 
	In the context of toric variety, a GKM-graph is a ``well-projected'' image (onto a lower dimensional Euclidean space) of the one-skeleton of a Delzant polytope. 
	For example, the right one of Figure \ref{figure_1_3} is the projection image $\mcal{G}$ of $\mcal{P}$ onto the plane $x + y - 2z = 0$ with the coordinate system 
	whose axes are spanned by $(1,-1,0)$ and $(1,1,1)$, respectively. 
	(Note that 
	we can also think of $\mcal{G}$ as a moment map image of the Hamiltonian $T^2$-action generated by $(1,-1,0)$ and $(1,1,1)$ in $\frak{t}$. Then a moment map for the $S^1$-action
	is just a projection of $\mcal{G}$ onto the $y$-axis.)
	 
	\begin{figure}[h]
		\scalebox{1}{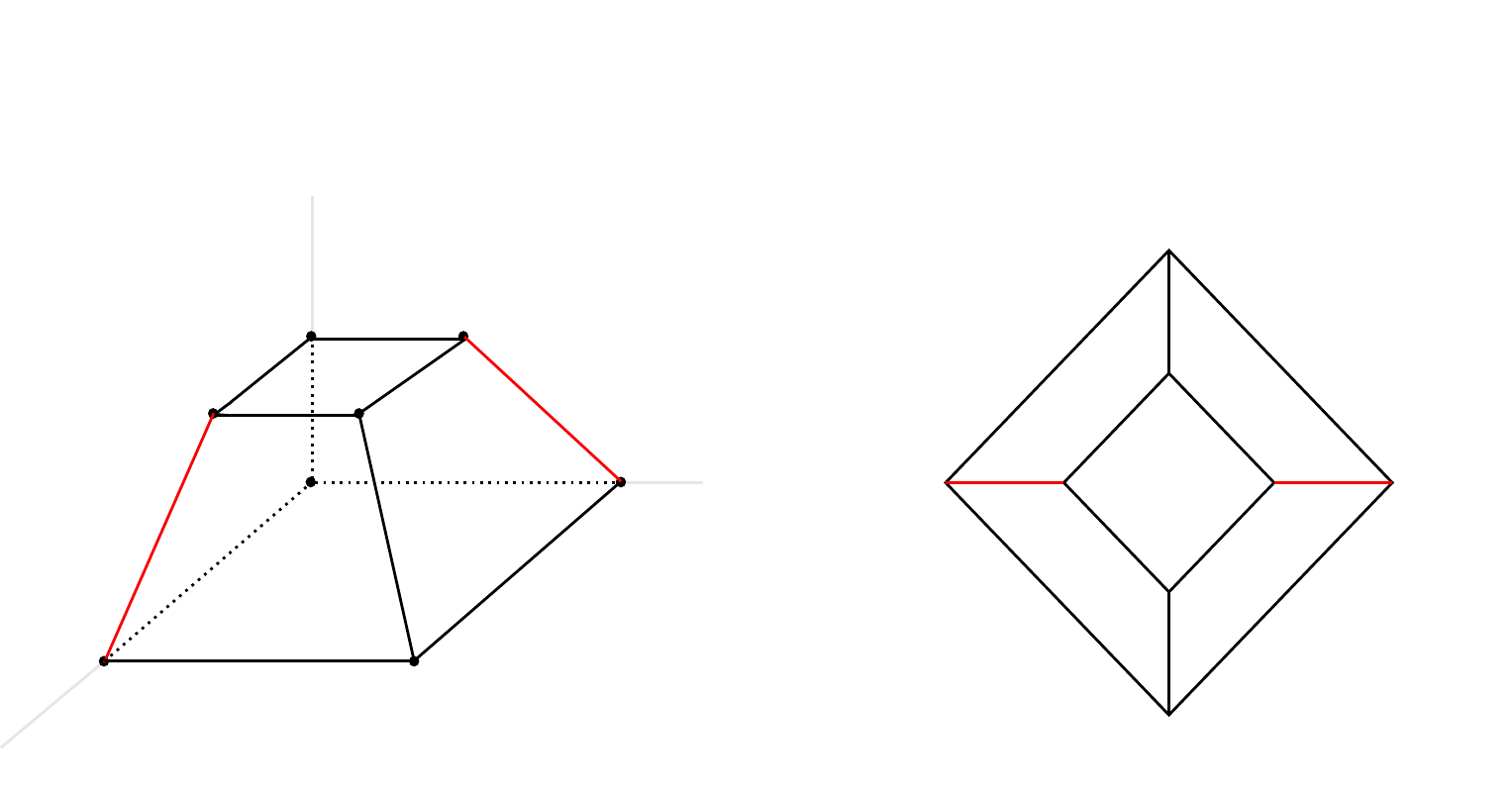}
		\caption{\label{figure_1_3} Moment graph of $\p(\mcal{O} \oplus \mcal{O}(1,1))$}
	\end{figure}

\end{example}

\section{Case II : $\dim Z_{\max} = 2$}
\label{secCaseIIDimZMax2}

In this section, we provide the classification of topological fixed point data for a semifree Hamiltonian circle action on a closed monotone symplectic six-manifold
for the case : $H(Z_{\min}) = -3$ and $H(Z_{\max}) = 2$, which is the case of $Z_{\max} \cong S^2$ and $Z_{\min} = \mathrm{point}$, see Lemma \ref{lemma_possible_critical_values}.
The main idea is basically the same as in Section \ref{secCaseIDimZMax}
but the computation is relatively more complicated.  

We start with two well-known facts about the number of index-two and four fixed points and the volume of the maximal fixed component.   
First, recall that $\mathrm{Crit}~ \mathring{H} \subseteq \{ 0, \pm 1\}$ and each non-extremal fixed component $Z$ satisfies
\[
	\begin{cases}
		\text{$Z$ = pt} \hspace{1cm} \text{if $H(Z) = \pm 1$, \quad or} \\
		\text{$\dim Z = 2$} \quad \text{if $H(Z) = 0$.}
	\end{cases}
\]
Since $H$ is perfect Morse-Bott, the Poincar\'{e} polynomial is given by 
\[
	P_M(t) = \sum b_i(M) t^i = 1 + |Z_{-1}|t^2 + \underbrace{P_{Z_0}(t)}_{ = r + st + rt^2} t^2 + |Z_{1}|t^4 + (1 + t^2) t^4
\]
where $r$ denotes the number of connected components of $Z_0$ and $s = \mathrm{rk} ~H_1(Z_0; \Z)$. In particular, 
the Poincar\'{e} duality implies that 
\begin{equation}\label{equation_plus_one}
	|Z_1| + 1= |Z_{-1}|,
\end{equation}
and therefore we get $Z_{-1} \neq \emptyset$. So, the set of interior critical values of $H$ is one of the followings : 
\[
	\mathrm{Crit}\mathring{H} = \{-1\}, \quad \{ -1,1\},\quad  \{ -1,0\}, \quad \text{or} \quad  \{-1,0,1\}.
\]

Second, we can compute the symplectic volume $\langle [\omega], [Z_{\max}] \rangle$ of $Z_{\max} \cong S^2$ as follows. 
Note that the reduced space near $Z_{\max}$ is an $S^2$-bundle over $S^2$ and it is well-known that there are two diffeomorphism types of 
$S^2$-bundles over $S^2$, namely a trivial bundle $S^2 \times S^2$ or a Hirzebruch surface denoted by $E_{S^2}$.

When $M_{2 - \epsilon}$ is a trivial bundle, we let by $x$ and $y$ in $H^2(M_{2 - \epsilon} ;\Z)$ be the dual classes of the fiber $S^2$ and the base $S^2$ respectively so that 
\[
	\langle xy, [M_{2 - \epsilon}] \rangle = 1, \quad \langle y^2, [M_{2 - \epsilon}] \rangle = \langle x^2, [M_{2 - \epsilon}] \rangle = 0.
\]
Similarly, when $M_{2 - \epsilon} \cong E_{S^2}$, we let $x$ and $y$ be the dual classes of the fiber $S^2$ and the base respectively which satisfy
\[
	\langle xy, [M_{2 - \epsilon}] \rangle = 1, \quad \langle y^2, [M_{2 - \epsilon}] \rangle = -1, \quad \langle x^2, [M_{2 - \epsilon}] \rangle = 0.
\]
In either case, we have the following.

\begin{lemma}\cite[Lemma 6, 7]{Li2}\label{lemma_volume}
	Let $b_{\max}$ be the first Chern number of the normal bundle of $Z_{\max}$. Then 
	\[
		\langle e(P_2^-)^2, [M_{2-\epsilon}] \rangle = -b_{\max}.
	\]
	Also, $b_{\max}$ is even if and only if $M_{2-\epsilon} \cong S^2 \times S^2$. Moreover, if $b_{\max} = 2k$ or $2k+1$, then
	\[
		e(P_2^-) = kx - y.
	\]
\end{lemma}

Using Lemma \ref{lemma_volume}, we obtain the following.

\begin{corollary}\label{corollary_volume}
	Let $(M,\omega)$ be a six-dimensional closed semifree Hamiltonian $S^1$-manifold. Suppose that $[\omega] = c_1(TM)$. If the maximal fixed component $Z_{\max}$ is diffeomorphic to 
	$S^2$ and $b_{\max} \in \Z$ is the first Chern number of the normal bundle of $Z_{\max}$, then 
	\[
		\int_{Z_{\max}} \omega = 2 + b_{\max}.
	\]
\end{corollary}

\begin{proof}
	The proof is straightforward from the fact that 
	\[
		\langle c_1(TM), [Z_{\max}] \rangle = 2 + b_{\max}.
	\]
\end{proof}

Now, we are ready to classify topological fixed point data for the case where $Z_{\max} \cong S^2$ and $Z_{\min} = \mathrm{point}.$
As we mentioned above, we 
\[
	\mathrm{Crit}\mathring{H} = \{-1\}, \quad \{ -1,1\},\quad  \{ -1,0\}, \quad \text{or} \quad  \{-1,0,1\}.
\]

\subsection{${\mathrm{Crit} \mathring{H}} = \{-1\}$}
\label{ssecMathrmCritMathringH1}
~\\

In this case, $Z_{-1}$ consists of a single point since $M_0 \cong M_{2 - \epsilon}$ is an $S^2$-bundle over $S^2 = Z_{\max}$, and in particular, 
$M_0$ is diffeomorphic to $\p^2 \# \overline{\p}^2$ by Proposition \ref{proposition_topology_reduced_space}. 

\begin{theorem}\label{theorem_2_1}
	There is no six-dimensional 
	closed monotone semifree Hamiltonian $S^1$-manifold such that $\mathrm{Crit} H = \{2, -1, -3\}$.
\end{theorem}

\begin{proof}
	Recall that $M_{2 - \epsilon}$ is an $S^2$-bundle over $Z_{\max} \cong S^2$.
	If we denote by $C$ a fiber of the bundle, then $[C] \cdot [C] = 0$ (by the local triviality of a fiber bundle), which implies that
	$\mathrm{PD}(C) = au \pm aE$ for some nonzero $a \in \Z$.
	 Furthermore, since 
	 \[
	 	[\omega_t] = [\omega_0] - te(P_{-1}^+) = 3u - E_1 - (-u + E_1)t = (3+t)u - (1 + t)E_1,  \quad t \in (-1, 2), 
	\] and $C$ is a vanishing cycle at $t=2$, we get
	\[
		\begin{array}{ccl} \vs{0.1cm}
			0 & = & \ds \lim_{t\rightarrow 2} \langle [\omega_t], [C] \rangle \\ \vs{0.1cm}
					& = & (5u - 3E)\cdot (au \pm aE) \\ \vs{0.1cm}
					& = & (5 \pm 3)a \neq 0
		\end{array}
	\]
	which leads to a contradiction. Therefore no such manifold exists.
\end{proof}
\vs{0.1cm}

\subsection{${\mathrm{Crit} \mathring{H}} = \{-1,1\}$}
\label{ssecMathrmCritMathringH11}
~\\

Suppose that $Z_1$ consists of $k$ points (so that $|Z_{-1}| = k+1$ by \eqref{equation_plus_one}.)

\begin{lemma}\label{lemma_2_2_number}
	$k=2$ is the only possible value of $k$.
\end{lemma}

\begin{proof}
	Note that the normal bundle of $Z_{\max}$ in $M$ splits into the direct sum of two complex  line bundles. We denote the first Chern classes of each line bundles by $d_1 q$ and $d_2 q$ in 
	$H^2(Z_{\max}; \Z)$, respectively.
	
	Applying the localization theorem \ref{theorem_localization} to $1$ and $c_1^{S^1}(TM)$, respectively, we get 
	\[
		\begin{array}{ccl} \vs{0.1cm}
			0 & = & \ds \int_M 1 \\ \vs{0.1cm}
				& = & \ds \frac{1}{x^3} + (k+1) \cdot \frac{1}{-x^3} + k \cdot \frac{1}{x^3} + \int_{Z_{\max}} \frac{1}{(-x + d_1 q)(-x + d_2 q)} \\ \vs{0.1cm}
				& = & \ds \frac{1}{x^3} \cdot \int_{Z_{\max}} (d_1 + d_2) q = \frac{d_1 + d_2}{x^3}
		\end{array}
	\]
	so that $d_1 + d_2 = 0$, and 
	\[
		\begin{array}{ccl} \vs{0.1cm}
			0 & = & \ds \int_M c_1^{S^1}(TM) \\ \vs{0.1cm}
				& = & \ds \frac{3x}{x^3} + (k+1) \cdot \frac{x}{-x^3} + k \cdot \frac{-x}{x^3} + \int_{Z_{\max}} 
				\frac{\overbrace{-2x + (d_1 + d_2) q + 2q}^{ = -2x + 2q}}{(\underbrace{-x + d_1 q)(-x + d_2 q)}_{ = x^2}} \\ \vs{0.1cm}
				& = & \ds \frac{1}{x^2} \cdot (3 - 2k - 1) + \int_{Z_{\max}} \frac{-2x^2 + 2xq}{x^3} \\ \vs{0.1cm}
				& = & \ds \frac{1}{x^2} \cdot (3 - 2k - 1 + 2).
		\end{array}
	\]
	So, we get $k = 2$.
\end{proof}

Then Lemma \ref{lemma_2_2_number} implies the following. 

\begin{theorem}\label{theorem_2_2}
	There is no six-dimensional closed monotone semifree Hamiltonian $S^1$-manifold such that $\mathrm{Crit} H = \{2, 1, -1, -3\}$.
\end{theorem}

\begin{proof}
	Lemma \ref{lemma_2_2_number} says that $Z_{-1}$ consists of three points so that $M_0 \cong \p^2 \# 3 \overline{\p}^2$. 
	As $t$ approaches to $1$, two exceptional spheres, namely $C_1$ and $C_2$, are getting smaller in a symplectic sense and eventually vanish on $M_1$.
	In other words, two simultaneous blow-downs occur on the level $t = 1$. 
	
	On the other hand, the Duistermaat-Heckman theorem \ref{theorem_DH} says that
	\[
		\begin{array}{ccl}\vs{0.1cm}
			[\omega_t] = [\omega_0] - t e(P_0^+) & = & (3u - E_1 - E_2 - E_3) - t (-u + E_1 + E_2 + E_3) \quad t\in (-1, 1) \\ \vs{0.1cm}
									& = & (3 + t)u - (E_1 + E_2 + E_3)(1+t). 
		\end{array}
	\]
	Observe that 
	\[
			\ds \langle [\omega_t], u - E_1 - E_2 \rangle = \langle [\omega_t], u - E_2 - E_3 \rangle = \langle [\omega_t], u - E_1 - E_3 \rangle = 1 - t.
	\]
	That is, three disjoint exceptional divisors $\{u - E_i - E_j ~|~ 1 \leq i, j \leq 3, i \neq j\}$ vanish at $t = 1$, which leads to a contradiction.
\end{proof}

	\begin{figure}[h]
		\scalebox{1}{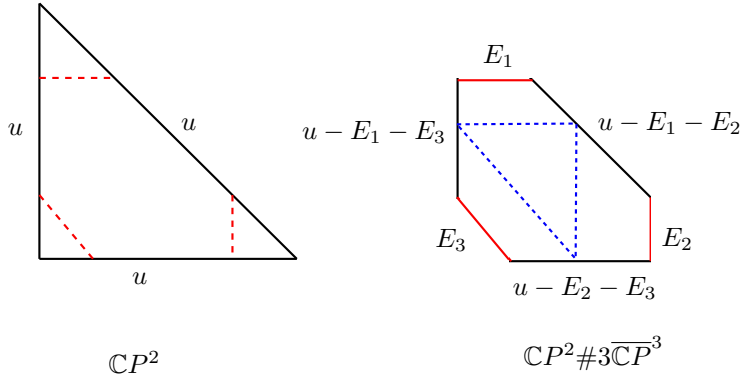}
		\caption{Blow-ups and blow-downs}
	\end{figure}

\subsection{${\mathrm{Crit} \mathring{H}} = \{-1,0\}$}
\label{ssecMathrmCritMathringH1}
~\\

	In this case, we have $|Z_{-1}| = 1$ by \eqref{equation_plus_one} and $M_0 \cong \p^2 \# \overline{\p}^2$.
	Regarding $Z_0$ as an embedded symplectic submanifold of $(M_0, \omega_0)$ via
	\[
		Z_0 \hookrightarrow H^{-1}(0) \stackrel{/S^1}\longrightarrow  M_0,
	\]
	let $\mathrm{PD}(Z_0) = au + bE_1 \in H^2(M_0; \Z)$ for some  $a, b \in \Z$.
	
	Note that $M_{2 - \epsilon}$ is a symplectic $S^2$-bundle over $Z_{\max}$ where we denote by $C$ a fiber of the bundle. Since $M_{2 - \epsilon} \cong M_0$
	by Proposition \ref{proposition_topology_reduced_space}, we can express 
	$\mathrm{PD}(C)$ as a linear combination of $u$ and $E_1$. 

	\begin{lemma}\label{lemma_2_3_1}
		$\mathrm{PD}(C) = u - E_1$. 
	\end{lemma}
	
	\begin{proof}
		Since $[C] \cdot [C] = 0$, we have $\mathrm{PD}(C) = pu \pm pE_1$ for some $p \neq 0$ in $\Z$. Also, the adjunction formula \eqref{equation_adjunction} implies that 
		\[
			 3p \pm p = \langle 3u - E_1, pu \pm pE_1 \rangle = \langle c_1(TM_0), [C] \rangle = [C] \cdot [C] + 2 = 2.
		\]
		So, we have $p=1$ and $\mathrm{PD}(C) = u - E_1$.
	\end{proof}	
	
	\begin{lemma}\label{lemma_2_3_1}
		All possible pairs of $(a,b)$ are $(0,1)$, $(1,0)$, or $(2, -1)$. In any case, we have $Z_0 \cong S^2$.
	\end{lemma}
	
	\begin{proof}
		We obtain three (in)equalities in $a$ and $b$ as follows. First, the Duistermaat-Heckman theorem \ref{theorem_DH} implies that 
		\[
			\begin{array}{ccl}\vs{0.1cm}
				[\omega_t] = [\omega_0] - t e(P_0^+) & = & (3u - E_1) - t(-u + E_1 + au + bE_1) \\ \vs{0.1cm}
										& = & (3 - at + t)u - (1 + bt + t)E_1, \quad t \in (0, 2).
			\end{array}
		\]
		Since the symplectic volume of $C$ tends to 0 as $t \rightarrow 2$, we have 
		\[
			\lim_{t\rightarrow 2} \langle [\omega_t], [C] \rangle = \lim_{t\rightarrow 2} \langle [\omega_t], u-E_1 \rangle = 5 - 2a -2b -3 = 0 \quad \Leftrightarrow \quad a + b = 1.
		\]
		Second, the condition $\langle [\omega_0], [Z_0] \rangle > 0$ implies that $3a+b>0$ (since $[\omega_0] = c_1(TM_0) = 3u - E_1$.)
		Third, consider any section $\sigma$ of the bundle $M_{2-\epsilon}$ over $Z_{\max} = Z_2$. Since the intersection of $\sigma$ and a fiber $C$
		equals one, we have 
		\[
			[\sigma] \cdot [C] = 1, 
		\]
		equivalently $\mathrm{PD}(\sigma) = u + d(u - E_1)$ for some $d \in \Z$. 
		In particular, we have 
		\[
			\begin{array}{ccl}\vs{0.1cm}
				\langle [\omega], [Z_{2}] \rangle & = & \lim_{t \rightarrow 2} \langle [\omega_t], [\sigma] \rangle \\ \vs{0.1cm}
									& = & \lim_{t\rightarrow 2} \langle [\omega_t] \cdot \mathrm{PD}(\sigma), [M_0] \rangle \\ \vs{0.1cm}
									& = & \lim_{t \rightarrow 2} \langle (5 - 2a) (u - E_1) \cdot (u + d(u-E_1)), [M_0] \rangle = 5- 2a > 0 
			\end{array}
		\]
		Combining the three (in)equalities $a+b = 1$, $3a + b > 0$, and $5 - 2a > 0$, we may conclude that $(a,b)$ is either $(0,1)$, $(1,0)$, or $(2,-1)$.
		(Note that this consequence is also obtained from Lemma \ref{lemma_volume})

		It remains to show that $Z_0 \cong S^2$. Using the adjunction formula \eqref{equation_adjunction}, we have 
		\[
			\underbrace{\langle 3u - E_1, au + bE_1 \rangle}_{= ~\text{symplectic area of $Z_0$ on $M_0$} ~= ~3a + b} = a^2 - b^2 + \sum_i 2 - 2g_i.
		\]
		where each connected component of $Z_0$ is indexed by $i$ and $g_i$ denotes its genus. For $(a,b) = (0,1)$, since $\mathrm{Vol}(Z_0) = 3a + b = 1$ and $2 = \sum_i (2 - 2g_i)$, 
		it is easy to see that $Z_0$ is connected and its genus is equal to zero. 
		
		For $(a,b) = (1,0)$, we have $\mathrm{Vol}(Z_0) = 3$ and $2 = \sum_i (2 - 2g_i)$, which implies that there is at least one sphere denoted by $Z_0^1$. Let
		$Z_0^2$ be the complement of $Z_0^1$ in $Z_0$ so that $Z_0^1$ and $Z_0^2$ are disjoint. If we let $\mathrm{PD}(Z_0^1) = a_1u + b_1E_1$ 
		and $\mathrm{PD}(Z_0^2) = a_2u + b_2E_1$, respectively, then 
		\begin{equation}\label{equation_2_3_Z0}
			a_1a_2 - b_1b_2 = 0, \quad \underbrace{a_1 + a_2 = 1, \quad b_1 + b_2 = 0}_{\mathrm{PD}(Z_0^1) + \mathrm{PD}(Z_0^2) = \mathrm{PD}(Z_0)}, 
			\quad \underbrace{\begin{cases} 3a_1 + b_1 = 1 ~\text{or}~ 2 \\ 3a_2 + b_2 = 2 ~\text{or}~ 1 \end{cases}}_{\mathrm{Vol}(Z_0^1) + \mathrm{Vol}(Z_0^2) = 3}
		\end{equation}
		Also, the adjunction formula \eqref{equation_adjunction} for $Z_0^1$ implies that 
		\begin{equation}\label{equation_2_3_adjunction}
			(\text{1 or 2} =) ~3a_1 + b_1 = a_1^2 - b_1^2 + 2.
		\end{equation}
		When $3a_1 + b_1 = 1$, then $a_1 = 0$ and $b_1 = 1$ which implies that $a_2 = 1$ and $b_2 = -1$. Then it contradicts the first equation in 
		\eqref{equation_2_3_Z0}. Similarly, if $3a_1 + b_1 = 2$, then $a_1 = \pm b_1$ which implies that $a_1 = 1$ and $b_1 = -1$, and so $a_2 = 0$ and $b_2 = 1$.
		This also violates the first equation in \eqref{equation_2_3_Z0}. Consequently, $Z_0$ is connected and $Z_0 \cong S^2$.
		
		We can show that $Z_0 \cong S^2$ when $(a,b) = (2,-1)$ in the same way as above. In this case, $\mathrm{Vol}(Z_0) = 5$ and $5 = 3 + \sum 2 - 2g_i$ by the adjunction 
		formula. Thus we can take $Z_0^1 \cong S^2$ and $Z_0^2$ as in the ``$(a,b) = (1,0)$''-case. Then 
		\begin{equation}\label{equation_2_3_Z0_2}
			a_1a_2 - b_1b_2 = 0, \quad \underbrace{a_1 + a_2 = 2, \quad b_1 + b_2 = -1}_{\mathrm{PD}(Z_0^1) + \mathrm{PD}(Z_0^2) = \mathrm{PD}(Z_0)}, 
			\quad \underbrace{\begin{cases} 3a_1 + b_1 = 1 ~\text{or}~ 2 
			~\text{or}~ 3 ~\text{or}~ 4\\ 3a_2 + b_2 = 4 ~\text{or}~ 3 ~\text{or}~ 2 ~\text{or}~ 1\end{cases}}_{\mathrm{Vol}(Z_0^1) + \mathrm{Vol}(Z_0^2) = 5}
		\end{equation}
		and 
		\begin{equation}\label{equation_2_3_adjunction}
			(\text{1 or 2 or 3 or 4} =) ~3a_1 + b_1 = a_1^2 - b_1^2 + 2.
		\end{equation}
		Then we have 
		\begin{table}[H]
			\begin{tabular}{|c|c|c|}
				\hline
				   $3a_1 + b_1$ & $(a_1, b_1)$ & $(a_2, b_2)$\\ \hline \hline
				    1 & $(0, 1)$ & $(2, -2)$ \\ \hline
				    2 & $(1, -1)$ & $(1, 0)$ \\ \hline
				    3 & $(1, 0)$ & $(1, -1)$ \\ \hline
				    4 & $\times$ & $\times$ \\ \hline
			\end{tabular}
		\end{table}
		\noindent
		and check that the first equation of \eqref{equation_2_3_Z0_2} fails in any case. Thus $Z_0$ is connected and $Z_0 = Z_0^1 \cong S^2$ and this completes the proof.
	\end{proof}

	Consequently, we obtain the following.

	\begin{theorem}\label{theorem_2_3}
		Let $(M,\omega)$ be a six-dimensional closed monotone semifree Hamiltonian $S^1$-manifold such that $\mathrm{Crit} H = \{ 2, 0, -1, -3\}$. 
		Then the topological fixed point data is given by
		\begin{table}[H]
			\begin{tabular}{|c|c|c|c|c|c|c|}
				\hline
				    & $(M_0, [\omega_0])$ & $Z_{-3}$ & $Z_{-1}$ & $Z_0$ & $Z_{2}$\\ \hline \hline
				   {\bf (II-3)} &  $(\p^2 \# \overline{\p^2}, 3u - E_1)$ & {\em pt} & {\em pt} &  \makecell{ $Z_0 = S^2$ \\
				    $[Z_0] = au + bE_1$} & $S^2$ \\ \hline    
			\end{tabular}
		\end{table}
		\noindent where $(a,b) = (0,1), (1,0),$ or $(2, -1)$. Moreover, $b_2(M) = 2$ and 
		the Chern number $\langle c_1(TM)^3, [M] \rangle$ is given by
		\[
			\langle c_1(TM)^3, [M] \rangle = 
			\begin{cases}
				\text{$62$ \quad if $(a,b) = (0,1)$ \quad {\bf (II-3.1)}} \\ \text{$54$ \quad if $(a,b) = (1,0)$ \quad {\bf (II-3.2)}} \\ \text{$46$ \quad if $(a,b) = (2,-1)$ \quad {\bf (II-3.3)}}
			\end{cases}
		\]	

	\end{theorem}

\begin{proof}
	We have already computed the topological fixed point data in Lemma \ref{lemma_2_3_1} where the fact $|Z_{-1}| = 1$ follows from \eqref{equation_plus_one}. 
	Also, the perfectness (as a Morse-Bott function) of the moment map $H$ implies that $b_2(M) = 2$. Thus we only need to compute the Chern number in each case. 
	
	If we let $b_+$ and $b_-$ be the first Chern numbers of the positive and negative normal (line) bundles $\xi_+$ and $\xi_-$ of $Z_0$, respectively, then the equivariant first 
	Chern class of the normal bundle of $Z_0$ in $M$
	is 
	\[
		(-x + b_-q) + (x + b_+ q) = (b_- + b_+)q = ([Z_0]\cdot [Z_0]) q
	\] because the normal bundle of $Z_0$ in $M_0$ is isomorphic to $\xi_+ \otimes \xi_-$, see \cite[Proof of Lemma 5]{McD1}. 
	So, we have $c_1(TM)|_{Z_0} = (b_- + b_+ + 2)q$. 
	Applying the localization theorem \ref{theorem_localization} to $c_1^{S^1}(TM)^3$, we get 
	\begin{equation}\label{equation_CN_2_3}
		\begin{array}{ccl}\vs{0.1cm}
			\ds \int_M c_1^{S^1}(TM)^3 & = &  \ds \sum_{Z \subset M^{S^1}} \int_{Z} 
								\frac{\overbrace{c_1^{S^1}(TM)^3|_Z}^{ = \left( c_1^{S^1}(TM)|_{Z}\right)^3}}{e_Z^{S^1}} \\ \vs{0.1cm}
							& = & \ds  \frac{(3x)^3}{x^3} + \frac{(x)^3}{-x^3} + 
							\int_{Z_0} \frac{\overbrace{\left(c_1^{S^1}(TM)|_{Z_0}\right)^3}^{ = \left(c_1(TM)|_{Z_0}\right)^3 = 0}}{e_{Z_0}^{S^1}} + 
							\int_{Z_2} \frac{(-2x + (d_1 + d_2 + 2)q)^3}{(-x +d_1q)(-x + d_2q)} \\ \vs{0.1cm}
							& = & 26 - 8(d_1 + d_2) + 12(d_1 + d_2 + 2) = 50 + 4(d_1 + d_2).
		\end{array}
	\end{equation}
	On the other hand, applying the localization theorem to $c_1^{S^1}(TM)$, we have 
	\begin{equation}\label{equation_2_3_localization}
		\begin{array}{ccl}\vs{0.1cm}
			0 & = & \ds \int_M c_1^{S^1}(TM) \\ \vs{0.1cm}
				& = &  \ds \sum_{Z \subset M^{S^1}} \int_{Z} \frac{c_1^{S^1}(TM)|_Z}{e_Z^{S^1}} \\ \vs{0.1cm}
						& = & 	\ds  \frac{3x}{x^3} + \frac{x}{-x^3} + 
							\int_{Z_0} \frac{\overbrace{c_1^{S^1}(TM)|_{Z_0}}^{ = ([Z_0]\cdot [Z_0] + 2)q = (a^2 - b^2 + 2)q}}{e_{Z_0}^{S^1}} + 
							\int_{Z_2} \frac{-2x + (d_1 + d_2 + 2)q}{(-x +d_1q)(-x + d_2q)} \\ \vs{0.1cm}
							& = & \ds \frac{1}{x^2} \cdot (2 - (a^2 - b^2 + 2) + (2 - d_1 - d_2)) \\ \vs{0.1cm}
		\end{array}		
	\end{equation}
	Therefore we get $d_1 + d_2 = 2 - a^2 + b^2$. Using \eqref{equation_CN_2_3} and \eqref{equation_2_3_localization}, we can confirm that the Chern numbers
	for $(a,b) = (0,1), (1,0),$ and $(2,-1)$ are the same as given in the theorem. This completes the proof.
\end{proof}

\begin{example}[Fano variety of type {\bf (II-3)}]\label{example_2_3}
	We provide algebraic Fano examples for each topological fixed point data given in Theorem \ref{theorem_2_3} as follows.
	\begin{enumerate}
		\item {\bf Case (II-3.1) \cite[No. 36 in the list in Section 12.3]{IP}} 
		: For $(a,b) = (0,1)$, let $M = \pp(\mcal{O} \oplus \mcal{O}(2))$. This is a toric variety with a moment map $\mu : M \rightarrow \mcal{P}$ 
		where the moment polytope $\mcal{P}$ 
		(with respect to the normalized monotone K\"{a}hler form)
		is described by 
		\begin{figure}[H]
			\scalebox{0.8}{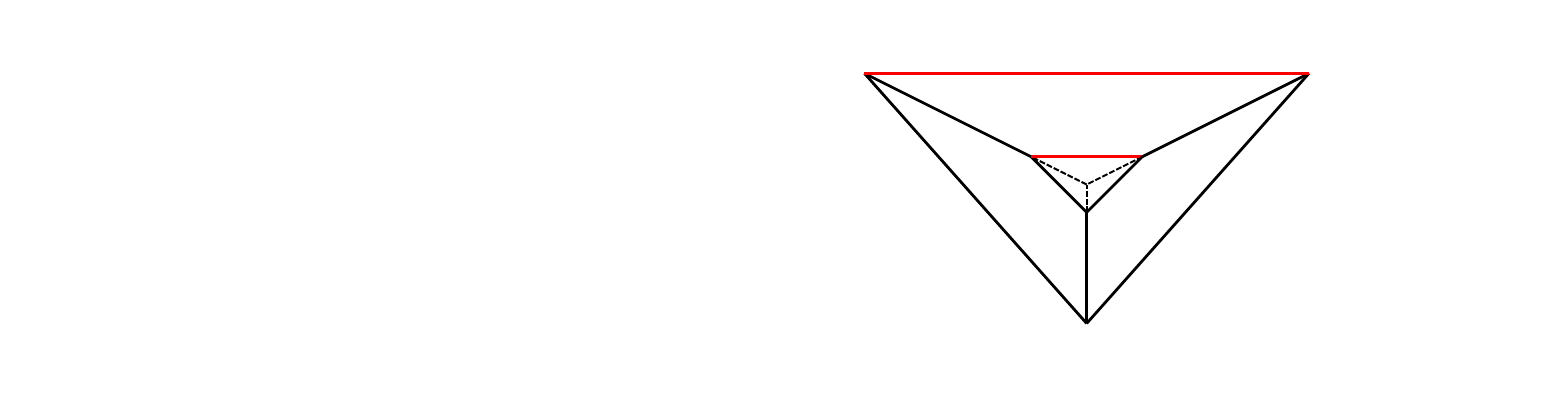}
			\caption{\label{figure_2_3_1} Moment polytope of $M = \pp(\mcal{O} \oplus \mcal{O}(2))$}
		\end{figure}
		\noindent
		where the right one of Figure \ref{figure_2_3_1} is the image of $\mcal{P}$ under the projection $\pi : \R^3 \rightarrow \R^2$ given by 
		\[
			\begin{pmatrix} a \\ b \\ c \end{pmatrix} \mapsto \begin{pmatrix} 1 & -1 & 0 \\ 1 & 1 & 1 \\ \end{pmatrix} \cdot \begin{pmatrix} a \\ b \\ c \end{pmatrix} + \begin{pmatrix} 0 \\ -3 \end{pmatrix}
		\]
		Then $\xi = (1,1,1) \in \frak{t}$ generates a semifree Hamiltonian circle action on $M$ with the balanced moment map 
		$
			\mu_\xi = \pi_y \circ \pi \circ \mu
		$
		where $\pi_y : \R^2 \rightarrow \R$ is the projection onto the $y$-axis. (Note that the ``semifreeness'' can be confirmed by showing that 
		\[
			\langle u_v, (1,1,1) \rangle = \pm 1 ~\text{or} ~0
		\]
		for any vertex $v$ of $\mcal{P}$ and a primitive integral edge vector $u_v$ at $v$. 
		
		To check that $M$ has the same topological fixed point data as in Theorem \ref{theorem_2_3} for $(a,b) = (0,1)$, observe that there are exactly four connected faces 
		corresponding to the fixed components $Z_{-3}, Z_{-1}, Z_0, Z_2$, respectively, for the $S^1$-action generated by $(1,1,1)$, namely
		\[
			(0,0,0), \quad (0,0,2), \quad \overline{(0,1,2) (1,0,2)}, \quad \overline{(0,5,0)(5,0,0)}.
		\]
		
		In fact, we can check other geometric data of $M$, such as the volume of fixed components, coincide with those in Theorem \ref{theorem_2_3}. 
		Note that
		\[
			\lim_{t \rightarrow 2} ~[\omega_t] = [\omega_0] - 2e(P_0^+) = 5(u - E_1)
		\]
		by the Duistermaat-Heckman theorem \ref{theorem_DH}. This implies that the symplectic area of $Z_2$ is $5$ while $Z_0$ has the symplectic area $1$. 
		(This can be obtained from the fact (used in the proof of Lemma \ref{lemma_2_3_1})
		that any section class of the $S^2$-bundle $M_{2 - \epsilon}$ over $Z_2$ 
		is of the form $u + d(u - E_1)$ for some $d \in \Z$.) 
		This is the reason why the edge $\overline{(5,-2)~(5,2)}$ (corresponding to $Z_2$) is five-times as long as $\overline{(-1,0)~(1,0)}$ (corresponding to $Z_0$).
		Furthermore, the Chern number also agrees, i.e., $\langle c_1^3 , [M] \rangle = 62$. \\
		\vs{0.1cm}
		
		\item {\bf Case (II-3.2) \cite[No. 34 in the list in Section 12.3]{IP}} : For $(a,b) = (1,0)$, consider the toric variety $M = \p^1 \times \p^2$ with the moment map $\mu : M \rightarrow \mcal{P}$ where 
		$\mcal{P}$ is given as follows. 
		\begin{figure}[H]
			\scalebox{0.8}{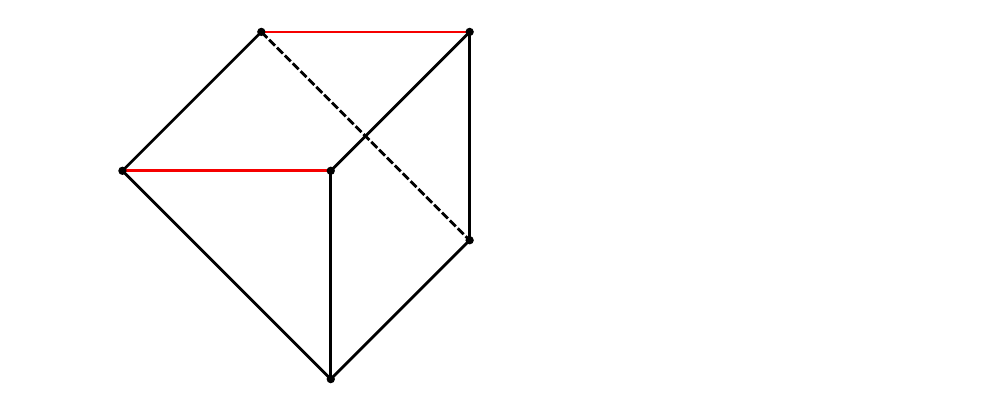}
			\caption{\label{figure_2_3_2} Moment polytope of $M = \p^1 \times \p^2$}
		\end{figure}
		\noindent
		Take $\xi = (0,-1,1) \in \frak{t}$. Then the balanced moment map for the action generated by $\xi$ is factored as $\mu_\xi = \pi_y \circ \pi \circ \mu$ where
		\[
			\pi : \R^3 \rightarrow \R^2, \quad \quad \pi(v) = \begin{pmatrix} 1 & -1 & 0 \\ 0 & -1 & 1 \end{pmatrix} \cdot v + (2,2).
		\]
		The $S^1$-action has four fixed components 
		\[
			\mu^{-1}(0,2,-3), \quad \mu^{-1}(0,0,-3), \quad \mu^{-1}(\overline{(-3,2,0) (0,2,0)}), \quad \mu^{-1}(\overline{(-3,0,0) (0,0,0)}).
		\]
		which correspond to $Z_{-3}, Z_{-1}, Z_0,$ and $Z_2$, respectively.
		
		If we want to whether the symplectic areas of $Z_2$ and $Z_0$ coincide with those given in Theorem \ref{theorem_2_3}, 
		recall that the Duistermaat-Heckman theorem \ref{theorem_DH} says that
		\[
			\lim_{t \rightarrow 2} ~[\omega_t] = [\omega_0] - 2e(P_0^+) = 3(u - E_1). 
		\]
		This implies that the symplectic volume of $Z_2$ (which coincides with the length of $\overline{(-3,0,0) ~(0,0,0)}$, is three and it coincides with the symplectic volume of $Z_0$, 
		the length of $\overline{(-3,2,0) ~(0,2,0)}$ in Figure \ref{figure_2_3_2}.
		Also, we may easily check that $\langle c_1^3 , [M] \rangle = 54$. \\
		\vs{0.1cm}
		
		\item {\bf Case (II-3.3) \cite[No. 31 in the list in Section 12.3]{IP}} : 
		For $(a,b) = (2,-1)$, consider the smooth quadric $Q \subset \p^4$ with a moment map $\mu : Q \rightarrow \mcal{P}$ for the $T^2$-action 
		described in Example \ref{example_1_1}. By taking an equivariant blow-up along the rational curve corresponding 
		to the edge $\overline{(0,3) ~(3,0)}$ in Figure \ref{figure_1_1}, we get a new algebraic variety denoted by $\widetilde{Q}$ whose moment polytope $\widetilde{\mcal{P}}$
		is described in Figure \ref{figure_2_3_3}. 
		
		\begin{figure}[H]
			\scalebox{1}{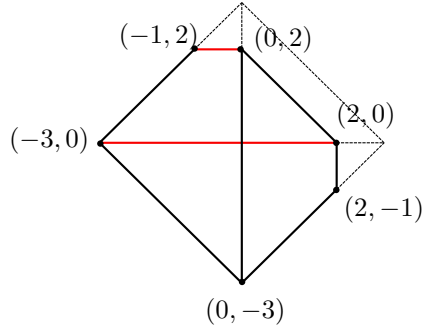}
			\caption{\label{figure_2_3_3} Moment polytope of $\widetilde{Q}$}
		\end{figure}
		\noindent
		The $S^1$-action generated by $\xi = (0,1) \in \frak{t}$ is semifree and has four fixed components corresponding to 
		\[
			(0,-3), \quad (2,-1), \quad \overline{(-3,0)(2,0)}, \quad \overline{(-1,2)(0,2)}.
		\]
		We can easily check that the moment map $\mu_\xi$ for the $S^1$-action is the projection of $\mcal{P}$ onto $y$-axis, and the four fixed components 
		has values $-3, -1, 0, 2$, respectively.
		So, $\widetilde{Q}$ together with the balanced moment map $\mu_\xi$ has the same topological fixed point data 
		given in Theorem \ref{theorem_2_3}.  

		To compare the symplectic areas of $Z_0$ and $Z_2$ with those in Theorem \ref{theorem_2_3}, 
		we use the Duistermaat-Heckman theorem \ref{theorem_DH} so that 
		\[
			\lim_{t \rightarrow 2} ~[\omega_t] = [\omega_0] - 2e(P_0^+) = (u - E_1). 
		\]
		Thus $\mathrm{Vol}(Z_0) = 1$ and it coincides with the length of $\overline{(-1,2) ~(0,2)}$. 
		Also, we obtain 
		\[
			\mathrm{Vol}(Z_0) = \int_{M_0} (3u - E_1) \cdot (2u - E_1) = 5, 
		\] which is equal to the length of $\overline{(-3,0) ~(2,0)}$. 
		Moreover, we may easily check that $\langle c_1^3 , [\widetilde{Q}] \rangle = 46$. \\
		\vs{0.1cm}
	\end{enumerate}
\end{example}

\subsection{${\mathrm{Crit} \mathring{H}} = \{-1,0,1\}$}
\label{ssecMathrmCritMathringH11}
~\\

	Assume that $Z_{-1}$ consists of $k$ points (so that $|Z_1| = k-1$ by \eqref{equation_plus_one}) for some $k \geq 2$. 

	\begin{lemma}\label{lemma_2_4_k}
		The only possible values of $k$ are $2$ and $3$.
	\end{lemma}
	
	\begin{proof}
		Applying the localization theorem \ref{theorem_localization} to $c_1^{S^1}(TM)$, we have 
	\begin{equation}\label{equation_2_4_localization}
		\begin{array}{ccl}\vs{0.1cm}
			0 & = & \ds \int_M c_1^{S^1}(TM) \\ \vs{0.1cm}
				& = &  \ds \sum_{Z \subset M^{S^1}} \int_{Z} \frac{c_1^{S^1}(TM)|_Z}{e_Z^{S^1}} \\ \vs{0.1cm}
						& = & 	\ds  \frac{3x}{x^3} + \frac{kx}{-x^3} + 
							\sum_{Z \subset Z_0} \int_{Z} \frac{\overbrace{c_1^{S^1}(TM)|_{Z}}^{ = (c_1(TM)|_Z)q = \mathrm{Vol}(Z)q}}{e_{Z}^{S^1}} + 
							\frac{-(k-1)x}{x^3} + 
							\int_{Z_2} \frac{-2x + (c_1(TM)|_{Z_2})q}{(-x +a_1q)(-x + a_2q)} \\ \vs{0.1cm}
							& = & \ds \frac{1}{x^2} \left( 4 - 2k - \mathrm{Vol}(Z_0) - \mathrm{Vol}(Z_2) + 4\right) \quad \quad 
							(\because a_1 + a_2 + 2 = \int_{Z_2} c_1(TM)|_{Z_2} = \mathrm{Vol}(Z_2))  \\
		\end{array}		
	\end{equation}
	where $a_1q$ and $a_2q$ denote the first Chern classes of the complex line bundles $\xi_1$ and $\xi_2$ over $Z_2$, respectively, such that $\xi_1 \oplus \xi_2$ is isomorphic to the normal bundle 
	over $Z_2$. Therefore,
	\[
		2k + \mathrm{Vol}(Z_0) + \mathrm{Vol}(Z_{\max}) = 8. 
	\]Since $\mathrm{Vol}(Z_0)$ and $\mathrm{Vol}(Z_2)$ are positive integers and $k \geq 2$, 
	the only possible value of $k$ is $2$ or $3$. This completes the proof.
	\end{proof}

\begin{theorem}\label{theorem_2_4}
		Let $(M,\omega)$ be a six-dimensional closed monotone semifree Hamiltonian $S^1$-manifold such that $\mathrm{Crit} H = \{ 2, 1, 0, -1, -3\}$. 
		Then, up to permutation of indices\footnote{Any permutation on $\{1,2,\cdots,k\}$ switches the ordering of exceptional divisors on $\p^2 \#~k\overline{\p^2}$. }, 
		there are two possible topological fixed point data given by 
			\begin{table}[H]
				\begin{tabular}{|c|c|c|c|c|c|c|c|}
					\hline
					   &  $(M_0, [\omega_0])$ & $Z_{-3}$ & $Z_{-1}$ & $Z_0$ & $Z_1$ & $Z_{2}$\\ \hline \hline
					   {\bf (II-4.1)} & $(\p^2 \# 2\overline{\p^2}, 3u - E_1 - E_2)$ & {\em pt} & {\em 2 pts} & \makecell{ $Z_0 = Z_0^1 ~\dot \cup ~ Z_0^2$ \\
					    $Z_0^1 \cong Z_0^2 \cong S^2$ \\ $[Z_0^1] = u - E_1$ \\ $[Z_0^2] = u - E_1 - E_2$} & {\em pt} & $S^2$ \\ \hline    
					   {\bf (II-4.2)} &  $(\p^2 \# 3\overline{\p^2}, 3u - E_1 - E_2 - E_3)$ & {\em pt} & {\em 3 pts} & \makecell{ $Z_0 = S^2$ \\		
					    			    $[Z_0] = u - E_2 - E_3$} & {\em 2 pts} & $S^2$ \\ \hline    
				\end{tabular}
			\end{table}
			\noindent
		In either case, we have $b_2(M) = 4$ and 
		\[
			\langle c_1(TM)^3, [M] \rangle = \begin{cases}
				44 & \text{\bf (II-4-1)} \\ 42 & \text{\bf (II-4-2)}
			\end{cases}
		\]
	\end{theorem}

\begin{proof}
	Thanks to Lemma \ref{lemma_2_4_k}, we know that $|Z_{-1}| = 2$ (and $|Z_1| = 1$) or $3$ (and $|Z_1| = 2$).  \vs{0.1cm}
	
	Suppose that $|Z_{-1}| = 2$. Then, we have $M_{-1 + \epsilon} \cong M_0 \cong \p^2 \# 2 \overline{\p}^2$. Denote by $\mathrm{PD}(Z_0) = au + bE_1 + cE_2$ for some $a,b,c \in \Z$. 
	Since $|Z_1| = 1$, by Proposition \ref{proposition_topology_reduced_space}, 
	exactly one symplectic blow-down occurs at $(M_1, \omega_1)$, i.e., there is a certain symplectic sphere $C$ with $[C] \cdot [C] = -1$ in $(M_0, \omega_0)$ vanishing at level $t = 1$.  
	So, 
		\begin{itemize}
			\item $\mathrm{PD}(C) = E_1, E_2,$ or $u - E_1 - E_2$ \quad (by Lemma \ref{lemma_list_exceptional}), 
			\item $\langle [\omega_0], [C] \rangle = \langle c_1(TM_0), [C] \rangle = [C]\cdot[C] + 2 = 1$ \quad (by the adjunction formula \eqref{equation_adjunction}),
			\item $\langle [\omega_1], [C] \rangle = 0$
		\end{itemize}
	where the last equation, letting $\mathrm{PD}(C) = xu + yE_1 + zE_2$, can be rephrased as  
	\begin{equation}\label{equation_2_4_blowdown}
			x(4-a) + y(b+2) + z(c+2) = 0.
	\end{equation}
	which follows from the fact that 
	\[
		\begin{array}{ccl}\vs{0.1cm}
			\lim_{t \rightarrow 1} ~[\omega_t] & = & [\omega_0] - e(P_0^+) \\ \vs{0.1cm}
									& = & = 3u - E_1 - E_2 - (-u + E_1 + E_2 + \mathrm{PD}(Z_0)) \\ \vs{0.1cm}
									& = & (4-a)u - (b+2)E_1 - (c+2)E_2
		\end{array}
	\]
	by the Duistermaat-Heckman theorem \ref{theorem_DH}. \vs{0.1cm}
	
	If $\mathrm{PD}(C) = u - E_1 - E_2$ (so that $x = 1, y = z = -1$), then $a + b + c = 0$ by \eqref{equation_2_4_blowdown}. Also, since $E_1$ and $E_2$ should not vanish on 
	the reduced space $M_1$, we have 
	\[
		\langle [\omega_1], E_1 \rangle = b+2 > 0 \quad \text{and} \quad \langle [\omega_1], E_2 \rangle = c+2 >0. 
	\]
	Moreover, as the symplectic area $\langle [\omega_t]^2, [M_t] \rangle$ is consistently positive for every $t \in (-3, 2)$,
	the coefficient of $u$ of $[\omega_t]$ should never vanish (by the mean value theorem), in particular, we have $4-a > 0$. 
	Furthermore, since $\langle [\omega_0], [Z_0] \rangle  = \langle c_1(TM), [Z_0] \rangle \geq 1$,  
	we also get $3a + b + c >0$. To sum up, we obtain 
	\begin{equation}\label{equation_II_4_inequality}
		a + b + c = 0, \quad a \leq 3,\quad  b \geq -1, \quad c \geq -1, \quad 3a + b + c \geq 1. 
	\end{equation}
	Solving \eqref{equation_II_4_inequality}, 
	we see that $\mathrm{PD}(Z_0) = 2u - E_1 - E_2$, $u - E_1$, or $u - E_2$. However, in either case, it satisfies 
	\[
		\lim_{t \rightarrow 2} \langle [\omega_t]^2, [M_t] \rangle = (5-2a)^2 - (2b+3)^2 - (2c+3)^2 < 0
	\] 
	which leads to a contradiction. Therefore, we have $\mathrm{PD}(C) \neq u - E_1 - E_2$. \vs{0.1cm}
	
	If $\pd(C) = E_1$ (or equally $\pd(C) = E_2$ up to permutation of indices), then we have $x = z = 0$ and $y = 1$ so that $b = -2$ by \eqref{equation_2_4_blowdown}. 
	Moreover, other exceptional classes $u - E_1 - E_2$ and $E_2$ should not vanish at $t = 1$ so that 
	$\langle [\omega_1], u - E_1 - E_2 \rangle = 4 - a - (c+2) >0$ and $\langle [\omega_1], E_2 \rangle = c+2 > 0$. So,
	\[
		a < 2 - c, \quad c \geq -1, \quad \text{and} \quad b = -2.
	\]
	Together with the condition $\langle [\omega_0], [Z_0] \rangle > 0$ (equivalently $3a + b + c > 0$), we can easily check that the only possible cases are 
	$\pd(Z_0) = 2u  - 2E_1 - E_2$ or $u - 2E_1$.  
	
	If $\pd(Z_0) = u - 2E_1$, then $Z_0$ has the symplectic area $\langle [\omega_0], [Z_0] \rangle = 1$ so that it is connected. 
	On the other hand, the adjunction formula \eqref{equation_adjunction} implies that 
	\[
		1 = \langle c_1(TM_0), [Z_0] \rangle = [Z_0] \cdot [Z_0] + 2 - 2g = -1 -2g
	\]
	which is impossible. Therefore, the only possible case is that $\pd(Z_0) = 2u - 2E_1 - E_2$. 
	So, the symplectic area of $Z_0$ is $\langle [\omega_0], [Z_0] \rangle = 3$, which means that $Z_0$ consists of at most three connected components. 
	In addition, the adjunction formula \eqref{equation_adjunction} says that 
	\[
		3 = \langle c_1(TM_0), [Z_0] \rangle = [Z_0] \cdot [Z_0] + \sum_i (2 - 2g_i) = -1 + \sum_i (2 - 2g_i)
	\]
	where the sum is taken over all connected components of $Z_0$. This implies that $Z_0$ should contain at least two spheres and we have two possibilities : 
	\begin{itemize}
		\item $Z_0 = S^2 \sqcup S^2$ (with symplectic areas $1$ and $2$), or 
		\item $Z_0 = S^2 \sqcup S^2 \sqcup T^2$ (with symplectic areas 1). 
	\end{itemize}
	In either case, we denote the two sphere components by $Z_0^1$ and $Z_0^2$ and assume the symplectic area of $Z_0^1$ is equal to one. 
	Then the adjunction formula \eqref{equation_adjunction} implies that $[Z_0^1] \cdot [Z_0^1] = -1$. Then, by Lemma \ref{lemma_list_exceptional}, 
	we have $\pd(Z_0^1) = u - E_1 - E_2$, $E_1$, or $E_2$. 
	
	First, if $Z_0 = Z_0^1 \sqcup Z_0^2$, then we have $\pd(Z_0^1) = u - E_1 - E_2$ and $\pd(Z_0^2) = u - E_1$ because, if $\pd(Z_0^1) = E_1$ 
	(or $E_2$ respectively), then 
	$\pd(Z_0^2)$ should be $2u - 3E_1 - E_2$ (or $2u - 2E_1 - 2E_2$ respectively), which implies that $[Z_0^1] \cdot [Z_0^2] \neq 0$ and this cannot be happened since $Z_0^1$ and $Z_0^2$
	are disjoint.
	
	Second, if $Z_0 = Z_0^1 \sqcup Z_0^2 \sqcup Z_0^3$ with $Z_0^3 \cong T^2$, then the symplectic area of $Z_0^1$ and $Z_0^2$ are all equal to one so that 
	$[Z_0^1] \cdot [Z_0^1] = [Z_0^2] \cdot [Z_0^2] = -1$ by
	the adjunction formula \eqref{equation_adjunction}. Again by Lemma \ref{lemma_list_exceptional}, $\pd(Z_0^1)$ and $\pd(Z_0^2)$ are one of $\{u - E_1 - E_2, E_1, E_2\}$, respectively. 
	Since $[Z_0^1] \cdot [Z_0^2] = 0$, the only possible case is that $\pd(Z_0^1) = E_1$ and $\pd(Z_0^2) = E_2$, or $\pd(Z_0^1) = E_2$ and $\pd(Z_0^2) = E_1$. However, 
	in either case, we have 
	$\pd(Z_0^3) = 2u - 3E_1 - 2E_2$, which is impossible because $[Z_0^1] \cdot [Z_0^3] \neq 0$. 
	
	Consequently, if $|Z_{-1}| = 2$, then the only possible case is where $Z_0$ consists of two spheres whose dual classes are $u - E_1 - E_2$ and $u - E_1$, respectively. This proves the half of 
	Theorem \ref{theorem_2_4}. \vs{0.3cm}
	
	Now, we consider the case where $|Z_{-1}| = 3$ (and $|Z_1| = 2$ by \eqref{equation_plus_one}). In this case, we have 
	\[
		M_{-1 + \epsilon} \cong M_0 \cong \p^2 \# 3 \p^3.
	\]
	On $(M_1, \omega_1)$, two blow-downs occur simultaneously and we denote the exceptional divisors by $C_1$ and $C_2$ (with $[C_1] \cdot [C_2] = 0$). 
	Let $\pd(Z_0) = au + bE_1 + cE_2 + dE_3$ for some $a,b,c,d \in \Z$. Note that, up to permutation of indices, there are three possible cases : 
	\begin{itemize}
		\item {\bf Case I :} $\pd(C_1) = E_1$ and $\pd(C_2) = E_2$, 
		\item {\bf Case II :} $\pd(C_1) = E_1$ and $\pd(C_2) = u - E_2 - E_3$, 
		\item {\bf Case III :} $\pd(C_1) = u - E_1 - E_2$ and $\pd(C_2) = u - E_1 - E_3$. 
	\end{itemize} 
	For each case, let us investigate the (in)equalities $\langle [\omega_1], [C_1] \rangle = \langle [\omega_1], [C_2] \rangle = 0$, $\langle [\omega_0], [Z_0] \rangle > 0$, and
	$\langle [\omega_t]^2, [M_t] \rangle > 0$ for $t < 2$. Note that 
	\[
		[\omega_1] = (4-a)u - (b+2)E_1 - (c+2)E_2 -(d+2)E_3
	\] by the Duistermaat-Heckman theorem \ref{theorem_DH}, and 
	\begin{equation}\label{equation_omega_t}
		\lim_{t \rightarrow 2} ~[\omega_t] = (5-2a)u - (2b+3)E_1 - (2c + 3)E_2 - (2d + 3)E_3 - \mathrm{PD}(C_1) - \mathrm{PD}(C_2).
	\end{equation}	
	\vs{0.3cm}
	
	\noindent {\bf Case I.} Suppose that $\pd(C_1) = E_1$ and $\pd(C_2) = E_2$. Then we have 
	\[
		\langle [\omega_1], [C_1] \rangle = \langle [\omega_1], [C_2] \rangle = b+2 = c+2 = 0, \quad (\Leftrightarrow \quad b = c = -2.)
	\]
	Also, by assumption, the classes $E_3$, $u - E_1 - E_2$, $u - E_1 - E_3$, and $u - E_2 - E_3$ do not vanish on $(M_1, \omega_1)$ and therefore
	we have 
	\[
		\langle [\omega_1], E_3 \rangle = d + 2 > 0, \quad \langle [\omega_1], u - E_1 - E_2 \rangle = -a + 4> 0, \quad 
	\] 
	and 
	\[
		\langle [\omega_1], u - E_1 - E_3 \rangle = -a - b - d > 0, \quad \langle [\omega_1], u - E_2 - E_3 \rangle = -a - c - d > 0
	\]
	(or equivalently $2 > a + d$).
	Also, since 
	$\langle [\omega_0], [Z_0] \rangle > 0$, and $\langle [\omega_t]^2, [M_t] \rangle > 0$ for $t < 2$, we have 
	\[
		3a +  b + c + d > 0, \quad 5 - 2a \geq 0.
	\]
	Consequently, we obtain
	\[
		a \leq 2, \quad b = c = -2, \quad d \geq -1, \quad a + d \leq 1, \quad 3a + d \geq 5.
	\]
	This has the only integral solution $a = 2, b = c = -2, d = -1$
	Thus $\pd(Z_0) = 2u - 2E_1 - 2E_2 - E_3$ and the symplectic area of $Z_0$ is 
	\[
		\langle [\omega_0], [Z_0] \rangle = 1,
	\]
	which implies that $Z_0$ is connected. On the other hand, the adjunction formula \eqref{equation_adjunction} says that 
	\[
		1 = [Z_0] \cdot [Z_0] + 2 - 2g = -3 - 2g, \quad g \geq 0
	\]
	which is impossible. 
	Therefore, $\left( \pd(C_1), \pd(C_2) \right) \neq (E_1, E_2)$. \vs{0.3cm}

	\noindent {\bf Case II.} Assume that $\pd(C_1) = E_1$ and $\pd(C_2) = u - E_2 - E_3$. Then we obtain 
	\[
		\langle [\omega_1], [C_1] \rangle = b+2 = 0, \quad \langle [\omega_1], [C_2] \rangle = (4-a) - (c+2) - (d+2) = 0 \quad (\Leftrightarrow \quad b = -2, ~a + c + d = 0).
	\]
	Also, since $E_2$, $E_3$, $u - E_1 - E_2$, and $u - E_1 - E_3$ do not vanish on $(M_1, \omega_1)$, 
	we get
	\[
		\langle [\omega_1], E_2 \rangle  = c + 2 > 0, \quad \langle [\omega_1], E_3 \rangle = d + 2 > 0
	\]
	and
	\[
		\langle [\omega_1], u - E_1 - E_2 \rangle = -a - b - c > 0, \quad \langle [\omega_1], u - E_1 - E_3 \rangle = -a - b - d > 0.
	\] 
	Furthermore, since 
	$\langle [\omega_0], [Z_0] \rangle > 0$, and $\langle [\omega_t]^2, [M_t] \rangle > 0$ for $t < 2$, we have 
	\[
		3a +  b + c + d > 0, \quad 5 - 2a \geq 0.
	\]
	Therefore, 
	\[
		a \leq 2, \quad b = -2, \quad c \geq -1, \quad d \geq -1, \quad a + c \leq 1, \quad a + d \leq 1, \quad 3a + c + d \geq 3.
	\]
	and it has a unique solution $a = 2, b = -2, c = d = -1$ so that $\pd(Z_0) = 2u - 2E_1 - E_2 - E_3$. 
	
	On the other hand, the Euler class is given by $e(P_2^-) = 2u - E_1 - E_3$ and, in particular, $\langle e(P_2^-)^2, [M_{2-\epsilon}] \rangle = 2$. 
	Applying Lemma \ref{lemma_volume}, the first Chern number of the normal bundle of $Z_{\max}$ is $b_{\max} = -2$.  Then, by Corollary \ref{corollary_volume} implies that 
	\[
		\int_{Z_{\max}} \omega = 2 + b_{\max} = 0
	\]
	which leads to a contradiction. Consequently, we see that $\left( \pd(C_1), \pd(C_2) \right) \neq (E_1, u - E_2 - E_3)$. \vs{0.3cm}
	

	\noindent {\bf Case III.} Assume that $\pd(C_1) = u - E_1 - E_2$ and $\pd(C_2) = u - E_1 - E_3$. 
	Then, by the fact that  $\langle [\omega_1], [C_1] \rangle = \langle [\omega_1], [C_2] \rangle = 0$, we obtain
	\[
		(4-a) - (b+2) - (c+2) = (4 - a) - (b+2) - (d+2) = 0, \quad (\Leftrightarrow \quad a + b + c = a + b + d = 0) .
	\]	
	Also, since the classes $E_1, E_2, E_3, u - E_2 - E_3$ 
	 do not vanish on $(M_1, \omega_1)$, we have 
	 \[
	 	\langle [\omega_1], E_1 \rangle = 2 + b > 0,  \quad 
	 	\langle [\omega_1], E_2 \rangle = 2 + c > 0,  \quad 
	 	\langle [\omega_1], E_3 \rangle = 2 + d > 0,
	\]
	and 
	\[
	 	\langle [\omega_1], u - E_2 - E_3 \rangle = (4-a) - (c+2) - (d+2) > 0 \quad (\Leftrightarrow \quad a + c + d < 0).
	\]
	Moreover, since $\langle [\omega_0], [Z_0] \rangle > 0$, and $\langle [\omega_t]^2, [M_t] \rangle > 0$ for $t < 2$, we have 
	\[
		\langle [\omega_0], [Z_0] \rangle = 3a + b + c + d > 0 \quad \text{and} \quad 	5 - 2a \geq 0.
	\]
	To sum up, we obtain
	\[
		a + b + c = a + b + d = 0,\quad b,c,d \geq -1,\quad a+c+d \leq -1, \quad 3a+b+c+d \geq 1, \quad a \leq 2
	\]
	and it has a unique solution $(a,b,c,d) = (1,0,-1,-1)$, i.e., $\mathrm{PD}(Z_0) = u - E_2 - E_3$. 
	It follows from the adjunction formula \eqref{equation_adjunction} that 
	\[
		1 = \langle c_1(TM), [Z_0] \rangle = [Z_0] \cdot [Z_0] + 2 - 2g = 1 - 2g \quad (\Leftrightarrow g = 0)
	\]
	and hence $Z_0 \cong S^2$. This completes the proof.
\end{proof}

\begin{example}[Fano variety of type {\bf (II-4)}]\label{example_2_4} In this example, we provide Fano varieties equipped with semifree $\C^*$-actions 
having topological fixed point data described in Theorem \ref{theorem_2_4}. We first denote by $S_d$ the $(9-d)$-times blow up of $\p^2$, that is, a del Pezzo surface of degree $d$ 
where $0 \leq d\leq 8$. \vs{0.1cm}

\begin{enumerate}
    \item {\bf Case (II-4.1) \cite[No. 11 in the list in Section 12.5]{IP}} :
    Let $M = \p^1 \times S_8$ equipped with the monotone K\"{a}hler form $\omega$ with $c_1(TM) = [\omega] \in H^2(M; \Z)$. Since $\p^1$ and $S_8$ are both 
    toric varieties, $M$ is also a toric variety with the induced Hamiltonian $T^3$-action whose moment map image is the product of a closed interval (of length 2)
    and a right trapezoid as in Figure \ref{figure_2_4_1}. Let $C$ be a smooth rational curve corresponding to the edge (a dotted edge in Figure \ref{figure_2_4_1}) connecting $(0,2,2)$ and $(1,2,2)$. 
    Let $\widetilde{M}$ be the monotone toric blow-up of $M$ along $C$ with a moment map $\mu$ where the moment map image $\mu(\widetilde{M})$ is described below.   \vs{0.3cm}
		\begin{figure}[H]
			\scalebox{0.8}{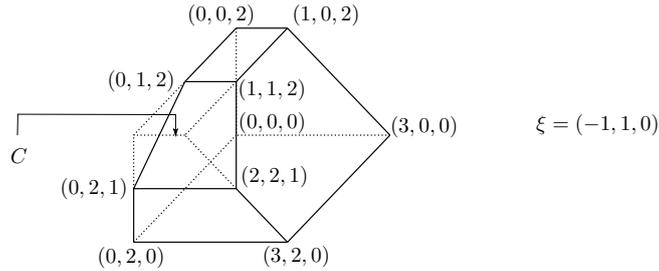} \vs{0.3cm}
			\caption{\label{figure_2_4_1} Blow-up of $\p^1 \times S_8$ along the sphere corresponding to the edge $C$.}
		\end{figure}
  Take the $S^1$-subgroup of $T^3$ generated by $\xi = (-1,1,0) \in \frak{t}$. The $S^1$-action is semifree since the dot product of each primitive edge vectors and $\xi$ is 
   either $0$ or $\pm 1$. Moreover, the fixed point set $\widetilde{M}^{S^1}$ can be expressed as \\
\begin{itemize}
   	\item $Z_{\min} = \mu^{-1}(3,0,0)) = \mathrm{pt},$ \vs{0.1cm}
   	\item $Z_{-1} = \mu^{-1}(3,2,0) ~\cup ~\mu^{-1}(1,0,2) = \mathrm{two~pts} ,$ \vs{0.2cm}
	\item $Z_0 = \underbrace{\mu^{-1}\left(\overline{(1,1,2) ~(2,2,1)}\right)}_{\text{volume} = 1} ~\cup ~\underbrace{\mu^{-1}\left(\overline{(0,0,0) ~(0,0,2)}\right)}_{\text{volume} = 2} ~\cong 
	~S^2 ~\dot \cup ~S^2,$ \vs{0.1cm}
	\item $Z_1 = \mu^{-1}(0,1,2) = \mathrm{pt}$, \vs{0.2cm}
	\item $Z_{\max} = Z_2 = \underbrace{\mu^{-1}\left(\overline{(0,2,0) ~(0,2,1)}\right)}_{\text{volume} = 1} ~\cong ~S^2$ \vs{0.1cm}
\end{itemize}
and this data coincides with the fixed point data {\bf (II-4.1)} in Theorem \ref{theorem_2_4}.
\vs{0.5cm}

    \item {\bf Case (II-4.2) \cite[No. 10 in the list in Section 12.5]{IP}} :
    Let $M = S_7 \times \p^1$ where $S_7$ denotes the del Pezzo surface of degree $7$, i.e., two points blow up of $\p^2$. Then the toric structure on 
    $S_7$ and $\p^1$ inherits a toric structure on $M$ whose moment polytope is given in Figure \ref{figure_2_4_2}. 
    
		\begin{figure}[H]
			\scalebox{0.8}{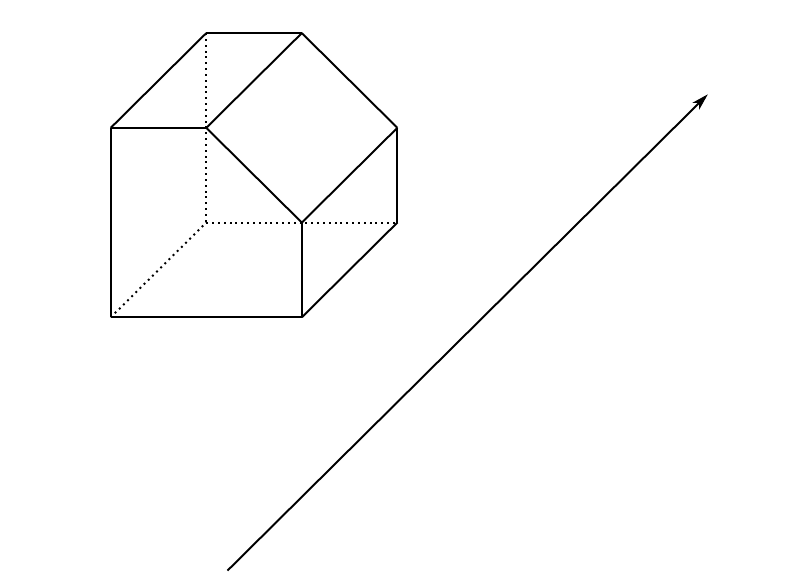} \vs{0.3cm}
			\caption{\label{figure_2_4_2} $\p^1 \times S_7$.}
		\end{figure}
		
	\noindent Now, let $S^1$ be the subgroup of $T^3$ generated by $\xi = (1, -1, 1) \in \frak{t}$. Then the $S^1$-action is semifree because every primitive edge vector 
	(except for those corresponding to fixed components of the $S^1$-action) is either one of $(\pm 1,0,0)$, $(0, \pm 1, 0)$, or $(0,0,\pm 1)$. Moreover, the fixed components for the action are given by \\ 
\begin{itemize}
   	\item $Z_{\min} = \mu^{-1}(0,2,0)) = \mathrm{pt}$, \vs{0.2cm}
   	\item $Z_{-1} = \mu^{-1}(2,2,0) ~\cup ~\mu^{-1}(0,2,2) ~\cup ~\mu^{-1}(0,0,0) = \mathrm{three~pts} ,$ \vs{0.3cm}
	\item $Z_0 = \underbrace{\mu^{-1}(\overline{(2,2,1) ~(1,2,2)})}_{\text{volume} = 1} ~\cong 
	~S^2,$ \vs{0.1cm}
	\item $Z_1 = \mu^{-1}(0,0,2) ~\cup ~\mu^{-1}(2,0,0) = \mathrm{two ~pts},$ \vs{0.2cm}
	\item $Z_{\max} = Z_2 = \underbrace{\mu^{-1}(\overline{(1,0,2) ~(2,0,1)})}_{\text{volume} = 1} ~\cong ~S^2$ \vs{0.2cm}
\end{itemize}
and are exactly the same as {\bf (II-4.2)} in Theorem \ref{theorem_2_4}.
\end{enumerate}

\end{example}

\section{Case III : $\dim Z_{\max} = 4$}
\label{secCaseIIIDimZMax4}

In this section, we give the complete classification of topological fixed point data in the case where 
\[
	H(Z_{\min}) = -3 \quad \text{and} \quad H(Z_{\max}) = 1, 
\]
equivalently, $Z_{\min} = \mathrm{point}$ and $\dim Z_{\max} = 4$. 
We also provide algebraic Fano examples for each cases and describe them in terms of moment polytopes for certain 
Hamiltonian torus actions as in Section \ref{secCaseIDimZMax} and \ref{secCaseIIDimZMax2}.

Let $(M,\omega)$ be a six-dimensional closed semifree Hamiltonian $S^1$-manifold with $[\omega] = c_1(TM)$ and let $H$ be the balanced moment map for the action.
Note that Lemma \ref{lemma_possible_critical_values} implies all possible non-extremal critical values of $H$ are $\pm 1$ or $0$. Also, since $Z_{\max}$ is four dimensional, 
we have $M_0 \cong Z_{\max}$ by Proposition \ref{proposition_topology_reduced_space}, and therefore 
$Z_{\max} \cong M_0 \cong M_{-1+\epsilon}$ is $|Z_{-1}|$-times blow-up of $\p^2$. \\

\subsection{${\mathrm{Crit} \mathring{H}} = \emptyset$}
~\\

Assume that there is no non-extremal fixed point. Then, 
\[
	M_{\max} \cong M_0 \cong M_{-3 + \epsilon} \cong \p^2. 
\]
In addition, since $e(P^+_{-3}) = -u \in H^2(\p^2 ; \Z)$, the cohomology class of the reduced symplectic form on each $M_t$ is given by 
\[
	[\omega_t] = (t + 3) u, \quad -3 \leq t \leq 1.
\]
by the Duistermaat-Heckman theorem \ref{theorem_DH}. 

\begin{theorem}\label{theorem_3_1}
		Let $(M,\omega)$ be a six-dimensional closed monotone Hamiltonian $S^1$-manifold such that $\mathrm{Crit} H = \{ 1,-3\}$. 
		Then there is a unique possible topological fixed point data given by 
			\begin{table}[H]
				\begin{tabular}{|c|c|c|c|}
					\hline
					   &  $(M_0, [\omega_0])$ & $Z_{-3}$ & $Z_1$ \\ \hline \hline
					   {\bf (III-1)} & $(\p^2, 3u)$ & {\em pt} & $\p^2$ \\ \hline
				\end{tabular}
			\end{table}
			\noindent
		Moreover, we have $b_2(M) = 1$ and 
		$
			\langle c_1(TM)^3, [M] \rangle = 64.
		$
\end{theorem}

\begin{proof}
	We only need to prove that $\langle c_1(TM)^3, [M] \rangle = 64.$ Using the localization theorem \ref{theorem_localization}, we obtain
	\[
		\begin{array}{ccl}\vs{0.1cm}
			\ds \int_M c_1^{S^1}(TM)^3 & = & \ds  \sum_{Z \subset M^{S^1}} \int_Z \frac{\left(c_1^{S^1}(TM)|_Z\right)^3}{e_Z^{S^1}} \\ \vs{0.2cm}
								& = & \ds  \frac{(3x)^3}{x^3} + \int_{Z_{\max}} \frac{\left(3u + (-x - e(P^-_1)\right)^3 }{-x - e(P^-_{1})} \quad \quad (e(P^-_1) = -u) \\ \vs{0.1cm}
								& = & 27 + \ds \int_{Z_{\max}} \frac{(4u - x)^3}{u - x} = 27 + \ds \int_{Z_{\max}} (48u^2 - 12u^2 + u^2) = 64.
		\end{array}
	\]
\end{proof}

\begin{example}[Fano variety of type {\bf (III-1)}]\cite[17th in the list in p. 215]{IP}\label{example_3_1}
	Let $M = \p^3$ equipped with the Fubini-Study form $\omega$ with $[\omega ] = 4u = c_1(TM)$. 
	Then $(\p^3, \omega)$ admits a toric structure given by 
	\[
		(t_1, t_2, t_3) \cdot [x,y,z,w] = [t_1x, t_2y, t_3z, w], \quad (t_1, t_2, t_3) \in T^3, [x,y,z,w] \in \p^3
	\]
	and the corresponding moment polytope is the 3-simplex in $\R^3$ as in Figure \ref{figure_3_1}.
	
		\begin{figure}[H]
			\scalebox{0.8}{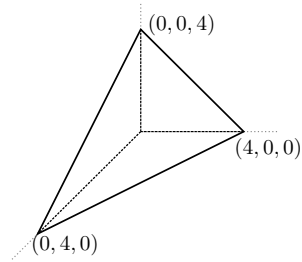} \vs{0.3cm}
			\caption{\label{figure_3_1} Moment polytope for $\p^3$.}
		\end{figure}
		
	\noindent If we take a subgroup $S^1 \subset T^3$ generated by $\xi = (1,1,1)$, then the induced action is expressed as
	\begin{equation}\label{equation_standard_action}
		t \cdot [x,y,z,w] = [tx,ty,tz,w], \quad t\in S^1, [x,y,z,w] \in \p^3
	\end{equation}
	which is semifree with two fixed components
	\[
		Z_{\min} = \mu^{-1}(0,0,0) = \mathrm{pt} \quad \text{and} \quad Z_{\max} = \mu^{-1}(\Delta) \cong \p^2
	\]
	where $\Delta$ is the triangle whose vertices are $(4,0,0)$, $(0,4,0)$, and $(0,0,4)$. This fixed point data exactly coincides with {\bf (III-1)} in Theorem \ref{theorem_3_1}.
\end{example}
\vs{0.3cm}

\subsection{${\mathrm{Crit} \mathring{H}} = \{ -1\}$}
~\\

Let $k = |Z_{-1}|$ be the number of index two fixed points. Then $M_{-1 + \epsilon} \cong M_0 \cong Z_{\max}$ is the blow-up of $\p^2$ at $k$-times. 
We denote by $E_1, \cdots, E_k$ the corresponding exceptional divisors. 
By Theorem \ref{theorem_DH}, we have 
\[
	[\omega_1] = 4u - 2E_1 - \cdots - 2E_k.
\]
\begin{lemma}\label{lemma_3_2}
	The only possible case is $k = 1$.
\end{lemma}
\begin{proof}
	Suppose that $k > 1$ and consider a symplectic sphere $C$ in the class $u - E_1 - E_2$. Then 
	\[
		\langle [\omega_1], [C] \rangle = 0
	\]
	so that $C$ vanishes, i.e., the symplectic blow-down occurs on $Z_{\max}$ and this contradicts that $M_{1 - \epsilon} \cong Z_{\max}$.
\end{proof}

\begin{theorem}\label{theorem_3_2}
		Let $(M,\omega)$ be a six-dimensional closed monotone semifree Hamiltonian $S^1$-manifold such that $\mathrm{Crit} H = \{ 1,-1,-3\}$. 
		Then there is a unique possible topological fixed point data given by 
			\begin{table}[H]
				\begin{tabular}{|c|c|c|c|c|}
					\hline
					   &  $(M_0, [\omega_0])$ & $Z_{-3}$ & $Z_{-1}$ & $Z_1$ \\ \hline \hline
					   {\bf (III-2)} & $(\p^2, 3u)$ & {\em pt} & {\em pt} & $\p^2 \# \overline{\p^2}$ \\ \hline
				\end{tabular}
			\end{table}
			\noindent
		Moreover, we have $b_2(M) = 2$ and 
		$
			\langle c_1(TM)^3, [M] \rangle = 56.
		$
\end{theorem}

\begin{proof}
	By Lemma \ref{lemma_3_2}, we only need to prove that 
		$
			\langle c_1(TM)^3, [M] \rangle = 56.
		$
	Using Theorem \ref{theorem_localization}, we obtain
	\[
		\begin{array}{ccl}\vs{0.1cm}
			\ds \int_M c_1^{S^1}(TM)^3 & = & \ds  \sum_{Z \subset M^{S^1}} \int_Z \frac{\left(c_1^{S^1}(TM)|_Z\right)^3}{e_Z^{S^1}} \\ \vs{0.2cm}
								& = & \ds  \frac{(3x)^3}{x^3} + \frac{x^3}{-x^3} + \int_{Z_{\max}} \frac{\left((3u-E_1) + (-x - e(P^-_1)\right)^3) }{-x - e(P^-_{1})} \quad \quad (e(P^-_1) = -u + E_1) \\ \vs{0.1cm}
								& = & 26 + \ds \int_{Z_{\max}} \frac{(4u - 2E_1 - x)^3}{u - E_1 - x} \\ \vs{0.1cm}
								& = & 26 + \ds \int_{Z_{\max}} \frac{(4u - 2E_1 - x)^3 \cdot ((u-E_1)^2 + (u-E_1) x + x^2)}{-x^3} \\ \vs{0.1cm}
								& = & 26 + 30 = 56. 
		\end{array}
	\]
\end{proof}

\begin{example}[Fano variety of type {\bf (III-2)}]\cite[No. 35 in the list in Section 12.3]{IP}\label{example_3_2}
	Let $M$ be the one-point blow-up of $\p^3$ equipped with the monotone K\"{a}hler form $\omega$ with $[\omega ] = c_1(TM)$. 
	(Following Mori-Mukai's notation, we give a special name on $M$ by $V_7$. See \cite{MM}.)
	If we consider a toric structure on $M$ regarding $M$ as the toric blow-up of $\p^3$, the corresponding moment polytope 
	is given in Figure \ref{figure_3_2}.
	
		\begin{figure}[H]
			\scalebox{0.8}{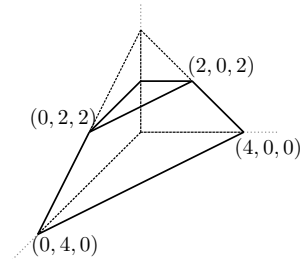} \vs{0.3cm}
			\caption{\label{figure_3_2} Moment polytope for the blow-up of $\p^3$.}
		\end{figure}
		
	\noindent Let $S^1$ be the subgroup of $T^3$ generated by $\xi = (0, -1, 0)$. Then, one can easily see that the action is semifree and the fixed point set of the $S^1$-action consists of 
	\[
		Z_{\min} = \mu^{-1}(0,4,0) = \mathrm{pt}, \quad Z_{-1} = \mu^{-1}(0,2,2) = \mathrm{pt}, \quad\text{and} \quad Z_{\max} = \mu^{-1}(\Delta) \cong \p^2
	\]
	where $\Delta$ is the trapezoid whose vertices are $(4,0,0)$, $(0,0,0)$, $(0,0,2)$, and $(2,0,2)$. So, the $S^1$-action on $M$, together with the balanced moment map
	$\langle \mu, \xi \rangle + 1$, has the same topological fixed point data as
	{\bf (III-2)} in Theorem \ref{theorem_3_2}.
\end{example}
\vs{0.3cm}

\subsection{${\mathrm{Crit} \mathring{H}} = \{0\}$}
~\\

In this case, we have $M_0 \cong \p^2 \cong Z_{\max}$. Let $\mathrm{PD}(Z_0) = au \in H^2(M_0 ; \Z)$ for some integer $a > 0$. 

\begin{lemma}\label{lemma_3_3}
	$Z_0$ is connected. Also, we have $a=1,2,$ or $3$. 
\end{lemma}

\begin{proof}
	Suppose not. If $Z_0^1$ and $Z_0^2$ are disjoint components in $Z_0$,  then $[Z_0^1] = a_1u$ and $[Z_0^2] = a_2u$ for some positive integers $a_1, a_2 \in \Z$. 
	Then $[Z_0^1] \cdot [Z_0^2] = a_1 a_2 \neq 0$ which contradicts that $Z_0^1$ and $Z_0^2$ are disjoint.
	
	For the second statement, using the Duistermaat-Heckman theorem \ref{theorem_DH}. it follows that
	\[
		[\omega_1] = (4-a)u. 
	\]
	Since $4-a>0$, the only possible values of $a$ are $1,2,$ and$3$. 
\end{proof}

\begin{theorem}\label{theorem_3_3}
		Let $(M,\omega)$ be a six-dimensional closed monotone semifree Hamiltonian $S^1$-manifold such that $\mathrm{Crit} H = \{ 1,0,-3\}$. 
		Then the  topological fixed point data is one of the followings :
			\begin{table}[H]
				\begin{tabular}{|c|c|c|c|c|}
					\hline
					   &  $(M_0, [\omega_0])$ & $Z_{-3}$ & $Z_0$ & $Z_1$ \\ \hline \hline
					   {\bf (III-3.1)} & $(\p^2, 3u)$ & {\em pt} & $Z_0 \cong S^2, [Z_0] = u$ & $\p^2$ \\ \hline
					   {\bf (III-3.2)} & $(\p^2, 3u)$ & {\em pt} & $Z_0 \cong S^2, [Z_0] = 2u$ & $\p^2$ \\ \hline
					   {\bf (III-3.3)} & $(\p^2, 3u)$ & {\em pt} & $Z_0 \cong T^2, [Z_0] = 3u$ & $\p^2$ \\ \hline
				\end{tabular}
			\end{table}
			\noindent
		In any case, we have $b_2(M) = 2$ and 
		\[
			\langle c_1(TM)^3, [M] \rangle = 
			\begin{cases}
				54 & \text{\bf (III-3.1)} \\ 
				46 & \text{\bf (III-3.2)} \\
				40 & \text{\bf (III-3.3)}
			\end{cases}
		\]	
\end{theorem}

\begin{proof}
	Since $Z_0$ is connected by Lemma \ref{lemma_3_3}, we apply the adjunction formula \eqref{equation_adjunction} to $Z_0$ so that 
	\[
		 3a = \langle c_1(M_0), [Z_0] \rangle = [Z_0] \cdot [Z_0] + 2 - 2g = a^2 + 2 - 2g, \quad g =~ \text{genus of $Z_0$}.
	\]
	So, we have $g=0$ if $a = 1$ or $2$ and $g = 1$ if $a=3$. This proves the first statement of Theorem \ref{theorem_3_3}. The second assertion ``$b_2(M) = 2$'' easily follows from the fact that 
	the moment map $H$ is perfect Morse-Bott so that the Poincar\'{e} polynomial of $M$ is given by 
	\[
		\begin{array}{ccl}\vs{0.1cm}
			P_t(M) & = & t^0 + t^2(P_t(Z_0) + P_t(Z_{\max})) \\ \vs{0.1cm}
				& = & 1 + t^2(1 + (2-2g)t + t^2 + 1 + t^2 + t^4) \\ \vs{0.1cm}
				& = & 1 + 2t^2 + (2-2g)t^3 + 2t^4 + t^6, \quad \quad (Z_0, Z_{\max} : \text{index~two})
		\end{array}
	\]
	For the final assertion (for Chern numbers), we apply the localization theorem \ref{theorem_localization} : 
	\[
		\begin{array}{ccl}\vs{0.1cm}
			\ds \int_M c_1^{S^1}(TM)^3 & = & \ds  \sum_{Z \subset M^{S^1}} \int_Z \frac{\left(c_1^{S^1}(TM)|_Z\right)^3}{e_Z^{S^1}} \\ \vs{0.2cm}
								& = & \ds  \frac{(3x)^3}{x^3} 
								+ \int_{Z_0} \frac{\left(\mathrm{Vol}(Z_0) q\right)^3) }{(x + b^+q)(-x + b^-q)} 
								+ \int_{Z_{\max}} \frac{\left(3u + (-x - e(P^-_1)\right)^3}{-x - e(P^-_{1})} \\ \vs{0.1cm}
								& = & 27 + \ds \int_{Z_{\max}} \frac{((4-a)u - x)^3}{- x - (a-1)u} \quad \quad (e(P_1^-) = -u + au = (a-1)u)\\ \vs{0.1cm}
								& = & 27 + a^2 - 11a + 37 \\ \vs{0.1cm}
								& = & \begin{cases}
									54 & \text{if $a = 1$} \\
									46 & \text{if $a = 2$} \\
									40 & \text{if $a = 3$} 
								\end{cases}
		\end{array}
	\]
			This completes the proof.
\end{proof}

\begin{example}[Fano variety of type {\bf (III-3)}]\label{example_3_3}
	For each {\bf (III-1.a)} (${\bf a}=1,2,3$), we present a Fano variety $X_{\text{\bf a}}$ equipped with a semifree Hamiltonian $\C^*$-action whose topological fixed point data coincides with {\bf (III-1.a)}
	in Theorem \ref{theorem_3_3}.
	(We will see that each manifold $X_{\text{\bf a}}$ can be obtained by {\em an $S^1$-equivariant blow-up} of $\p^3$ along some smooth curve $Q_{\text{\bf a}}$.)
	
	Following Example \ref{example_3_1}, we consider $\p^3$ with the Fubini-Study form $\omega$ with $[\omega] = c_1(T\p^3) = 4u$. 
	Also, we consider the $S^1$-action induced from the standard $T^3$-action (given by \eqref{equation_standard_action})
	generated by $\xi = (1,1,1)$ so that the fixed point set of the $S^1$-action is given by 
	\[
		\mu^{-1}(0,0,0) = \mathrm{pt}, \quad \mu^{-1}(\Delta) = \{ [x,y,z,0] \} \subset \p^3 \}.	
	\]
	Let $Q_{\text{\bf a}}$ be the smooth curve in $\mu^{-1}(\Delta)$ defined by $\{ [x,y,z,0] ~|~ x^{\text{\bf a}} + y^{\text{\bf a}} + z^{\text{\bf a}} = 0 \}$. 
	Note that the adjuction formula \eqref{equation_adjunction} implies that 
	\[
		Q_{\text{\bf a}} \cong \begin{cases}
			\p^1 & \text{\bf a} = 1, 2 \\
			T^2 &  \text{\bf a} = 3. \\
		\end{cases}
	\]
	If we perform an $S^1$-equivariant symlectic blow up $\p^3$ along $Q_{\text {\bf a}}$, then we obtain a complex manifold $M_{\text{\bf a}}$ with an induced Hamiltonian $S^1$-action.
	It is worth mentioning that 
	\begin{itemize}
		\item $M_{\text{\bf a}}$ is Fano as \cite[No. 33,30,28 in the list in Section 12.3]{IP}, 
		\item the induced action is semifree in the following reason : for a fixed point $z_0 \in Q_{\text{\bf a}}$, let $\mcal{U}$ be an $S^1$-equivariant open neighborhood 
		of $z_0$ with a local complex coordinates $(z_1, z_2, w)$ such that 
			\begin{itemize}
				\item $(z_1, z_2, 0)$ is a local coordinate system of $\p^2 \cong \mu^{-1}(\Delta)$ near $z_0 = (0,0,0)$, 
				\item $(z_1,0,0)$ is a local coordinate system of $Q_{\text{\bf a}}$ near $z_0$
			\end{itemize} 
		where the action can be expressed as 
		\[
			t(z_1, z_2, w) = (z_1, z_2, t^{-1}w), \quad t \in S^1.
		\]
		Then, an $S^1$-equivariant blow-up of $M_{\text{\bf a}}$ along $Q_{\text{\bf a}}$ is locally described as a blow-up 
		of $\mcal{U}$ along a submanifold $\{(z_1, 0,0) \} \cong \C \subset \mcal{U}$ : 
		\[
			\widetilde{\mcal{U}} = \{ \left(z_1, \left( [z_2, w], \lambda z_2, \lambda w \right) \right) \in \C \times (\p^1 \times \C^2)  ~|~ \lambda \in \C  \}
		\]
		where the induced $S^1$-action is given by 
		\[
			t \cdot (z_1, ([z_2, w], \lambda z_2, \lambda w)) =  (z_1, ([z_2, t^{-1} w], \lambda z_2, t^{-1} \lambda w)), \quad t \in S^1.
		\]
		It can be easily verified that the $S^1$-action on $\widetilde{\mcal{U}}$ is semifree (since there is no point having a finite non-trivial stabilizer). Moreover, there are two fixed components
		\[
			\{(z_1, ([1,0], \lambda, 0)) \} \cong \C^2 \quad \text{and} \quad \{(z_1, ([0,1], 0, 0) \} \cong \C
		\]
		where the first one corresponds to an open subset of $Z_{\max}$ ($= Z_1$) and the latter corresponds to an open subset of $Z_0$.
	\end{itemize}
	
		\begin{figure}[H]
			\scalebox{1}{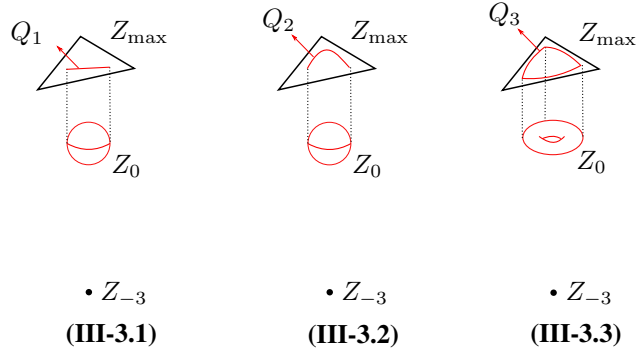} \vs{0.3cm}
			\caption{\label{figure_3_3} $S^1$-equivariant blow-ups of $\p^3$ along $Q_{\text{\bf a}}$ for ${\text{\bf a}} = 1,2,3$.}
		\end{figure}

	Note that we may also choose $Q = \{ [x,y,0,0] \}$ as a $T^3$-invariant rational curve of degree one in $\mu^{-1}(\Delta)$. Then the toric blow-up of $\p^3$ along 
	$Q$ inherits a toric structure and the induced $S^1$-action also has a topological fixed point data that coincides with {\bf (III-3.1)}. See Figure \ref{figure_3_3_toric}.

		\begin{figure}[H]
			\scalebox{1}{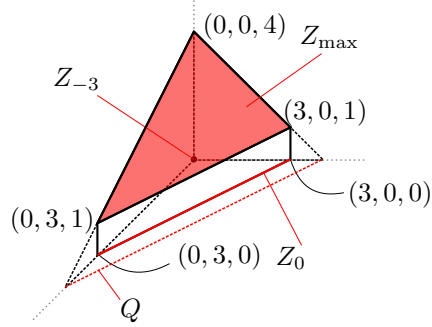} \vs{0.3cm}
			\caption{\label{figure_3_3_toric} Toric blow-up of $\p^3$ along $Q$.}
		\end{figure}	
	
\end{example}

\subsection{${\mathrm{Crit} \mathring{H}} = \{ -1,0\}$}
~\\

Let $|Z_{-1}| = k \in \Z_+$ be the number of fixed points of index two and 
\[
	\mathrm{PD}(Z_0) = au + b_1E_1 + \cdots + b_kE_k, \quad a, b_1, \cdots, b_k \in \Z.
\]
From Lemma \ref{lemma_Euler_class}, we obtain
\[
	e(P_1^-) = -u + \sum_{i=1}^k E_i + \mathrm{PD}(Z_0) = (a-1)u + \sum_{i=1}^k (b_i + 1)E_i
\]
Also, the Duistermaat-Heckman theorem \ref{theorem_DH} implies that 
\begin{equation}\label{equation_3_4_omega}
	[\omega_1] = (3u - \sum_{i=1}^k E_i) - e(P_1^-) = (4-a)u - \sum_{i=1}^k (2 + b_i) E_i
\end{equation}

\begin{lemma}\label{lemma_3_4_coef}
	The following inequalities hold : 
	\[
		a < 4, \quad b_i > -2, \quad a + b_i + b_j < 0, \quad (4-a)^2 - \sum_{i=1}^k (2 + b_i)^2 > 0, \quad 3a + \sum_{i=1}^k b_i > 0 
	\]
	for  $i,j = 1,\cdots, k$ and $i \neq j$.
\end{lemma}

\begin{proof}	
	By the Duistermaat-Heckman theorem \ref{theorem_DH}, we obtain 
	\[
		[\omega_t] = (3 + (1-a)t ) u - \sum_{k=1}^k (1 + (b_i + 1)t) E_i, \quad 0 \leq t \leq 1. 
	\] 
	Note that $3+(1-a)t$ should never vanish on $(0,1)$ since $\langle [\omega_t]^2, [M_0] \rangle > 0$ for every $t \in [0,1]$. Thus we get $a > 4$. For the second and third inequalities, 
	we consider exceptional classes $E_i$ or $u - E_i - E_j$ with $i \neq j$. As the ``(-1)-curve theorem'' by Li and Liu \cite[Theorem A]{LL} guarantees the existence of a symplectic sphere
	representing $E_i$ or $u - E_i - E_j$, the symplectic volume of each class should be positive, that is, 
	\[
		\langle [\omega_t], E_i \rangle = 2 + b_i > 0 \quad \text{and} \quad \langle [\omega_t], u - E_i - E_j \rangle = -a - b_i - b_j > 0.
	\]
	The last two inequalities immediately follow from the fact that $\langle [\omega_1]^2, [M_1] \rangle > 0$ and $\langle [\omega_0], [Z_0] \rangle > 0$.
\end{proof}

\begin{lemma}\label{lemma_3_4_a}
	We have $a \geq 0$. In particular, if $k > 1$, then $a \leq 1.$
\end{lemma}
	
\begin{proof}
	Suppose that $a < 0$. Then there is a connected component, say $Z_0^1$, of $Z_0$ such that the coefficient, say $a_1$, of $u$ in $\mathrm{PD}(Z_0^1)$ is negative. 
	If $[Z_0^1] \cdot [Z_0^1] < 0$, then by the adjunction formula \eqref{equation_adjunction}
	\[
		\langle c_1(TM_0), [Z_0^1] \rangle = [Z_0^1] \cdot [Z_0^1] + 2 - 2g
	\]
	implies that $\langle c_1(TM_0), [Z_0^1] \rangle = 1$, $[Z_0^1] \cdot [Z_0^1] = -1$, and $g = 0$ since $\langle c_1(TM_0), [Z_0^1] \rangle = \langle [\omega_0], [Z_0^1] \rangle$ is a positive integer.
	It means that $Z_0^1$ is an exceptional sphere so that, by Lemma \ref{lemma_list_exceptional}, $\mathrm{PD}(Z_0^1)$ cannot have a negative coefficient of $u$. So, we have 
	$[Z_0^1] \cdot [Z_0^1] \geq 0$. 
	
	Now, let $\mathrm{PD}(Z_0^1) = a_1 u + b_1^1 E_1 + \cdots + b_k^1 E_k$. 
	By the last inequality of Lemma \ref{lemma_3_4_coef}, there exists some $b_i^1 > 0$, which implies that $[Z_0^1] \cdot E_i = - b_i^1 < 0$. 
	This situation exactly fits into the case of T-J. Li's Theorem \cite[Corollary 3.10]{Li} which states that any symplectic surface with non-negative self-intersection number 
	should intersects the exceptional class $E_i$ non-negatively.
	(See also \cite[Theorem 5.1]{W}.) Consequently, $[Z_0^1] \cdot E_i = b_i$ cannot be negative and 	this leads to a contradiction. So, we have $a \geq 0$.
	
	For the second statement, it follows from Lemma \ref{lemma_3_4_coef} that
	\[
		b_i \geq -1~\text{for every $i$} \quad \Rightarrow \quad a - 2 \leq a + b_i + b_j \leq -1 \quad \Rightarrow \quad a \leq 1.
	\]
\end{proof}

\begin{lemma}\label{lemma_3_4_k}
	The only possible values of $k$ are $1$ and $2$. 
\end{lemma}

\begin{proof}
	Assume to the contrary that $k > 2$. Then, 
	\begin{itemize}
		\item $b_i \leq 0$ for every $i$ (by Lemma \ref{lemma_3_4_a} and the second and third inequalities of Lemma \ref{lemma_3_4_coef},)
		\item $a > 0$ (by the last inequality of Lemma \ref{lemma_3_4_coef}, and $b_i \leq 0$,)
	\end{itemize}
	From the second part of Lemma \ref{lemma_3_4_a}, we have $a=1$. Moreover  $b_i = -1$ for every $i$ 
	(since $a+ b_i + b_j < 0$ and $-1 \leq b_i \leq 0$.) Then, by the last inequality of Lemma \ref{lemma_3_4_coef}, 
	\[
		3 - k = 3a + \sum_{i=1}^k b_i \geq 1,
	\]
	which implies that $k \leq 2$. 
\end{proof}

\begin{theorem}\label{theorem_3_4}
		Let $(M,\omega)$ be a six-dimensional closed monotone semifree Hamiltonian $S^1$-manifold such that $\mathrm{Crit} H = \{ 1,0,-1,-3\}$. 
		Then the  topological fixed point data is one of the followings :
			\begin{table}[H]
				\begin{tabular}{|c|c|c|c|c|c|}
					\hline
					   &  $(M_0, [\omega_0])$ & $Z_{-3}$ & $Z_{-1}$ & $Z_0$ & $Z_1$ \\ \hline \hline
					   {\bf (III-4.1)} & $(\p^2 \# \overline{\p^2}, 3u - E_1)$ & {\em pt} & {\em pt} & $Z_0 \cong S^2, [Z_0] = E_1$ & $\p^2 \# \overline{\p^2}$ \\ \hline
					   {\bf (III-4.2)} & $(\p^2 \# \overline{\p^2}, 3u - E_1)$ & {\em pt} & {\em pt} & $Z_0 \cong S^2, [Z_0] = u - E_1$ & $\p^2 \# \overline{\p^2}$ \\ \hline
					   {\bf (III-4.3)} & $(\p^2 \# \overline{\p^2}, 3u - E_1)$ & {\em pt} & {\em pt} & $Z_0 \cong S^2, [Z_0] = u$ & $\p^2 \# \overline{\p^2}$ \\ \hline
					   {\bf (III-4.4)} & $(\p^2 \# \overline{\p^2}, 3u - E_1)$ & {\em pt} & {\em pt} & $Z_0 \cong S^2, [Z_0] = 2u - E_1$ & $\p^2 \# \overline{\p^2}$ \\ \hline
					   {\bf (III-4.5)} & $(\p^2 \# 2\overline{\p^2}, 3u - E_1-E_2)$ & {\em pt} & {\em 2 pts} & $Z_0 \cong S^2, [Z_0] = u - E_1 - E_2$ & $\p^2 \# 2\overline{\p^2}$ \\ \hline
				\end{tabular}
			\end{table}
			\noindent
		Also, we have 
		\[
			b_2(M) = \begin{cases}
				3 & \text{\bf (III-4.1$\sim$4)} \\ 
				4 & \text{\bf (III-4.5)} \\ 
			\end{cases}
			\quad \text{and} \quad 
			\langle c_1(TM)^3, [M] \rangle = 
			\begin{cases}
				50 & \text{\bf (III-4.1)} \\ 
				50 & \text{\bf (III-4.2)} \\
				46 & \text{\bf (III-4.3)} \\
				42 & \text{\bf (III-4.4)} \\
				46 & \text{\bf (III-4.5)} \\
			\end{cases}
		\]	
\end{theorem}

\begin{proof}
	We divide the proof into two cases : $k = 1$ and $k=2$. \vs{0.2cm}
	\begin{itemize}
		\item {\bf Case I : ${\bf k = 1}$.} Recall that Lemma \ref{lemma_3_4_coef}, together with Lemma \ref{lemma_3_4_a}, says that we have 
		\[
			0 \leq a \leq 3, \quad b_1 \geq -1, \quad (4-a)^2 - (2+b_1)^2 \geq 1, \quad 3a + b_1 \geq 1.
		\]
		Thus the list of all possible pairs $(a, b_1)$ is as below : 
		\[
			(a, b_1) = (0,1), (1,-1), (1,0), (2,-1), 
		\]	
		or equivalently, $\mathrm{PD}(Z_0) = E_1, u - E_1, u, 2u - E_1$ as listed in Theorem \ref{theorem_3_4}. Moreover, $Z_0$ is connected in any case in the following reasons. \vs{0.1cm}
		\begin{itemize}
			\item If $\mathrm{PD}(Z_0) = E_1$, then $\langle [\omega_0], [Z_0] \rangle = 1$.
			\item If $\mathrm{PD}(Z_0) = u - E_1$ or $\mathrm{PD}(Z_0) = u$, let us suppose that $Z_0$ is disconnected. Then 
			there exists a connected component of $Z_0$, say $Z_0^1$, such that $\mathrm{PD}(Z_0^1) = pE_1$ for some $p \in \Z_+$ 
			(because of Lemma \ref{lemma_3_4_a} and the last inequality in Lemma \ref{lemma_3_4_coef}). Moreover, since $Z_0^1$ does not intersect other components, we have 
			\[
				pE_1 \cdot (\mathrm{PD}(Z_0) - pE_1) = 0 
			\]  
			which is impossible unless $p = 0$. \vs{0.1cm}
			\item If $\mathrm{PD}(Z_0) = 2u - E_1$, we assume that $Z_0$ is disconnected. 
			Then we can easily see that $Z_0$ should consist of exactly two components, namely $Z_0^1$ and $Z_0^2$,
			and $\mathrm{PD}(Z_0^i) = u + p_iE_1$ for some $p_i \in \Z$ where $i=1,2$. 
			(Otherwise there is a component whose dual class is of the form $pE_1$ for some $p \in \Z_+$ and this is impossible since 
			$pE_1 \cdot (\mathrm{PD}(Z_0) - pE_1) \neq 0$.) Moreover, since $[Z_0^1] \cdot [Z_0^2] = 0$, we have $p_1p_2 = -1$. In other words, we have 
			$p_1 = 1$ and $p_2 = -1$ (by rearranging the order of $Z_0^1$ and $Z_0^2$ if necessary). However, there cannot exist a symplectic surface representing class $u + E_1$ 
			by \cite[Corollary 3.10]{Li} since it has non-negative (actually zero) 
			self-intersection number and intersect the stable class $E_1$ negatively. Therefore $Z_0$ is connected. \vs{0.1cm}
		\end{itemize}
	
		Now, we apply the adjunction formula \eqref{equation_adjunction} to each case
		\[
			\langle c_1(M_0), [Z_0] \rangle = [Z_0] \cdot [Z_0] + 2 - 2g, \quad c_1(M_0) = 3u - E_1
		\]
		where $g$ is the genus of $Z_0$. Then,\vs{0.1cm}
		\begin{itemize}
			\item If $\mathrm{PD}(Z_0) = E_1$, then $1 = -1 + 2 - 2g$. 
			\item If $\mathrm{PD}(Z_0) = u - E_1$, then $2 = 0 + 2 - 2g$. 
			\item If $\mathrm{PD}(Z_0) = u$, then $3 = 1 + 2 - 2g$.
			\item If $\mathrm{PD}(Z_0) = 2u  - E_1$, then $5 = 3 + 2 - 2g$. \vs{0.1cm}
		\end{itemize}
		In either case, we have $g=0$ and hence $Z_0 \cong S^2$. This proves the first claim of Theorem \ref{theorem_3_4}. The claim $b_2(M) = 3$ can be obtained directly 
		by computing the Poincar\'{e} polynomial of $M$ in terms of fixed components. 
		
		It remains to compute the Chern numbers for each case. Applying the localization theorem \ref{theorem_localization}, we get
		\[
			\begin{array}{ccl}\vs{0.1cm}
				\ds \int_M c_1^{S^1}(TM)^3 & = & \ds  \sum_{Z \subset M^{S^1}} \int_Z \frac{\left(c_1^{S^1}(TM)|_Z\right)^3}{e_Z^{S^1}} \\ \vs{0.2cm}
									& = & \ds  \frac{(3x)^3}{x^3} + \frac{x^3}{-x^3} 
									+ \int_{Z_0} \frac{\left(\mathrm{Vol}(Z_0)q\right)^3 }{(x + b^+q)(-x + b^-q)} 
									+ \int_{Z_{\max}} \frac{\left((3u - E_1 + (-x - e(P^-_1)\right)^3) }{-x - e(P^-_{1})} \\ \vs{0.1cm}
									& = & 26 + \ds \int_{Z_{\max}} \frac{\left((3u - E_1 + (-x - e(P^-_1)\right)^3) }{-x - e(P^-_{1})} \quad \quad (e(P_1^-) = (a-1)u + (b_1 + 1)E_1)\\ \vs{0.1cm}
									& = & 26 + 24 -3(3(a-1) + (b_1+1)) + (a-1)^2 - (b_1+1)^2 \\ \vs{0.1cm}
									& = & \begin{cases}
										50 & \text{if $(a, b_1) = (0,1)$} \\
										50 & \text{if $(a, b_1) = (1,-1)$} \\
										46 & \text{if $(a, b_1) = (1,0)$} \\
										42 & \text{if $(a, b_1) = (2,-1)$} \\
									\end{cases}
			\end{array}
		\]
		\vs{0.2cm} 
		
		\item {\bf Case II : ${\bf k = 2}$.} Note that we have $0 \leq a \leq 1$ by Lemma \ref{lemma_3_4_a}. If $a = 0$, then two inequalities 
		\[
			a + b_1 + b_2 < 0 \quad \text{and} \quad 3a + b_1 + b_2 > 0
		\]
		in Lemma \ref{lemma_3_4_coef} contradict each other, that is, we have $a = 1$ (and hence $b_1 = b_2 = -1$.) Therefore, the only possible triple
		$(a,b_1,b_2)$ is $(1,-1,-1)$, or equivalently, $\mathrm{PD}(Z_0) = u -E_1 - E_2$. The symplectic area of $Z_0$ is then $1$ so that $Z_0$ is connected. 
		Also the adjunction formula \eqref{equation_adjunction} implies that $Z_0 \cong S^2$. Furthermore, we have 
		\[
			\begin{array}{ccl}\vs{0.1cm}
				\ds \int_M c_1^{S^1}(TM)^3 & = & \ds  \sum_{Z \subset M^{S^1}} \int_Z \frac{\left(c_1^{S^1}(TM)|_Z\right)^3}{e_Z^{S^1}} \\ \vs{0.2cm}
									& = & \ds  \frac{(3x)^3}{x^3} + \frac{2x^3}{-x^3} 
									+ \int_{Z_0} \frac{\left(\mathrm{Vol}(Z_0)q\right)^3 }{(x + b^+q)(-x + b^-q)} 
									+ \int_{Z_{\max}} \frac{\left((3u - E_1 - E_2 + (-x - e(P^-_1)\right)^3) }{-x - e(P^-_{1})} \\ \vs{0.1cm}
									& = & 25 + \ds \int_{Z_{\max}} \frac{\left((3u - E_1 - E_2 + (-x - e(P^-_1)\right)^3) }{-x - e(P^-_{1})} \quad \quad (e(P_1^-) = 0) 
									\\ \vs{0.1cm}
									& = & 25 + \ds \int_{Z_{\max}} 3(3u - E_1 - E_2)^2 = 46.
			\end{array}
		\]
		The statement $b_2(M) = 4$ follows from the perfectness of a moment map (as a Morse-Bott function). 
	\end{itemize}
	
\end{proof}

\begin{example}[Fano varieties of type {\bf (III-4)}]\label{example_3_4}
	We follow Mori-Mukai's notation in \cite{MM}. 
	Let $V_7$ be the one-point toric blow-up of $\p^3$ so that the corresponding moment polytope (with respect to $\omega$ with $[\omega] = c_1(TV_7)$) is given in Figure \ref{figure_3_4}.
	See also Example \ref{example_3_2} (case {\bf (III-2)}). 
	
		\begin{figure}[H]
			\scalebox{0.8}{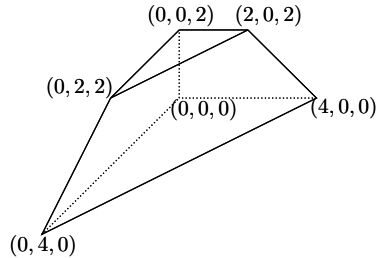} \vs{0.3cm}
			\caption{\label{figure_3_4} $V_7$ : Toric blow-up of $\p^3$ at a fixed point corresponding to the vertex $(0,0,4)$.}
		\end{figure}	
	\noindent		
	We will see that Fano varieties of type {\bf (III-4.1$\sim$3)} can be obtained as (toric) blow-ups of $V_7$ along some rational curves.
	We also construct a Fano variety of type {\bf (III-4.4)} as a blow-up of $V_7$ which is not toric.  
	Moreover, an example of a Fano variety of type {\bf (III-4.5)} can be given as the blow-up of $Y$ along two rational curves where $Y$ is the toric blow-up of $\p^3$ along a torus invariant line,  
	see Figure \ref{figure_3_4_5}.

	\begin{enumerate}
		\item {\bf Case (III-4.1) \cite[No. 29 in the list in Section 12.4]{IP}} : 
		Let $M$ be the toric blow-up of $V_7$ along the torus invariant sphere which is the preimage of the edge ${\bf e}_1$ connecting 
		$(0,2,2)$ and $(2,0,2)$ in Figure \ref{figure_3_4_1}, i.e., the blow-up of a line lying on the exceptional divisor of $V_7$.
		Then the corresponding moment polytope can be illustrated as in Figure \ref{figure_3_4_1}. Let $S^1$ be the subgroup of $T^3$ generated 
		by $\xi = (1,1,1)$. Then we can easily check (by calculating the inner products of each primitive edge vectors and $\xi$) that the induced $S^1$-action is semifree. Also, 
		with respect to the balanced moment map $H = \langle \mu, \xi \rangle - 3$, the fixed point set for the $S^1$-action consists of 
		\[
			\begin{array}{ll}
				Z_{-3} = \mu^{-1}((0,0,0)) = \mathrm{pt}, &  Z_{-1} = \mu^{-1}((0,0,2)) = \mathrm{pt}, \\
				Z_{0} = \mu^{-1}(e) \cong S^2,  & Z_1 = \mu^{-1}(\Delta)
			\end{array}
		\]
		where $e$ is the edge connecting $(0,1,2)$ and $(1,0,2)$ and $\Delta$ is the trapezoid whose vertex set is given by 
		$\{ (3,0,1), (0,3,1), (0,4,0), (4,0,0) \}$ in Figure \ref{figure_3_4_1} (on the right.)
		
		\begin{figure}[H]
			\scalebox{0.9}{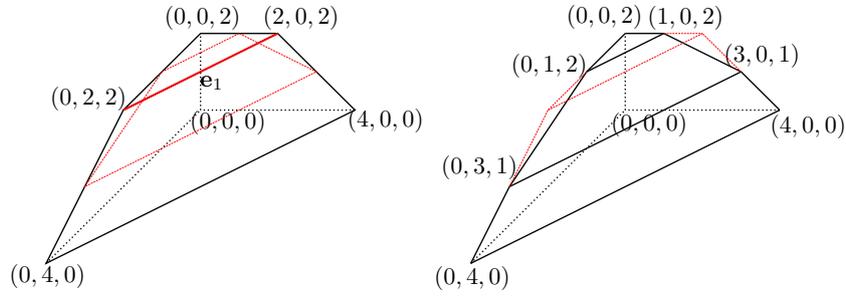} \vs{0.3cm}
			\caption{\label{figure_3_4_1} {\bf (III-4.1)} Toric blow-up of $V_7$ along the sphere corresponding to the edge ${\bf e}_1$.}
		\end{figure}			
		
		\item {\bf Case (III-4.2) \cite[No. 30 in the list in Section 12.4]{IP}} : Let $M$ be the toric blow-up of $V_7$ along the sphere corresponding to the edge ${\bf e}_2$ 
		in Figure \ref{figure_3_4_2} (on the left) where the corresponding moment polytope is described on the right. 
		(In other words, $M$ is the blow-up of $V_7$ along a line passing through the exceptional divisor of $V_7$.)
		Let $S^1$ be the subgroup of $T^3$ generated 
		by $\xi = (-1,0,0)$. The induced $S^1$-action is semifree and has the balanced moment map $H = \langle \mu, \xi \rangle +1$. Also the fixed point set for the $S^1$-action is given by
		\[
			\begin{array}{ll}
				Z_{-3} = \mu^{-1}((4,0,0)) = \mathrm{pt}, &  Z_{-1} = \mu^{-1}((2,0,2)) = \mathrm{pt}, \\
				Z_{0} = \mu^{-1}(e) \cong S^2,  & Z_1 = \mu^{-1}(\Delta)
			\end{array}
		\]
		where $e$ is the edge connecting $(1,0,2)$ and $(1,0,0)$ and $\Delta$ is the trapezoid whose vertex set is given by 
		$\{ (0,4,0), (0,1,0), (0,1,2), (0,2,2) \}$ in Figure \ref{figure_3_4_2} on the right. 

		\begin{figure}[H]
			\scalebox{0.9}{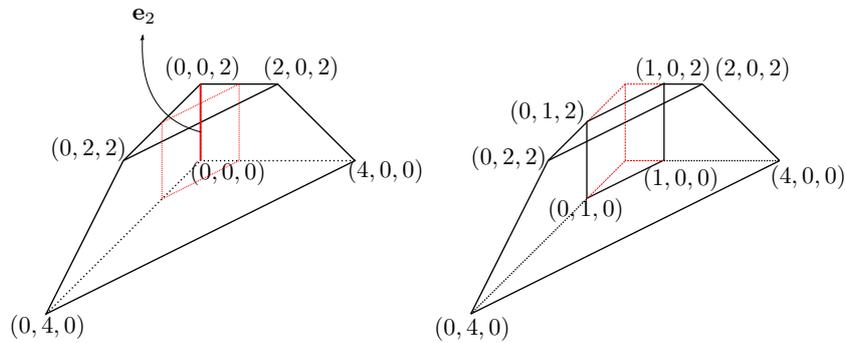} \vs{0.3cm}
			\caption{\label{figure_3_4_2} {\bf (III-4.2)} Toric blow-up of $V_7$ along the sphere corresponding to the edge ${\bf e}_2$.}
		\end{figure}			
		
		\item {\bf Case (III-4.3) \cite[No. 26 in the list in Section 12.4]{IP}} : Now, let $M$ be the toric blow-up of $V_7$ along the sphere corresponding to the edge ${\bf e}_3$ 
		in Figure \ref{figure_3_4_3}, that is, $M$ is the blow-up of $V_7$ along a line not intersecting the exceptional divisor of $V_7$.
		The corresponding moment polytope is described in Figure \ref{figure_3_4_3} on the right. 
		Let $S^1$ be the subgroup of $T^3$ generated 
		by $\xi = (1,1,1)$. Then we can easily check (by looking up Figure \ref{figure_3_4_3}) that the induced $S^1$-action is semifree and has the balanced moment map $H = \langle \mu, \xi \rangle -3$. 
		Also the fixed point set for the $S^1$-action is listed as 
		\[
			\begin{array}{ll}
				Z_{-3} = \mu^{-1}((0,0,0)) = \mathrm{pt}, &  Z_{-1} = \mu^{-1}((0,0,2)) = \mathrm{pt}, \\
				Z_{0} = \mu^{-1}(e) \cong S^2,  & Z_1 = \mu^{-1}(\Delta)
			\end{array}
		\]
		where $e$ is the edge connecting $(3,0,0)$ and $(0,3,0)$ and $\Delta$ is the trapezoid whose vertex set is given by 
		$\{ (0,3,1), (3,0,1), (2,0,2), (0,2,2) \}$ in Figure \ref{figure_3_4_3}. 
		
		\begin{figure}[H]
			\scalebox{0.9}{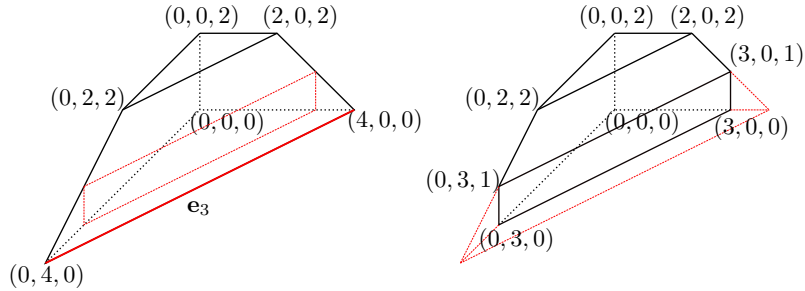} \vs{0.3cm}
			\caption{\label{figure_3_4_3} {\bf (III-4.3)} Toric blow-up of $V_7$ along the sphere corresponding to the edge ${\bf e}_3$.}
		\end{figure}			

		\item {\bf Case (III-4.4) \cite[No. 23 in the list in Section 12.4]{IP}} : 
		Consider $V_7$ as a toric variety and let $S^1$ be the subgroup of $T^3$ generated by $\xi = (-1,0,0)$. Then the induced $S^1$-action on $V_7$ is semifree and 
		the fixed point set is given by
		\[
			Z_{-3} = \mathrm{pt}, \quad Z_{-1} = \mathrm{pt}, \quad Z_{1} = \mu^{-1}(\Delta) \cong \p^2 \# \overline{\p^2}.
		\]
		Let $Q \subset Z_1 \subset V_7$ be the proper transformation of a conic in the hyperplane $\p^2 \subset \p^3$ passing through the blown-up point, i.e., the center of $V_7$.
		As a symplectic submanifold of $V_7$, one can describe $Q$ as a smoothing of two spheres in $Z_1$ representing $u$ and $u-E$ respectively. 
		Then, similar to the case of {\bf (III-3)} in Example \ref{example_3_3}, we perform an $S^1$-equivariant blow up of $V_7$ along $Q$ and we denote the resulting manifold by $M$.
		(Note that the procedure of the blowing-up construction is exactly the same as described in Example \ref{example_3_3}.)
		As appeared in \cite[Table 3. no.23]{MM}, $M$ is a smooth Fano variety and the induced $S^1$-action on $M$ has a fixed point set which coincides with {\bf (III-4.4)}. 
		See Figure \ref{figure_3_4_non}.
		
		\begin{figure}[H]
			\scalebox{1}{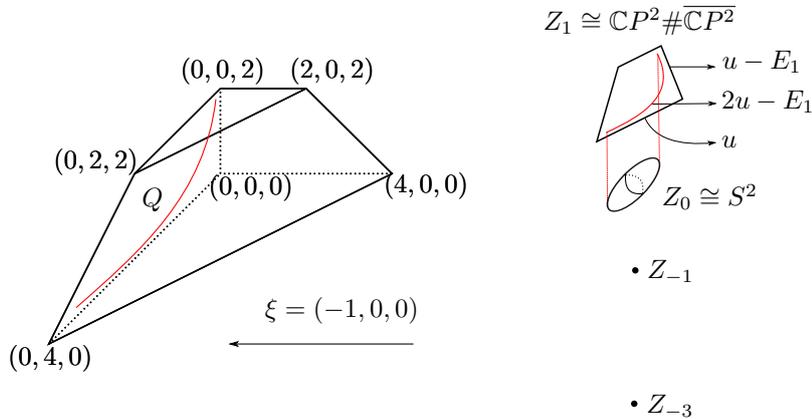} \vs{0.3cm}
			\caption{\label{figure_3_4_non} {\bf (III-4.4)} (Non-toric) blow-up of $V_7$ along a conic.}
		\end{figure}			

		\item {\bf Case (III-4.5) \cite[No. 12 in the list in Section 12.5]{IP}} : 
		In this case, we consider $Y$, the toric blow-up of $\p^3$ along a torus invariant line whose moment polytope is 
		described in Figure \ref{figure_3_4_5} on the left. Then, we let 
		$M$ be the toric blow-up of $Y$ along two disjoint spheres corresponding to the edges ${\bf C}_1$ and ${\bf C}_2$. 
		The corresponding moment polytope is given in Figure \ref{figure_3_4_5} on the right. 
		Let $S^1$ be the subgroup of $T^3$ generated 
		by $\xi = (-1,0,0)$. Then it follows that the induced $S^1$-action is semifree and has the balanced moment map $H = \langle \mu, \xi \rangle +1$. 
		Moreover, the fixed point set for the $S^1$-action is given by
		\[
			\begin{array}{ll}
				Z_{-3} = \mu^{-1}((4,0,0)) = \mathrm{pt}, &  Z_{-1} = \mu^{-1}((2,0,2)) ~\cup ~ \mu^{-1}((2,0,0)) = \mathrm{2 ~pts}, \\
				Z_{0} = \mu^{-1}(e) \cong S^2,  & Z_1 = \mu^{-1}(\Delta)
			\end{array}
		\]
		where $e$ is the edge connecting $(1,0,1)$ and $(1,0,2)$ and $\Delta$ is the five gon whose vertex set is given by 
		$\{ (0,4,0), (0,2,0), (0,3,1), (0,1,2), (0,2,2) \}$ in Figure \ref{figure_3_4_5}. 
		
		\begin{figure}[H]
			\scalebox{0.9}{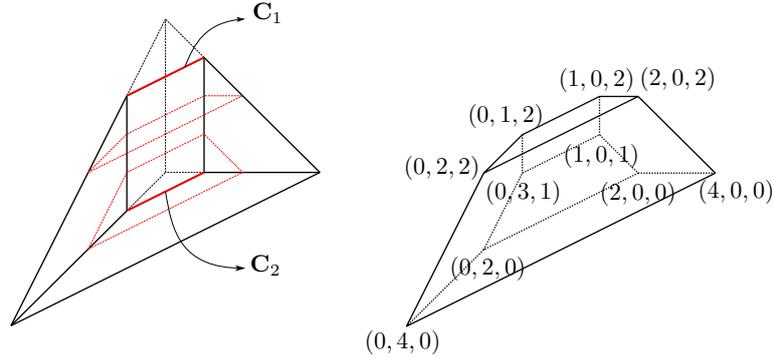} \vs{0.3cm}
			\caption{\label{figure_3_4_5} {\bf (III-4.5)} Toric blow-up of $Y$ along two spheres corresponding to the edge ${\bf C}_1$ and ${\bf C}_2$.}
		\end{figure}					
		
	\end{enumerate}		
				
\end{example}

\section{Main Theorem}
\label{secMainTheorem}

	In this section, we prove our main theorem as follows.
	
	\begin{theorem}[Theorem \ref{theorem_main}]
		Let $(M,\omega)$ be a six-dimensional closed monotone symplectic manifold equipped with a semifree Hamiltonian 
		circle action. Suppose that the maximal or the minimal fixed component of the action is an isolated point. 
		Then $(M,\omega)$ is $S^1$-equivariantly symplectomorphic to some 
		K\"{a}hler Fano manifold with some holomorphic Hamiltonian circle action. 
	\end{theorem}

	We notice that, according to our classification result of topological fixed point data, any reduced space of $(M,\omega)$ in Theorem \ref{theorem_main} is either 
	$\p^2$, $\p^1 \times \p^1$, or $\p^2 \#~k~ \overline{\p^2}$ for $1 \leq k \leq 3$. See Table \ref{table_list}. The following theorems then imply that those spaces 
	are symplectically rigid (in the sense of \cite[Definition 2.13]{McD2} or \cite[Definition 1.4]{G}). (See also Section \ref{secFixedPointData}.)

	\begin{table}[h]
		\begin{tabular}{|c|c|c|c|c|c|c|c|c|c|c|}
			\hline
			    & $(M_0, [\omega_0])$ & $Z_{-3}$ & $Z_{-1}$ &  $Z_0$ & $Z_1$ & $Z_2$ & $Z_3$ & $c_1^3$ \\ \hline \hline
			    {\bf (I-1)} & $(\p^2, 3u)$ & pt &  & $Z_0 \cong S^2$, $[Z_0] = 2u$  &  & & pt &54 \\ \hline    
			   {\bf (I-2)} & $(\p^2 \# 3 \overline{\p^2}, 3u - E_1 - E_2 - E_3)$ &pt & 3 ~pts & & 3 ~pts & & pt & 48 \\ \hline    
			   
			   {\bf (I-3)} & $(\p^2 \# \overline{\p^2}, 3u - E_1)$ & pt & pt &  \makecell{ $Z_0 = Z_0^1 ~\dot \cup ~ Z_0^2$ \\
			    $Z_0^1 \cong Z_0^2 \cong S^2$ \\ $[Z_0^1] = [Z_0^2] = u - E_1$} & pt & & pt & 52\\ \hline    
			    
			   {\bf (II-3.1)} &  $(\p^2 \# \overline{\p^2}, 3u - E_1)$ & pt & pt &  \makecell{ $Z_0 = S^2$ \\
			    $[Z_0] = E_1$} & & $S^2$ & & 62\\ \hline    				    
			   {\bf (II-3.2)} &  $(\p^2 \# \overline{\p^2}, 3u - E_1)$ & pt & pt &  \makecell{ $Z_0 = S^2$ \\
			    $[Z_0] = u $} & & $S^2$ & & 54\\ \hline    				    
			   {\bf (II-3.3)} &  $(\p^2 \# \overline{\p^2}, 3u - E_1)$ & pt & pt &  \makecell{ $Z_0 = S^2$ \\
			    $[Z_0] = 2u - E_1$} & & $S^2$ & & 46 \\ \hline    			
			    
			   {\bf (II-4.1)} & $(\p^2 \# 2\overline{\p^2}, 3u - E_1 - E_2)$ & pt & 2 pts & \makecell{ $Z_0 = Z_0^1 ~\dot \cup ~ Z_0^2$ \\
				    $Z_0^1 \cong Z_0^2 \cong S^2$ \\ $[Z_0^1] = u - E_1$ \\ $[Z_0^2] = u - E_1 - E_2$} & pt & $S^2$ & & 44\\ \hline    
			   {\bf (II-4.2)} &  $(\p^2 \# 3\overline{\p^2}, 3u - E_1 - E_2 - E_3)$ & pt & 3 pts & \makecell{ $Z_0 = S^2$ \\		
		    		    $[Z_0] = u - E_2 - E_3$} & 2pts & $S^2$ & &42\\ \hline    	    
		    		    
			   {\bf (III-1)} & $(\p^2, 3u)$ & pt & & & $\p^2$ & &  & 64\\ \hline
			   {\bf (III-2)} & $(\p^2, 3u)$ & pt & pt & & $\p^2 \# \overline{\p^2}$ & & &56 \\ \hline
			   {\bf (III-3.1)} & $(\p^2, 3u)$ & pt & & $Z_0 \cong S^2, [Z_0] = u$ & $\p^2$ & & &54\\ \hline
			   {\bf (III-3.2)} & $(\p^2, 3u)$ & pt & & $Z_0 \cong S^2, [Z_0] = 2u$ & $\p^2$ & & & 46\\ \hline
			   {\bf (III-3.3)} & $(\p^2, 3u)$ & pt & & $Z_0 \cong T^2, [Z_0] = 3u$ & $\p^2$ & & & 40\\ \hline
			   
			   {\bf (III-4.1)} & $(\p^2 \# \overline{\p^2}, 3u - E_1)$ & pt & pt  & $Z_0 \cong S^2, [Z_0] = E_1$ & $\p^2 \# \overline{\p^2}$ & & & 50\\ \hline
			   {\bf (III-4.2)} & $(\p^2 \# \overline{\p^2}, 3u - E_1)$ & pt & pt  & \makecell{ $Z_0 \cong S^2$,\\ $ [Z_0] = u - E_1$} & $\p^2 \# \overline{\p^2}$ & & & 50\\ \hline
			   {\bf (III-4.3)} & $(\p^2 \# \overline{\p^2}, 3u - E_1)$ & pt & pt & $Z_0 \cong S^2, [Z_0] = u$ & $\p^2 \# \overline{\p^2}$ & & &46 \\ \hline
			   {\bf (III-4.4)} & $(\p^2 \# \overline{\p^2}, 3u - E_1)$ & pt & pt & \makecell{$Z_0 \cong S^2$, \\ $[Z_0] = 2u - E_1$} & $\p^2 \# \overline{\p^2}$ & & &42 \\ \hline
			   {\bf (III-4.5)} & $(\p^2 \# 2\overline{\p^2}, 3u - E_1-E_2)$ & pt & 2 pts & \makecell{$Z_0 \cong S^2$, \\ $[Z_0] = u - E_1 - E_2$} & $\p^2 \# 2\overline{\p^2}$ & & &46\\ \hline
		\end{tabular}
		\vs{0.3cm}
		\caption {List of topological fixed point data} \label{table_list} 
	\end{table}

	\begin{theorem}\cite[Theorem 1.2]{McD4}\label{theorem_uniqueness}
		Let $M$ be a blow-up of a rational or a ruled symplectic four manifold. Then any two cohomologous and deformation equivalent\footnote{Two symplectic forms $\omega_0$ and $\omega_1$
		are said to be {\em deformation equivalent} if there exists a family of symplectic forms $\{ \omega_t  ~|~  0 \leq t \leq 1 \}$ connecting $\omega_0$ and $\omega_1$. We also say that 
		$\omega_0$ and $\omega_1$ are {\em isotopic} if such a family can be chosen such that $[\omega_t]$ is a constant path in $H^2(M; \Z)$.}
		symplectic forms on $M$ are isotopic.
	\end{theorem}

	\begin{theorem}\cite[Lemma 4.2]{G}\label{theorem_symplectomorphism_group}
	For any of the following symplectic manifolds, the group of symplectomorphisms  which act trivially on homology is path-connected. 
		\begin{itemize}
			\item $\p^2$ with the Fubini-Study form. \cite[Remark in p.311]{Gr}
			\item $\p^1 \times \p^1$ with any symplectic form. \cite[Theorem 1.1]{AM}
			\item $\p^2 \# ~k~\overline{\p^2}$ with any blow-up symplectic form for $k \leq 3$. \cite[Theorem 1.4]{AM}, \cite{E}, \cite{LaP}, \cite{Pin}. 
		\end{itemize}
	\end{theorem}
	
	From now on, we discuss how a topological fixed point data determines a fixed point data. 
	Note that a topological fixed point data only records homology classes of fixed components regarded as embedded submanifolds of reduced spaces. 
	In general, we cannot rule out the possibility that there are many distinct
	fixed point data which have the same topological fixed point data. 
	
	Recall that any non-extremal part of a topological fixed point data in Table \ref{table_list} is on of the forms
	\[
		(M_c, [\omega_c], [Z_c^1], \cdots, [Z_c^{k_c}]), \quad c = -1, 0, 1. 
	\]
	If $c = \pm 1$, then all $Z_c^i$'s are isolated points. In this case, the topological fixed point data determines a fixed point data uniquely, since if 
	\[
		(M_c, \omega_c, p_1, \cdots, p_r) \quad \text{and} \quad (M_c, \omega_c' , q_1, \cdots, q_r), \quad \quad p_i, q_j : \text{points},\quad  [\omega_c] = [\omega_c'],
	\]
	then it follows from the symplectic rigidity of $M_c$ (obtained by Theorem \ref{theorem_uniqueness} and Theorem \ref{theorem_symplectomorphism_group})
	that there exists a symplectomorphism $\phi : (M_c, \omega_c) \rightarrow (M_c, \omega_c')$ sending $p_i$ to $q_i$ for $i=1,\cdots,r$. 
	(See \cite[Proposition 0.3]{ST}.)  
	
	For $c= 0$, it is not clear whether a topological fixed point data determines a fixed point data uniquely. On the other hand, the following theorems guarantee that 
	any symplectic embedding $Z_0 \hookrightarrow M_0$ in Table \ref{table_list}
	can be identified with an algebraic embedding. 
	(Note that every $Z_0^i$, except for the case {\bf (III-3.3)}, in Table \ref{table_list} is a sphere with self intersection greater than equal to $-1$. Moreover, in case of $M_0 \cong \p^2$, 
	the degree of $Z_0$ is less than equal to $3$. In particular, $Z_0 \cong T^2$ in {\bf (III-3.3)} is of degree $3$, i.e., cubic, in $\p^2$.)
	
	\begin{theorem}\cite[Theorem C]{ST}\label{theorem_ST}
		Any symplectic surface in $\p^2$ of degree $d \leq17$ is symplectically isotopic to an algebraic curve.
	\end{theorem}
	
	\begin{theorem}\cite[Proposition 3.2]{LW}\cite[Theorem 6.9]{Z}\label{theorem_Z}
		Any symplectic sphere $S$ with self-intersection $[S]\cdot[S] \geq 0$ in a symplectic four manifold $(M,\omega)$ is symplectically isotopic to an (algebraic) rational curve.
		Any two homologous spheres with self-intersection $-1$ are symplectically isotopic to each other.
	\end{theorem}
	
	Now we are ready to prove Theorem \ref{theorem_main}
	
	\begin{proof}[Proof of Theorem \ref{theorem_main}]
		Since every reduced space is symplecticaly rigid (by Theorem \ref{theorem_uniqueness} and Theorem \ref{theorem_symplectomorphism_group}), 
		it is enough to show that for each $(M,\omega,H)$, there exists a smooth Fano 3-fold admitting semifree holomorphic Hamiltonian $S^1$-action whose fixed point data 
		equals $\frak{F}(M,\omega,H)$. Then the proof immediately follows from Theorem \ref{theorem_Gonzalez}. 
		
		Recall that for any $(M,\omega,H)$ satisfying the conditions in Theorem \ref{theorem_main}, 
		there exists a smooth Fano 3-fold $(X, \omega_X, H_X)$ with a holomorphic Hamiltonian $S^1$-action
		whose topological fixed point data equals $\frak{F}_{\mathrm{top}} (M,\omega,H)$. (See examples in Section \ref{secCaseIDimZMax}, \ref{secCaseIIDimZMax2}, \ref{secCaseIIIDimZMax4}.) 
		By Theorem \ref{theorem_ST} and Theorem \ref{theorem_Z}, we may assume that every $(M_c, \omega_c, Z_c) \in \frak{F}(M,\omega,H)$ is an algebraic tuple, that is, 
		$Z_c$ is a complex (and hence K\"{a}hler) submanifold of $M_c$ for every critical value $c$ of the balanced moment map $H$. 
		Note that every reduced space of $(M,\omega,H)$ is either $\p^2$, $\p^1 \times \p^1$, or $\p^2 \# ~k \overline{\p^2}$ for $k \leq 3$, see Table \ref{table_list}. 
		In particular, any reduced space is birationally equivalent to $\p^2$, and therefore $H^1(M_c, \mcal{O}_{M_c}) = 0$. 
		Then Lemma \ref{lemma_isotopic} (stated below) implies that $(M_c, \omega_c, Z_c)$ is equivalent to the fixed point data $(X_c, (\omega_X)_c, (Z_X)_c)$ of $X$ at level $c$. 
		This finishes the proof.
	\end{proof}

	The following lemma can be obtained by composing Pereira's post \cite{MO} in MathOverflow (originally given in \cite[Remark 2]{To}) and \cite[Proposition 0.3]{ST}.

	\begin{lemma}\label{lemma_isotopic}
		Suppose that $X$ is a smooth projective surface with $H^1(X, \mcal{O}_X) = 0$. Let $H_1$ and $H_2$ be two smooth curves of $X$ representing the same homology class.
		Then $H_1$ is symplectically isotopic to $H_2$ with respect to the symplectic form $\omega_X = \omega_{\mathrm{FS}}|_X$ on $X$.
	\end{lemma}
	
	\begin{proof}
		Using 
		\begin{itemize}
			\item \cite[Proposition 0.3]{ST}, and 
			\item the fact that every smooth algebraic curve of $X$ is a symplectic submanifold of $(X,\omega_X)$, 
		\end{itemize}
		It is enough to find a family $\{H_t\}_{1 \leq t \leq 2}$ of smooth algebraic curves in $X$ which induces a constant homology class. 
		Note that the set of effective divisors in the class $[H_1]$ contains both $H_1$ and $H_2$ since the Chern class map 
		\[
			c : H^1(X, \mcal{O}_X^*) \rightarrow H^2(X; \Z)
		\]
		is injective by our assumption that $H^1(X, \mcal{O}_X) = 0$. 
		
		Now consider the complete linear system $|H_1|$ isomorphic to some projective space $\p^N$. 
		Since the set of smooth divisors in $|H_1|$ is Zariski open in $\p^N$, it is connected. This completes the proof.
	\end{proof}

\appendix

\section{Monotone symplectic four manifolds with semifree $S^1$-actions}
\label{secMonotoneSymplecticFourManifoldsWithSemifreeS1Actions}

In this section, we classify semifree Hamiltonian $S^1$-actions on compact monotone symplectic four manifolds up to $S^1$-equivariant symplectomorphism.

Let $(M,\omega)$ be a four dimensional closed monotone semifree Hamiltonian $S^1$-manifold such that $c_1(TM) = [\omega]$ with the balanced moment map. 
Then, similar to Lemma \ref{lemma_possible_critical_values}, we have the following. 

\begin{lemma}
	All possible critical values of $H$ are $\pm 2, \pm 1$, and $0$. Moreover, any connected component $Z$ of $M^{S^1}$ satisfies one of the followings : 
		\begin{table}[h]
			\begin{tabular}{|c|c|c|c|}
			\hline
			    $H(Z)$ & $\dim Z$ & $\mathrm{ind}(Z)$ & $\mathrm{Remark}$ \\ \hline 
			    $2$ &  $0$ & $4$ & $Z = Z_{\max} = \mathrm{point}$ \\ \hline
			    $1$ &  $2$ & $2$ & $Z = Z_{\max} \cong S^2$ \\ \hline
			    $0$ &  $0$ & $2$ & $Z = \mathrm{point}$ \\ \hline
			    $-1$ &  $2$ & $0$ & $Z = Z_{\min} \cong S^2$ \\ \hline
			    $-2$ &  $0$ & $0$ & $Z = Z_{\min} = \mathrm{point}$ \\ \hline
			\end{tabular}
			\vspace{0.2cm}
			\caption{\label{table_fixed_4dim} List of possible fixed components}
		\end{table}
\end{lemma}

\begin{proof}
	The dimension and the Morse-Bott index of each fixed component can be obtained from Corollary \ref{corollary_sum_weights_moment_value}. Also, if $H(Z) = \pm 1$, then 
	$Z$ should be an extremal fixed component since the weights of the $S^1$-action at $Z$ is $(\mp 1, 0)$. Furthermore, $Z \cong S^2$ 
	by \cite{Li1} since $M$ is diffeomorphic to some del Pezzo surface (which is simply connected). 
\end{proof}

So, we may divide into three cases (where other cases are recovered by taking reversed orientation of $M$) : 

\begin{itemize}
	\item $H(Z_{\min}) = -2$ and $H(Z_{\max}) = 2$.
	\item$H(Z_{\min}) = -2$ and $H(Z_{\max}) = 1$.
	\item $H(Z_{\min}) = -1$ and $H(Z_{\max}) = 1$.
\end{itemize}

\begin{theorem}\label{theorem_classification_4dim}
	Let $(M,\omega)$ be a four dimensional closed monotone semifree Hamiltonian $S^1$-manifold such that $c_1(TM) = [\omega]$ with the balanced moment map $H$. 
	Then all possible $(M,\omega)$, together with their fixed point data, are list as follows : 
	\begin{table}[H]
		\begin{tabular}{|c|c|c|c|c|c|c|c|}
			\hline
			                &  $M$                            & $Z_{-2}$ & $Z_{-1}$ &  $Z_0$          & $Z_1$ & $Z_2$ & $e(P_{\min}^+)$\\ \hline \hline
			{\bf (I-1)} &  $\p^1 \times \p^1$ & {\em pt} &                    & {\em 2 pts} &              & {\em pt} & $-u$ \\ \hline    
			
			{\bf (II-1)} &  $\p^2$ & {\em pt} &                    & 	 & $\p^1$             &  & $-u$ \\ \hline    
			{\bf (II-2)} &  $\p^2 \# ~\overline{\p^2}$ & {\em pt} &                    & {\em pt}	 & $\p^1$             & & $-u$\\ \hline    
			{\bf (II-3)} &  $\p^2 \# ~2 ~\overline{\p^2}$ & {\em pt} &                    & {\em 2 pts}	 & $\p^1$        &     & $-u$ \\ \hline    
			
			{\bf (III-1)} &  $\p^1 \times \p^1$                      &  & $\p^1$ &                    & $\p^1$  &  & $0$\\ \hline    
			{\bf (III-2)} &  $\p^2 \# ~\overline{\p^2}$                &  & $\p^1$ &                    & $\p^1$  & & $-u$  \\ \hline    
			{\bf (III-3)} &  $\p^2 \# ~2~\overline{\p^2}$ &  & $\p^1$ &                  {\em pt} & $\p^1$ &  & $0$ \\ \hline    
			{\bf (III-4)} &  $\p^2 \# ~3~\overline{\p^2}$ &  & $\p^1$ &                  {\em 2 pts} & $\p^1$ &  & $-u$\\ \hline    			
		\end{tabular}
		\vs{0.3cm}
		\caption {List of topological fixed point data} \label{table_list_4dim} 
	\end{table}
	\noindent
	Moreover, the corresponding $S^1$-actions for each cases are described in Figure \ref{figure_4dim}.
\end{theorem}

\begin{figure}[h]
	\scalebox{0.9}{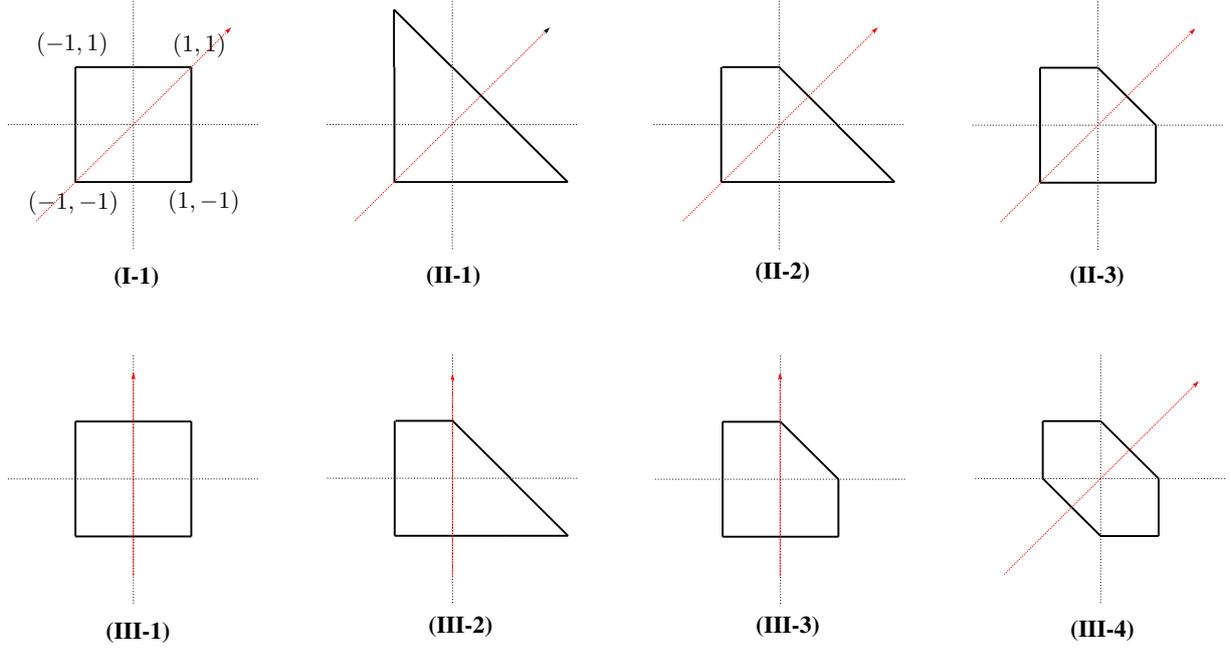}
	\vs{0.3cm}
	\caption{\label{figure_4dim} Classification in dimension four}
\end{figure}

\begin{proof}
	First note that $(M_0, [\omega_0]) = (\p^1, 2u)$ by Proposition \ref{proposition_monotonicity_preserved_under_reduction}.
	Let $k = |Z_0|$ be the number of interior fixed points. \vs{0.1cm}
	
	\noindent {\bf Case I : $H(Z_{\min}) = -2$ and $H(Z_{\max}) = 2$.} In this case, we have $Z_{\min}$ = $Z_{\max}$ = pt and 
	the Euler classes of the principal bundles 
	\[
		H^{-1}(-2 + \epsilon) \rightarrow M_{-2 + \epsilon} \quad \text{and} \quad H^{-1}(2 - \epsilon) \rightarrow M_{2 - \epsilon}
	\]
	are $-u$ and $u$, respectively. (See Section \ref{secCaseIDimZMax}.) By Lemma \ref{lemma_Euler_class}, we have 
	\[
		e(P_{2}^-) = e(P_{- 2}^+) + \mathrm{PD}(Z_0) \quad \Leftrightarrow \quad u = -u + ku \quad \Leftrightarrow \quad k = 2.
	\]
	
	\noindent {\bf Case II : $H(Z_{\min}) = -2$ and $H(Z_{\max}) = 1$.} In this case, we have $Z_{\min}$ = pt and $Z_{\max}$ = $\p^1$. 
	Note that if $k \geq 3$, then 
	\[
		[\omega_1] = 2u - (k-1)u = (3 - k)u 
	\]
	so that $\int_{Z_{\max}} \omega_1 \leq 0$ which is impossible. Thus we get $k \leq 2$.
	\vs{0.2cm}
	
	\noindent {\bf Case III : $H(Z_{\min}) = -1$ and $H(Z_{\max}) = 1$.} In this case, we have $Z_{\min}$ = $Z_{\max}$ = $\p^1$. Set 
	\[
		[\omega_{-1}] = au, \quad  e(P_{-1}^+) = bu, \quad a > 0 , ~b \in \Z.
	\]
	Then, 
	\[
		[\omega_0] = 2u = (a-b)u, \quad [\omega_1] = a - b - (b+k)u = (a - 2b - k)u > 0,
	\]
	where all possible tuples $(a,b,k)$ are 
	\begin{itemize}
		\item $\{(1, -1, k) ~|~ 0 \leq k \leq 2\},$
		\item $\{(2, 0,  k) ~|~ 0 \leq k \leq 1\},$
		\item $\{(3, 1,0 )\}$. 
	\end{itemize}	
	Note that $(2, 0, 1)$ and $(1, -1, 1)$ are essentially same where one can be obtained from another by taking an opposite orientation $M$. 
	Similarly $(3,1,0)$ and $(1, -1, 0)$ induce the same fixed point data. 
	Therefore, 
	there are four possibilities 
	\[
		\underbrace{(2,0,0)}_{\text{\bf (III-1)}}, \quad \underbrace{(1,-1,0)}_{\text{\bf (III-2)}}, \quad \underbrace{(2,0,1)}_{\text{\bf (III-3)}}, \quad \underbrace{(1,-1,2)}_{\text{\bf (III-4)}}. 
	\]
	
	To show that $(M,\omega)$ is one of those in Table \ref{table_list_4dim} (or in Figure \ref{figure_4dim}), note that every reduced space (at any regular level) is diffeomorphic to $S^2$, 
	and hence it is symplectically rigid by the classical Moser Lemma. Also, the topological fixed point data determines the fixed point data uniquely since 
	\[
		(S^2, \omega_0, p_1, \cdots, p_k) \underbrace{\cong}_{\text{symplectic isotopic}} (S^2, \omega_0, q_1, \cdots, q_k) 
	\]
	for any distinct $k$ points $\{p_1, \cdots, p_k\}$ and $\{q_1, \cdots, q_k\}$ on $S^2$ by \cite[Proposition 0.3]{ST}. Furthermore, each toric Fano manifold with the specific choice of the 
	$S^1$-action described in Figure \ref{figure_4dim} has the same fixed point data as the corresponding one in Table \ref{table_list_4dim}. So, our theorem follows from
	Theorem \ref{theorem_Gonzalez}. (See also Remark \ref{remark_Gonzalez_5}.)
\end{proof}

\section{Symplectic capacities of smooth Fano 3-folds}
\label{secSymplecticCapacitiesOfSmoothFano3Folds}

In this section, we compute two kinds of symplectic capacities, namely the {\em Gromov width} and the {\em Hofer-Zehnder capacity}, of symplectic manifolds given in Table \ref{table_list}. 

Recall that the {\em Gromov width} of a closed $2n$-dimensional symplectic manifold $(M,\omega)$ is defined as 
\[
	w_G(M,\omega) := \sup ~\{ \pi a^2 ~|~ \text{$B(a)$ is symplectically embedded in $(M,\omega)$} \}
\]
where $B(a) \subset \C^n$ is a $2n$-dimensional open ball of radius $a$ with a standard symplectic structure on $\C^n$. The Hofer-Zehnder capacity of $(M,\omega)$ is defined to be 
\[
	c_{HZ}(M,\omega) := \sup ~\{ K_{\max} - K_{\min} ~|~ K : \text{admissible} \}.
\]
Here, a smooth function $K : M \rightarrow \R$ is said to be {\em admissible} if there are open subsets $U, V$ in $M$ such that 
\begin{itemize}
	\item $K|_U = K_{\max}$ and $K|_V = K_{\min}$,
	\item there is no non-constant periodic orbit of $K$ whose period is less than one.
\end{itemize}
(We refer to \cite{MS} for more details.) In \cite{HS}, Hwang and Suh proved the following. 

\begin{theorem}\cite[Theorem 1.1]{HS}\label{theorem_HS}
	Let $(M,\omega)$ be a closed monotone symplectic manifold with a semifree Hamiltonian circle action such that $c_1(TM) = [\omega]$ and let $H$ be a moment map.
	If the minimal fixed component is an isolated point, then the Gromov width and the Hofer-Zehnder capacity of $(M,\omega)$ are respectively given by $H_{\mathrm{smin}} - H_{\min}$ 
	and $H_{\max} - H_{\min}$. Here, $M_{\mathrm{smix}}$ is the second minimal critical value of the moment map $H$.
\end{theorem}

Using Theorem \ref{theorem_HS}, we obtain the followings. 
\begin{proposition}
	For each smooth Fano 3-fold admitting semifree Hamiltonian circle action, the Gromov width and the Hofer-Zehnder capacity can be compute as follows : 
	\begin{table}[H]
		\begin{tabular}{|c|c|c|c|c|c|}
			\hline
			    & $H_{\max}$ & $H_{\mathrm{smin}}$ & $H_{\min}$ & $w_G$ & $c_{HZ}$ \\ \hline \hline
			    {\bf (I-1)} & 3 & -1 &  -3& 2 & 6 \\ \hline    
			   {\bf (I-2)} &  3 & -1 &  -3& 2 & 6 \\ \hline    
			   
			   {\bf (I-3)} & 3 & -1 &  -3& 2 & 6 \\ \hline    
			    
			   {\bf (II-3.1)} & 2 & -1 &  -3& 2 & 5 \\ \hline    
			   {\bf (II-3.2)} & 2 & -1 &  -3& 2 & 5 \\ \hline    
			   {\bf (II-3.3)} & 2 & -1 &  -3& 2 & 5 \\ \hline    
			    
			   {\bf (II-4.1)} & 2 & -1 &  -3& 2 & 5 \\ \hline    
			   {\bf (II-4.2)} & 2 & -1 &  -3& 2 & 5 \\ \hline    
		    		    
			   {\bf (III-1)} & 1 & 1 &  -3& 4 & 4 \\ \hline    
			   {\bf (III-2)} & 1 & -1 &  -3& 2 & 4 \\ \hline    
			   {\bf (III-3.1)} & 1 & 0 &  -3& 3 & 4 \\ \hline    
			   {\bf (III-3.2)} & 1 & 0 &  -3& 3 & 4 \\ \hline    
			   {\bf (III-3.3)} & 1 & 0 &  -3& 3 & 4 \\ \hline    
			   
			   {\bf (III-4.1)} &  1 & -1 &  -3& 2 & 4 \\ \hline
			   {\bf (III-4.2)} &  1 & -1 &  -3& 2 & 4 \\ \hline
			   {\bf (III-4.3)} &  1 & -1 &  -3& 2 & 4 \\ \hline
			   {\bf (III-4.4)} &  1 & -1 &  -3& 2 & 4 \\ \hline
			   {\bf (III-4.5)} &  1 & -1 &  -3& 2 & 4 \\ \hline
		\end{tabular}
		\vs{0.3cm}
		\caption {Gromov width and Hofer-Zehnder capacity of smooth Fano 3-folds} \label{table_list_capacities} 
	\end{table}
	
\end{proposition}

\bibliographystyle{annotation}

\end{document}